\documentclass[11pt]{amsart}

\usepackage{amsfonts,latexsym,amsthm,amssymb,amsmath,amscd,euscript}
\usepackage[alphabetic,msc-links,bibtex-style,nobysame]{amsrefs}
\makeatletter
\def\bibtex@style{amsrn}
\makeatother
\BibSpec{collection.article}{%
    +{}  {\PrintAuthors}                {author}
    +{.} { }                            {title}
    +{.} { }                            {part}
    +{:} { }                            {subtitle}
    +{.} { \PrintContributions}         {contribution}
    +{.} { \PrintConference}            {conference}
    +{.} { \PrintBook}                  {book}
    +{.} { In: }                         {booktitle}
    +{,} { pages~}                      {pages}
    +{.} { }                            {publisher}
    +{,} { }                            {address}
    +{,} { \PrintDateB}                 {date}
    +{.} { \PrintTranslation}           {translation}
    +{.} { Reprinted in \PrintReprint}  {reprint}
    +{.} { }                            {note}
    +{.} {}                             {transition}
}

\usepackage{stackengine}
\usepackage{framed}
\usepackage{xfrac}
\usepackage[makeroom]{cancel}
\usepackage{faktor}
\usepackage{braket}

\usepackage{pgf}
\usepackage{tikz}
\usepackage{tikz-cd}
\usepackage{pgfplots}
\pgfplotsset{compat=1.17}
\usepgfplotslibrary{fillbetween}

\usetikzlibrary{arrows,matrix,calc,positioning,decorations.pathreplacing,patterns,knots}

\usepackage{mathrsfs}

\definecolor{wrwrwr}{rgb}{0.3803921568627451,0.3803921568627451,0.3803921568627451}
\definecolor{rvwvcq}{rgb}{0.08235294117647059,0.396078431372549,0.7529411764705882}
\definecolor{mblue}{rgb}{0.2,0.3,0.8}
\definecolor{morange}{rgb}{1,0.5,0}
\definecolor{mgreen}{rgb}{0.1,0.4,0.2}
\definecolor{mred}{rgb}{0.5,0,0}
\definecolor{ForestGreen}{RGB}{34,139,34}

\usepackage{float}
\usepackage{lmodern}
\usepackage{supertabular}
\usepackage{verbatim}
\usepackage{enumerate}
\usepackage{stmaryrd}
\usepackage{bbm}
\usepackage{mathtools}
\usepackage[nopatch=footnote]{microtype}
\usepackage{setspace}
\usepackage{url}

\usepackage[hmargin=1in,vmargin=1in]{geometry}
\usepackage{graphicx}
\usepackage{subcaption}

\usepackage{hyperref}
\hypersetup{
    colorlinks=true,
    citecolor=ForestGreen,
    linkcolor=blue,
    urlcolor=black,
    linkbordercolor={1 0 0}
}

\newcommand{\la}{\langle}
\newcommand{\rg}{\rangle}

\newtheorem{theorem}{{Theorem}}[section]
\newtheorem*{theorem*}{Theorem}
\newtheorem{lemma}[theorem]{Lemma}
\newtheorem{proposition}[theorem]{Proposition}

\newtheorem{definition}[theorem]{Definition}
\newtheorem{corollary}[theorem]{Corollary}
\newtheorem*{corollary*}{Corollary}

\newcommand{\ve}{\varepsilon}
\newcommand{\mr}[1]{{\rm #1}}

\newcommand{\cE}{\mathcal{E}}
\newcommand{\cF}{\mathcal{F}}
\newcommand{\cG}{\mathcal{G}}

\newcommand{\cQ}{\mathcal{Q}}

\newcommand{\bR}{\mathbb{R}}
\newcommand{\bS}{\mathbb{S}}

\newcommand{\bZ}{\mathbb{Z}}

\newcommand{\fo}{\mathfrak{o}}

\newcommand{\fs}{\mathfrak{s}}

\newcommand{\nc}{\newcommand}

\nc{\on}{\operatorname}
\nc{\p}{\partial}
\nc{\ol}{\overline}
\nc{\ul}{\underline}
\nc{\pa}{\partial}

\nc{\pb}{\partial_b}
\nc{\pc}{\partial_c}
\nc{\pd}{\partial_d}
\nc{\pe}{\partial_e}
\nc{\pf}{\partial_f}
\nc{\pg}{\partial_g}
\nc{\ph}{\partial_h}
\nc{\pari}{\partial_i}
\nc{\pj}{\partial_j}
\nc{\pk}{\partial_k}
\nc{\pl}{\partial_l}
\nc{\pell}{\partial_\ell}
\nc{\parm}{\partial_m}
\nc{\pn}{\partial_n}
\nc{\po}{\partial_o}
\nc{\pp}{\partial_p}
\nc{\pq}{\partial_q}
\nc{\pr}{\partial_r}
\nc{\ps}{\partial_s}
\nc{\pt}{\partial_t}
\nc{\pu}{\partial_u}
\nc{\pv}{\partial_v}
\nc{\pw}{\partial_w}
\nc{\px}{\partial_x}
\nc{\py}{\partial_x}
\nc{\pz}{\partial_z}

\nc{\Spec}{\textup{Spec}}

\nc{\sn}{\mr{sn}}
\nc{\cn}{\mr{cn}}
\nc{\dn}{\mr{dn}}

\makeatletter

\def\@tocline#1#2#3#4#5#6#7{%
    \relax
    \ifnum #1>\c@tocdepth
    \else
        \par
        \addpenalty\@secpenalty
        \addvspace{#2}%
        \begingroup
            \hyphenpenalty\@M
            \@ifempty{#4}{%
                \@tempdima
                \csname r@tocindent\number#1\endcsname
                \relax
            }{%
                \@tempdima#4\relax
            }%
            \parindent\z@
            \leftskip#3\relax
            \advance\leftskip\@tempdima\relax
            \rightskip\@pnumwidth plus4em
            \parfillskip-\@pnumwidth
            #5%
            \leavevmode
            \hskip-\@tempdima
            #6%
            \nobreak
            \leaders\hbox to .65em{\hss.\hss}\hfill
            \nobreak
            \hbox to\@pnumwidth{\@tocpagenum{#7}}%
            \par
        \endgroup
    \fi
}

\renewcommand{\l@section}
    {\@tocline{1}{2pt}{0pt}{}{\bfseries}}

\renewcommand{\l@subsection}
    {\@tocline{2}{0pt}{1.5em}{3.2em}{}}

\renewcommand{\l@subsubsection}
    {\@tocline{3}{0pt}{4.7em}{3.8em}{}}

\makeatother

\title[Ricci flow on manifolds with positive isotropic curvature in all dimensions]{Ricci flow on manifolds with positive isotropic curvature in all dimensions}
\date{\today}

\author[Simon Brendle]{Simon Brendle}
\address{Department of Mathematics, Columbia University
\newline {\href{mailto:simon.brendle@columbia.edu}{simon.brendle@columbia.edu}}}
\author[Raphael Tsiamis]{Raphael Tsiamis}
\address{Department of Mathematics, Columbia University
\newline {\href{mailto:r.tsiamis@columbia.edu}{r.tsiamis@columbia.edu}}}

\begin{document}

\begin{abstract}
We extend the classification of manifolds with positive isotropic curvature to all dimensions $n \geq 5$. This result also completes the classification of uniformly positive isotropic curvature on complete non-compact $\kappa$-noncollapsed ancient solutions and gradient shrinking Ricci solitons.

The proof involves a construction of pinching cones in all dimensions, refining and extending to all dimensions the results of the first author. We have used \textsc{Mathematica} to find viable parameters for the cone families. The final presentation is self-contained and can be verified manually. No stage of this work involved \textsc{LLMs}.
\end{abstract}

\vspace*{0.2in}
\maketitle

%\vspace{-0.2in}

\tableofcontents

\section{Introduction}

We complete in all dimensions the classification of manifolds with positive isotropic
curvature established in~\cite{brendle-PIC}.
Positive isotropic curvature is a natural curvature condition for the Riemann curvature tensor in dimensions four and above, introduced by Micallef and Moore in their study of minimal two-spheres in Riemannian manifolds. 
Their work showed, in particular, that a compact simply connected manifold with positive isotropic curvature is homeomorphic to a sphere~\cite{micallef-moore}.
Variants of the PIC condition were instrumental in the proof of the Differentiable Sphere Theorem~\cite{brendle-schoen}.

Our main result is the following classification, valid in every dimension $n \geq 4$.

\begin{theorem}\label{thm:brendle-pinching}
Let $(M^n,g)$ be a compact Riemannian manifold with positive isotropic curvature.
If $M$ contains no nontrivial incompressible $(n-1)$-dimensional space forms, then $M$ is diffeomorphic to a finite connected sum of quotients of $\bS^n$ and $\bS^{n-1} \times \bR$ by standard isometries.
\end{theorem}

The topology of manifolds with positive isotropic curvature has also been studied from the viewpoint of their fundamental groups. 
Fraser proved that the fundamental group of a compact PIC manifold of dimension at least five contains no subgroup isomorphic to $\bZ \oplus \bZ$~\cite{fraser}.
These results form part of a broader picture suggested by conjectures of Gromov and Schoen: the fundamental group of a closed PIC manifold should be virtually free, and, more strongly, a finite cover should be diffeomorphic to a sphere or to a connected sum of copies of $\bS^n$ and $\bS^{n-1} \times \bS^1$~\cites{gromov}.

The Ricci flow approach to this problem originates in the work of Hamilton, who developed a surgery program for four-dimensional manifolds with positive isotropic curvature, exploiting the special structure of the curvature operator to obtain a pinching estimate toward positive curvature operator~\cite{hamilton-four-manifolds}.
The four-dimensional problem was subsequently completed to a full classification by Chen, Tang, and Zhu~\cites{chen-zhu-jdg , chen-tang-zhu}.
In dimension three, Hamilton's program was transformed by Perelman's noncollapsing theorem, canonical neighborhood theory, and Ricci flow with surgery~\cites{hamilton,ivey,perelman-entropy,perelman-surgery}.

A central problem in Ricci flow for higher-dimensional manifolds with positive isotropic curvature is to establish suitable pinching properties of the Riemann curvature tensor.
The first author established a pinching estimate which forces high-curvature regions toward the PIC2 cone, and may be viewed as a higher-dimensional counterpart of the Hamilton-Ivey phenomenon~\cite{brendle-PIC}.
Once this estimate is available, the analysis of ancient solutions, the Canonical Neighborhood Theorem, and the surgery construction follow the higher-dimensional version of Perelman's program.
The main restriction in the original argument of~\cite{brendle-PIC} was dimensional, due to the construction of suitable cones of operators in dimension $n \geq 12$.
Later, Chen~\cite{pic-dimension-nine} showed that this construction extends to dimensions $n \geq 9$.
Very recently, Chow and Wang established the sharp lower bound $(n-3)$ for the index of harmonic $2$-spheres in $n$-manifolds with PIC, thereby obtaining an alternative proof of Theorem~\ref{thm:brendle-pinching} in dimensions five and six~\cite{chow-wang-PIC}.
Theorem~\ref{thm:brendle-pinching} completes the classification in every dimension, using the techniques of pinching cones for positive isotropic curvature of~\cite{brendle-PIC}.

The pinching estimate used in Theorem~\ref{thm:brendle-pinching} has consequences beyond the compact topological classification. The work of the first author and Naff classifies noncompact ancient $\kappa$-solutions which are uniformly PIC and weakly PIC2: they are shrinking round cylinders, quotients thereof, or the Bryant soliton~\cite{brendle-naff}. Combining that theorem with the pinching estimate obtained here yields the following consequences in every dimension $n\geq4$.
\begin{corollary}\label{cor:brendle-naff}
The following statements hold in every dimension $n\geq 4$.
\begin{enumerate}[(i)]
    \item Every complete non-compact $\kappa$-noncollapsed ancient solution of Ricci flow with uniformly positive isotropic curvature is weakly PIC2 and has bounded curvature. Consequently, it is a shrinking round cylinder, a quotient of a shrinking round cylinder, or the Bryant soliton.
    \item Every non-flat complete gradient shrinking Ricci soliton with uniformly positive isotropic curvature is a quotient of the round sphere $\bS^n$ or the round cylinder $\bS^{n-1} \times \bR$.
\end{enumerate}
\end{corollary}
The second statement extends the works of~\cites{li-ni-wang , naff} to all dimensions; see~\cites{cho-li,chen-shrinking} for related work.
The compact ancient theory around this problem has also been developed substantially by work of the first author with Daskalopoulos, Naff, and \v{S}e\v{s}um, who classified compact ancient $\kappa$-solutions on $\bS^n$, showing that they are either shrinking round spheres or the higher-dimensional Type II ancient solution constructed by Perelman~\cite{brendle-daskalopoulos-naff-sesum}.

We note that Huang has proposed an extension of Ricci flow with surgery for compact PIC orbifolds with isolated singularities, intended to remove the hypothesis of Theorem~\ref{thm:brendle-pinching} regarding incompressible space forms~\cite{ huang-classification-2023}.

\subsection{Strategy of the proof}

We follow the construction of pinching cones in~\cite{brendle-PIC}.
Our main technical contribution is the following result.

\begin{theorem}\label{thm:technical-construction}
For $n \geq 5$, let $S^2_B(\fs \fo(n))$ denote the space of self-adjoint endomorphisms of $\fs \fo(n)$ satisfying the first Bianchi identity.
There exist positive constants $B, \theta_*$ and a continuous family of closed, convex, $O(n)$-invariant cones $\{ \hat{C}(b) \}_{ b \in (0,B) } \subset S^2_B( \fs \fo(n))$ such that each $\hat{C}(b)$ is transversally invariant under Hamilton's ODE $\frac{d}{dt} R= \cQ(R)$, and the following properties hold.
\begin{enumerate}[(i)]
    \item For every $\theta \in (0, \theta_*]$, there exists a $\beta_0 = \beta_0(\theta) \in (0,B)$ such that, for every $N>0$,
    \begin{align*}
        &\{ R : R - \theta \, \textup{scal}(R) \, \textup{id} \owedge \textup{id} \in C_{\textup{PIC}} \} \cap \{ R : \textup{Ric}(R)_{11} + \textup{Ric}(R)_{22} - \theta \, \textup{scal}(R) + N \geq 0 \} \\
        & \quad\subset \{R : R + N \, \textup{id} \owedge \textup{id} \in \hat{C}(\beta_0)\}.
    \end{align*}
    The second condition is understood for every orthonormal two-frame $\{e_1, e_2 \}$.
    \item For any sequence $\beta_j \nearrow B$, there exists a sequence of positive real numbers $\ve_j \searrow 0$ such that
    \[
    \hat{C}(\beta_j) \subset \{ R : R + \ve_j \,  \textup{scal}(R) \, \textup{id} \owedge \textup{id} \in C_{\textup{PIC}1} \} \qquad \text{for every } \; j.
    \]
\end{enumerate}
\end{theorem}

Following the approach of~\cite{brendle-PIC}, we will build the family of cones $\{ \hat{C}(b) \}$ pinching toward $C_{\textup{PIC}1}$ by connecting two families $\{ C_1(b) \}_{0 < b \leq b_1}$ and $\{ C_2(b) \}_{0 < b \leq b_2}$.
The family $C_1$ is an inward deformation of the cone $\{ S \in C_{\textup{PIC}} : \textup{Ric}(S)_{11} + \textup{Ric}(S)_{22} \geq 0 \}$, while the family $C_2$ is an outward deformation of the cone $C_1(b_1) \cap C_{\textup{PIC}1}$.
The arguments follow the architecture of~\cite{brendle-PIC}, but several estimates must be sharpened substantially in the construction of both families. 

The first cone family has the greatest potential for refining the estimates, which we carry out in Section~\ref{section:first-family}.
An important observation is that, while various bounds obtained in~\cite{brendle-PIC}*{\S 3} are separately sharp, their equality cases force the spectrum of $\textup{Ric}(S)$ to be concentrated in different regions.
Thus, we obtain sharper estimates by retaining additional variables and optimizing them jointly only at the last step; this produces quadratic expressions over the domain of PIC curvature components, notably the function $\cF(\xi,s,\tau)$ studied in Lemma~\ref{lemma:important-properties}.

While the above strategy extends the viable dimension to $n \geq 7$, obtaining the final dimensions $n \in \{ 5, 6\}$ requires a new technique for the second family.
Compared to~\cite{brendle-PIC}*{\S 4}, in Definition~\ref{def:a=b} we introduce the auxiliary term $B_n b \, \textup{scal}(S)$ to the second family, offering additional flexibility.
Because this family is simpler, defined only through the operator~\eqref{eqn:Z-tensor} that controls the pinching toward PIC1, the only available improvement comes from estimating the derivative $\frac{d}{dt}\big\rvert_{t=0} \tilde{Z}_{\lambda}(S)$ optimally on boundary faces with $\tilde{Z}_{\lambda}(S) = 0$.
As $\lambda \uparrow 1$, this term becomes unfavorable, and the conditions derived in~\cite{brendle-PIC}*{\S 4} cannot be simultaneously satisfied in low dimensions.
Instead, in Proposition~\ref{prop:why-the-long-face}, we prove that the faces $\tilde{Z}_{\lambda}(S) = 0$ are not reached when $\lambda$ is quantitatively close to $1$.
Moreover, we select the parameters of the second family to allocate the PIC inequalities more efficiently between the $\lambda=0$ and $\lambda=1$ constraints for joining the two cones.
This choice is explained in Section~\ref{section:computational-tools}, and the refinement for joining the two families is established in Section~\ref{section:joining-condition}.

This manuscript is organized as follows.
In Section~\ref{sec:preliminaries}, we recall the key expressions involved in Hamilton's ODE.
We then introduce the first and second families in Sections~\ref{section:first-family} and~\ref{section:second-family}.
In Lemmas~\ref{lemma:important-properties} and \ref{lemma:properties-II}, we collect some elementary inequalities satisfied by the parameters of these families, which are crucial for establishing their transversal invariance under Hamilton's ODE.
Finally, in Section~\ref{section:joining-condition}, we join the cones $C_1(b_1) = C_2(b_2)$, enabling the construction of the pinching family $\{ \hat{C}(b) \}_{b \in (0,B)}$ of Theorem~\ref{thm:technical-construction}.
In low dimensions, connecting the cones is a delicate step that requires combining multiple inequalities derived from PIC, weighted with appropriate coefficients.
In dimension five, which is the most challenging, the final improvement is obtained by invoking the quadratic pinching condition derived from the first family; this step is carried out as \hyperref[step:pic-n=5]{Case I} of Proposition~\ref{prop:joining-lower-bound}.

\subsection{Computational tools}\label{section:computational-tools}

The manuscript is entirely self-contained and does not require \textsc{Mathematica}.
The construction of the families $C_1(b)$ and $C_2(b)$ involved significant algebraic manipulations and numerical optimization, which were performed and verified using \textsc{Mathematica}.

We first designed families $C_1(b)$ and $C_2(b)$ depending on parameters $a,\gamma , \omega : (0,b_1] \to (0,\infty)$, for the first family, and $a,\Xi,r : (0,b_2] \to (0,\infty)$, for the second one, as in the construction of~\cite{brendle-PIC}.
The requirements of transversal invariance under the modified ODE $\frac{d}{dt} S = \cQ(S) + D_{a,b}(S)$ and of the endpoint condition $C_1(b_1) = C_2(b_2)$ can be reduced, through the computations performed in Sections~\ref{section:first-family} and~\ref{section:second-family}, to a finite list of conditions on these auxiliary functions.
After abstractly deriving the various constraints for these parameters, we imposed the ansatz
\begin{align*}
        C_1(b) &: \qquad a(b) = b + D_n b^2, \qquad & \gamma(b) &= C b , \qquad & \omega(b) &= \Omega_n b^{-1}, \\
        C_2(b) &: \qquad a(b) = b + A_n b^2, \qquad & \Xi(b) &= \sqrt{U b}, \qquad & r(b) &= B_n \,  b \Xi(b).
\end{align*}
As $n \to \infty$, the computation of~\cite{brendle-PIC}*{\S 3 and \S 4} indicates that the parameters $C,Q,B_n$ can be chosen independently of the dimension, while $A_n, D_n = O(n)$, $\Omega_n  = O(n^{-1})$.
Moreover, there should exist constants $\omega_1, B_0$, independent of the dimension, for which the endpoints $b_1, b_2$ satisfy
\[
\Omega_n = \omega_1 b(n), \qquad b_1(n) = O(n^{-1}), \qquad b_2(n)^{-1} = B_0 ( n- n_0)^2.
\]
A numerical search using \textsc{Mathematica} produces uniform first family parameters satisfying the conditions of Section~\ref{section:first-family} for every $n \geq 5$, upon taking $(D_n, C, \Omega_n, \omega_1) = ( \frac{n-1}{2}, \frac{2}{5}, \frac{1}{n-2}, 4)$.
In Section~\ref{section:second-family}, we obtain viable second family parameters for every $n \geq 7$.
The search for compatible parameters between the two families is greatly simplified upon further imposing the endpoint requirement
\[
\Xi(b_2) ( 1 + (n-2) p (b_2)) - p(b_2) = \tfrac{1}{n-2},
\]
which enables us to apply Proposition~\ref{prop:joining-lower-bound}.
Thus, we arrive at the admissible choices
\[
A_n = \tfrac{n-2}{2}, \qquad U = 2, \qquad B_n=4, \qquad b_2(n)^{-1} = 2(n-2)^2.
\]
The most challenging dimensions $n \in \{ 5, 6\}$ require a more precise modification, described in~\eqref{eqn:b2(n)-parameters}.

\subsection{Acknowledgments}
We thank Yipeng Wang for helpful conversations.
The first author was supported by the National Science Foundation under grant DMS2403981 and by the Simons Foundation.
The second author was supported by a Simons Dissertation Fellowship, an A.G.~Leventis Foundation Scholarship, and an Onassis Foundation Scholarship.

\section{Preliminaries}\label{sec:preliminaries}

First, we recall some preliminary properties of algebraic curvature operators after~\cites{brendle-schoen , bohm-wilking}.
We let $S^2 ( \fs \fo(n))$ denote the space of self-adjoint endomorphisms of $\fs \fo(n)$, and let $S^2_B(\fs \fo(n)) \subset S^2( \fs \fo(n))$ denote the subspace of tensors satisfying the first Bianchi identity.
We refer to elements of $S^2_B(\fs \fo(n))$ as \textbf{algebraic curvature operators}.

For an endomorphism $T \in S^2(\fs \fo(n))$, which need not satisfy the first Bianchi identity, we define the bilinear form $\textup{Ric}(T)$ and the quantity $\textup{scal}(T)$ by contraction.
We say that a $T \in S^2( \fs \fo(n))$ is \textbf{weakly PIC (WPIC)}, denoted $T \in C_{\textup{PIC}}$, if
\begin{equation}\label{eqn:I(T;F)}\tag{WPIC}
    Z_1(T;F) := T_{1313} + T_{1414} + T_{2323} + T_{2424} + 2 \, T_{1342} + 2 \, T_{1423} \geq 0
\end{equation}
for every orthonormal four-frame $F = \{ e_1, e_2, e_3, e_4 \}$.
If $T \in S^2_B(\fs \fo(n))$, then the first Bianchi identity gives $T_{1342} + T_{1423} = - T_{1234}$, so this condition is equivalent to the usual WPIC inequality
\[
T_{1313} + T_{1414} + T_{2323} + T_{2424} - 2 \,\ T_{1234} \geq 0.
\]
Similarly, we say that $T \in S^2 (\fs \fo(n))$ is \textbf{weakly PIC1 (WPIC1)}, denoted $T \in C_{\textup{PIC}1}$, if
\begin{equation}\label{eqn:wpic-1-condition}
Z_{\lambda}(T;F) := T_{1313} + \lambda^2 T_{1414} + T_{2323} + \lambda^2 T_{2424} + 2 \lambda (T_{1342} + T_{1423}) \geq 0 \tag{WPIC1}
\end{equation}
holds for every orthonormal four-frame $F = \{ e_1, e_2, e_3, e_4 \}$ and every $\lambda \in [0,1]$.
When $T \in S^2_B(\fs \fo(n))$, the first Bianchi identity makes the WPIC1 condition equivalent to
\[
T_{1313}+\lambda^2T_{1414}
+T_{2323}+\lambda^2T_{2424}
-2\lambda T_{1234}\ge 0
\]
for every orthonormal four-frame and every $\lambda \in [0,1]$.
The cones consisting of operators $T$ satisfying these conditions are respectively denoted by $C_{\mathrm{PIC}}$ and $C_{\mathrm{PIC}1}$.
If an operator $T$ satisfies the condition~\eqref{eqn:I(T;F)} (respectively, ~\eqref{eqn:wpic-1-condition}) strictly, it is called \textbf{PIC} (respectively, \textbf{PIC1}).
In what follows, whenever the choice of a four-frame is clear, we suppress the dependence of $Z_1(T;F)$ and $Z_{\lambda}(T;F)$ on $F$ and write simply $Z_1(T)$ and $Z_{\lambda}(T)$.

We denote by $\owedge$ the Kulkarni-Nomizu product of curvature operators, defined componentwise as
\[
( A \owedge B)_{ijk\ell} = A_{ik} B_{j \ell} - A_{i \ell} B_{jk} - A_{jk} B_{i \ell} + A_{j \ell} B_{ik}.
\]
Following~\cites{bohm-wilking , brendle-duke}, we have
\[
\textup{Ric}(h \owedge h) = 2 \, ( \textup{tr} \, h) h - 2h^2, \quad \textup{Ric} ( h \owedge \textup{id}) = (n-2) h + ( \textup{tr} \, h) \, \textup{id}, \quad \textup{Ric} ( \textup{id} \owedge \textup{id}) = 2(n-1) \, \textup{id}.
\]
Given an algebraic curvature operator $S \in S^2_B (\fs \fo(n))$, we define
\begin{equation}\label{eqn:ell-a,b}
    \ell_{a,b}(S) := S + b \, \textup{Ric}(S) \owedge \textup{id} + \frac{1}{n} (a-b) \, \textup{scal}(S) \, \textup{id} \owedge \textup{id}.
\end{equation}
We define the operator $\cQ(S) := S^2 + S^{\#}$ on $S^2(\fs \fo(n))$, where $\#$ is Hamilton's bilinear operator on $S^2(\fs \fo(n))$.
Notably, $S^2_B(\fs\fo(n))$ is invariant under $\cQ$.
We also denote $D_{a,b}(S) = \ell_{a,b}^{-1} \circ \cQ \circ \ell_{a,b} (S) - \cQ(S)$, so the computation of~\cite{bohm-wilking} gives
\begin{equation}\label{eqn:D-a-b}
    \begin{split}
    D_{a,b}(S) &= 2( b + \tfrac{n-2}{2} b^2-a) \, \overset{\circ}{\textup{Ric}}(S) \owedge \overset{\circ}{\textup{Ric}}(S) + 2a \, \textup{Ric}(S) \owedge \textup{Ric}(S) + 2b^2 \, \overset{\circ}{\textup{Ric}}(S)^2 \owedge \textup{id} \\
    & \qquad + \frac{nb^2(1-2b) - 2(a-b)(1-2b+nb^2)}{n(1+2 (n-1) a)} |\overset{\circ}{\textup{Ric}}(S)|^2 \, \textup{id} \owedge \textup{id}.
    \end{split}
\end{equation}
In what follows, we will denote by $\kappa$ the last coefficient appearing in this evolution, namely
\begin{equation}\label{eqn:kappa}
    \kappa(a,b) := \frac{nb^2(1-2b) - 2(a-b)(1-2b+nb^2)}{n (1 + 2(n-1) a)}
\end{equation}
Taking traces in this expression, we obtain the Ricci tensor and the scalar curvature as
\begin{align}
    \textup{Ric} (D_{a,b}(S)) 
        &= - 4b \, \textup{Ric}(S)^2 
        + \frac{8}{n} \Bigl(b + \frac{n-2}{2} a \Bigr)
        \textup{scal}(S) \textup{Ric}(S), \notag \\
        &\qquad + \frac{2}{n}
        \frac{n^2 b^2 - 2(n-1)(a-b)(1-2b)}
        {1+ 2(n-1) a} |\overset{\circ}{\textup{Ric}}(S)|^2 \, \textup{id}
        + \frac{4}{n^2} (a-b) \textup{scal}(S)^2 \, \textup{id},
        \label{eqn:ric-dab}\raisetag{1.2\baselineskip} \\
    \textup{scal} (D_{a,b}(S)) 
        &=  - 4b \, |\textup{Ric}(S)|^2 
        + \frac{4}{n} (b + (n-1) a) \textup{scal}(S)^2, \notag \\
        &\qquad + 2 \frac{n^2 b^2 - 2(n-1) (a-b) (1-2b)}
        {1+2(n-1)a} |\overset{\circ}{\textup{Ric}}(S)|^2.
        \label{eqn:scal-dab}\raisetag{1.2\baselineskip}
\end{align}
Moreover, if the algebraic curvature operator $R \in S^2_B(\fs \fo(n))$ evolves by Hamilton's ODE $\frac{d}{dt} R = \cQ(R)$, then the operators $S = \ell_{a,b}^{-1}(R)$ and $\textup{Ric}(S)$ evolve by
\allowdisplaybreaks{
\begin{align}
    \frac{d}{dt} S &= \cQ(S) + D_{a,b}(S), \label{eqn:evolve-S} \\
    \frac{d}{dt} \textup{Ric}(S) &= 2 \, S \ast \textup{Ric}(S) - 4b \, \textup{Ric}(S)^2 + \frac{8}{n} \Bigl( b + \frac{n-2}{2} a \Bigr) \textup{scal}(S) \textup{Ric}(S) \label{eqn:evolve-ric-S} \\
    & \qquad + \frac{2}{n} \frac{n^2 b^2 - 2(n-1) (a-b)(1-2b)}{1 + 2(n-1) a} |\overset{\circ}{\textup{Ric}}(S)|^2 \, \textup{id} + \frac{4}{n^2} (a-b) \, \textup{scal}(S)^2 \, \textup{id}, \notag
\end{align}}
where we denoted $(S \ast H)_{ik} := \sum_{p,q=1}^n S_{ipkq} H_{pq}$.
Along a path $a = a(b)$, we define the functions
\begin{equation}\label{eqn:P(ab)Q(ab)}
    P := 2 \frac{(1+ (n-2) b)^2}{1+2 (n-1) a}, \qquad Q := \frac{2}{n} \frac{(1+ 2(n-1) a)^2 - (1+ (n-2) b)^2}{1+2(n-1)a}.
\end{equation}
By~\eqref{eqn:scal-dab}, an algebraic curvature operator $S$ evolving by the ODE $\frac{d}{dt} S = \cQ(S) + D_{a,b}(S)$ satisfies
\begin{equation}\label{eqn:ddt-scal}
    \frac{d}{dt} \textup{scal}(S) = P \, |\textup{Ric}(S)|^2 + Q \, \textup{scal}(S)^2.
\end{equation}

We record that weakly positive isotropic curvature implies the weak three-positivity of the Ricci tensor; see also~\cite{bamler-cabezas-rivas-wilking}*{Proposition 2.2} and~\cite{brendle-PIC}*{Lemma A.4}.

\begin{lemma}\label{lemma:consequences-of-PIC}
For $n \geq 5$, let $T \in S^2_B(\fs \fo(n)) \cap C_{\textup{PIC}}$ be a weakly PIC algebraic curvature operator.
Then, for every $k$-plane $\Pi_k$ with $3 \leq k \leq n$, we have $\textup{tr}_{\Pi_k} \textup{Ric}(T) \geq 0$.
\end{lemma}
\begin{proof}
Throughout, we will apply the WPIC inequality on four-frames $\{ e_i , e_j , e_k, \pm e_{\ell} \}$ for pairwise distinct $i,j,k,\ell$ and average over the two signs, whereby
\begin{equation}\label{eqn:Tijkl}
    T_{ikik} + T_{i \ell i \ell} + T_{jkjk} + T_{j \ell j \ell} \geq 0.
\end{equation}
Fix a $3$-frame $\{ e_1, e_2 , e_3\}$ and apply this inequality to the three partitions of $\{ 1,2, 3, \alpha\}$, for each $\alpha$:
\begin{align*}
    T_{1313} + T_{1\alpha 1 \alpha} + T_{2323} + T_{2 \alpha 2 \alpha} &\geq 0, \\
    T_{1212} + T_{1 \alpha 1 \alpha } + T_{2323} + T_{3 \alpha 3 \alpha } &\geq 0, \\
    T_{1212} + T_{1313} + T_{2 \alpha 2 \alpha} + T_{3 \alpha 3 \alpha} &\geq 0.
\end{align*}
Adding these inequalities and dividing by $2$, we obtain
\[
(T_{1212} + T_{1313} + T_{2323}) + (T_{1 \alpha 1 \alpha} + T_{2 \alpha 2 \alpha} + T_{3 \alpha 3 \alpha}) \geq 0.
\]
Summing this inequality over $a \geq 4$ shows that
\begin{equation}\label{eqn:3-ric}
    (n-3) (T_{1212} + T_{1313} + T_{2323}) + \sum_{\alpha=4}^n (T_{1 \alpha 1 \alpha} + T_{2 \alpha 2 \alpha} + T_{3 \alpha 3 \alpha}) \geq 0.
\end{equation}
Next, for distinct $\alpha, \beta \geq 4$, we apply~\eqref{eqn:Tijkl} to the four-frames
\[
\{ e_1, e_2, e_{\alpha}, e_{\beta} \}, \qquad \{ e_1, e_3, e_{\alpha} ,  e_{\beta} \}, \qquad \{ e_2, e_3, e_{\alpha} , e_{\beta} \}.
\]
Adding the resulting three inequalities and dividing by $2$ gives
\[
(T_{1 \alpha 1 \alpha } + T_{2 \alpha 2 \alpha} + T_{3 \alpha 3 \alpha}) + (T_{1 \beta 1 \beta} + T_{2 \beta 2 \beta} + T_{3 \beta 3 \beta}) \geq 0.
\]
Summing this inequality over all pairs $4 \leq \alpha<\beta \leq n$, we deduce that
\[
\sum_{\alpha=4}^n (T_{1 \alpha 1 \alpha} + T_{2 \alpha 2 \alpha} + T_{3 \alpha 3 \alpha}) \geq 0.
\]
Combining this inequality with~\eqref{eqn:3-ric}, multiplied by $\frac{2}{n-3}$, we obtain
\begin{align*}
    0 &\leq 2 ( T_{1212} + T_{1313} + T_{2323}) + \tfrac{2}{n-3} \sum_{\alpha=4}^n (T_{1 \alpha 1 \alpha} + T_{2 \alpha 2 \alpha} + T_{3 \alpha 3 \alpha}) \\
    &\leq 2 ( T_{1212} + T_{1313} + T_{2323}) + \sum_{\alpha=4}^n (T_{1 \alpha 1 \alpha} + T_{2 \alpha 2 \alpha} + T_{3 \alpha 3 \alpha})
\end{align*}
because the second term is non-negative and $\frac{2}{n-3} \leq 1$ for $n \geq 5$.
Finally, we compute that
\[
\textup{Ric}(T)_{11} + \textup{Ric}(T)_{22} + \textup{Ric}(T)_{33} = 2 ( T_{1212} + T_{1313} + T_{2323}) + \sum_{\alpha=4}^n (T_{1 \alpha 1 \alpha} + T_{2 \alpha 2 \alpha} + T_{3 \alpha 3 \alpha}) \geq 0
\]
by the above expressions.
Therefore, $\textup{tr}_{\Pi_3} \textup{Ric}(T) \geq 0$ for every $3$-frame $\Pi_3$, and the result extends to every $k$-frame with $3 \leq k \leq n$ by averaging.
This completes the proof.
\end{proof}

\begin{lemma}\label{lemma:pinching-from-PIC}
Let $H$ be a symmetric bilinear form on $\mathbb{R}^n$ with $\textup{tr}(H)>0$ and suppose that $\lambda_{\max} (H) \leq \frac{1}{2} \textup{tr}(H)$ and $\textup{tr}_{\Pi_2} H \geq \frac{2 \alpha}{n} \textup{tr}(H)$ on every two-plane $\Pi_2$, for some $\alpha \leq 1$.
Then, 
\begin{enumerate}[(i)]
    \item Every unit vector $v$ satisfies
    \begin{equation}\label{eqn:h-bound-1}
     \tfrac{2(n-1) \alpha - n}{n(n-2)} \textup{tr}(H) \leq H(v,v) \leq \min \bigl\{ 1 - \tfrac{n-1}{n} \alpha , \tfrac{1}{2} \bigr\} \, \textup{tr}(H).   
    \end{equation}
    \item Every $k$-plane $\Pi_k$, with $2 \leq k \leq n-2$, satisfies
    \begin{equation}\label{eqn:trace-k-plane}
    \max \bigl\{ \tfrac{k}{n} \alpha, 1 - \tfrac{n-k}{2} \bigr\} \, \textup{tr}(H) \leq \textup{tr}_{\Pi_k} H \leq \min \bigl\{ 1 - \tfrac{(n-k) \alpha}{n} , \tfrac{k}{2} \bigr\} \textup{tr}(H).
    \end{equation}
\end{enumerate}
\end{lemma}
\begin{proof}
We first obtain a lower bound for all $k$-planes $\Pi_k = \textup{span} \{ e_1, \dots, e_k \}$.
We extend the vectors $\{ e_i \}_{i=1}^k$ to an $n$-frame $\{ e_1, \dots, e_k, f_1, \dots, f_{n-k}\}$ and use $H(f_{\alpha}, f_{\alpha}) \leq \frac{1}{2} \textup{tr}(H)$ for each $\alpha$, coming from the property $\lambda_{\max} (H) \leq \frac{1}{2} \textup{tr}(H)$.
Therefore,
\begin{equation}\label{eqn:planes-bound-1}
\textup{tr}_{\Pi_k} H = \textup{tr}(H) - \textstyle{ \sum_{\alpha=1}^{n-k} H(f_{\alpha}, f_{\alpha}) } \geq \textup{tr}(H) - \tfrac{1}{2} (n-k) \, \textup{tr}(H).
\end{equation}
On the other hand, applying the bound~\eqref{eqn:planes-bound-1} to the plane $\Pi^{\perp}_k$ produces
\begin{equation}\label{eqn:planes-bound-2}
\textup{tr}(H) - \textup{tr}_{\Pi_k} H = \textup{tr}_{\Pi^{\perp}_k} H \geq \textup{tr}(H) - \tfrac{1}{2} k \, \textup{tr}(H).
\end{equation}
Combining these bounds shows that, for every $k$-plane $\Pi_k$,
\begin{equation}\label{eqn:first-part-18-19}
    ( 1 - \tfrac{n-k}{2} ) \, \textup{tr}(H) \leq \textup{tr}_{\Pi_k} H \leq \tfrac{k}{2} \, \textup{tr}(H).
\end{equation}
Next, for $k \geq 2$, we apply the bound $\textup{tr}_{\Pi_2} H \geq \frac{2 \alpha}{n} \textup{tr}(H)$ for each $2$-plane $\textup{span} \{ e_i, e_j \}$ with $i<j$ and sum over the $\binom{k}{2} = \frac{k(k-1)}{2}$ such pairs.
Thus, we obtain
\[
\frac{k(k-1)}{2} \frac{2 \alpha}{n} \textup{tr}(H) \leq \sum_{1 \leq i < j \leq k} \bigl[ H(e_i, e_i) + H(e_j, e_j) \bigr] = (k-1) \sum_{i=1}^k H(e_i, e_i) =(k-1) \, \textup{tr}_{\Pi_k} H.
\]
This implies the bound
\begin{equation}\label{eqn:second-k-plane-bound}
    \textup{tr}_{\Pi_k } H \geq \tfrac{k}{n} \alpha \, \textup{tr}(H), \qquad \text{for every } \; k \geq 2.
\end{equation}
Moreover, if $k \leq n-2$, we can apply the bound~\eqref{eqn:second-k-plane-bound} to $\Pi_k^{\perp}$ to obtain
\begin{equation}\label{eqn:final-k-plane-bound}
    \textup{tr}_{\Pi_k} H = \textup{tr}(H) - \textup{tr}_{\Pi^{\perp}_k} (H) \leq (1 - \tfrac{n-k}{n} \alpha ) \, \textup{tr}(H).
\end{equation}
Finally, for $k=1$, the above extension yields
\[
\tfrac{2(n-1)}{n} \alpha \, \textup{tr}(H) \leq \sum_{i=1}^{n-1} \textup{tr}_{ \la v, e_i \rg } H = (n-2) H(v,v) + \textup{tr}(H).
\]
The claimed inequalities~\eqref{eqn:h-bound-1} and~\eqref{eqn:trace-k-plane} follow from combining these properties.
\end{proof}

\section{A one-parameter family of invariant cones \texorpdfstring{$C_1(b)$}{C1(b)} }\label{section:first-family}

In this section we construct a one-parameter family of cones $C_1(b)$ that are invariant under the Hamilton ODE $\frac{d}{dt} R = \cQ(R)$.
\begin{definition}\label{def:first-cone-family}
For $n \geq 5$, we let $b_1(n) = \frac{1}{4(n-2)}$ and $a(b) = b + \frac{n-1}{2} b^2$.
For $0 < b \leq b_1$, let $\cE_1(b)$ denote the set of algebraic curvature operators $S \in S^2_B(\fs \fo(n))$ satisfying the following two conditions.
\begin{enumerate}[(i)]
    \item For every orthonormal pair $\{ e_1, e_2 \}$, we have the inequality
\begin{equation}\label{eqn:ric-11-first}
    \textup{Ric}(S)_{11} + \textup{Ric}(S)_{22} + \tfrac{4}{ 5n} b \, \textup{scal}(S) \geq 0.
\end{equation}
    \item There exists a non-negative operator $N \in S^2 (\fs \fo(n))$, denoted $N \geq 0$ and a tensor $U \in C_{\textup{PIC}}$ such that $S = U+N$ and for every orthonormal frame $\{ e_1, \dots, e_n \}$, we have
\begin{equation}\label{eqn:S-define-ON}
\begin{split}
    \textup{Ric}(S)_{22} - \textup{Ric}(S)_{11} \leq \left[ \frac{\textup{scal}(S)^{\frac{1}{2}}}{(n-2) b} \sum_{p=3}^n (N_{1p1p} +N_{2p2p}) \right]^{\frac{1}{2}}.
\end{split}
\end{equation}
We note that the tensors $U,N \in S^2(\fs \fo(n))$ are not required to satisfy the first Bianchi identity.
\end{enumerate}
For $0 < b \leq b_1$, we define the cone of operators $C_1(b) := \ell_{a,b}(\cE_1(b))$, where $\ell_{a,b}$ is given in~\eqref{eqn:ell-a,b}.
\end{definition}
As in the construction of~\cite{brendle-PIC}*{\S 3}, we see that the family $C_1(b)$ is closed, convex, and $O(n)$-invariant for each $0 < b \leq b_1$.
Moreover, every $S \in \cE_1(b) \setminus \{ 0 \}$ has $\textup{scal}(S)>0$.

We first collect some properties resulting from these definitions, in terms of the functions $\kappa, P, Q$ introduced in Section~\ref{sec:preliminaries}.
We observe the following identities.
\begin{equation}\label{eqn:P-Q-kappa-identites}
\begin{split}
    1 + 2(n-1) a &= (1 + (n-1) b)^2, \\
    \kappa &= \frac{b^2 (1-2b-n(n-1) b^2)}{n (1 + (n-1) b)^2}, \\
    P &= 2 \Bigl( \frac{1+(n-2) b}{1+(n-1) b} \Bigr)^2, \\
    Q &= 2b \frac{ ( n + (n-1)^2 b) (2 + (3n-4) b + (n-1)^2 b^2) }{n (1 + (n-1) b)^2}.
\end{split}
\end{equation}

To show that the family $C_1(b)$ is invariant under the Hamilton ODE, we will need some elementary properties of the parameters, which we collect below.

\begin{lemma}\label{lemma:important-properties}
We define $a = b + \frac{n-1}{2} b^2$ and $b_1(n) = \frac{1}{4(n-2)}$.
For every $n \geq 5$ and every $b \in (0,b_1(n)]$, the following properties are satisfied.
\begin{enumerate}[(i)]
    \item The quantities $\kappa, P , Q$ are strictly positive.
    Moreover, $\frac{7}{4} < P < 2$.
    \item We consider the quantity
    \[
    K := \frac{2+(14n-13) b + (12n^2 - 28n+18) b^2}{5n (1 + (n-1) b)^2} b.
    \]
    This satisfies $K \geq \frac{2}{5n} b$.
    Moreover, the cubic function
    \[
    (n-2) b z^3 + (K-b) z^2 + K \Bigl[ \frac{(1 + \frac{2}{5} b)^2}{n(n-2)} + \frac{n-2}{2(n-3)} \Bigl( \frac{1 + \frac{2}{5} b}{n-2} - z \Bigr)^2 \Bigr]  + \frac{b(1 + \frac{2}{5} b) ( 2 +(5n-9) b)}{10n^2}
    \]
    is strictly positive for $z \geq 0$.
    \item We define the function
\begin{align*} 
\cF(\xi, s, \tau) &:=\frac{P \xi^4}{4} + \Bigl[ \frac{P}{2n} + \frac{Q}{2} + s + \frac{n}{4(n-2)} P \tau^2 + \Bigl( 4b + \frac{n}{n-2} \Bigr) \tau - \frac{4(n-2)}{n} a \Bigr] \, \xi^2  +\frac{\xi^2+s^2}{2(n-2)^2 b} 
\end{align*}
We suppose that $\xi \geq 0$ and
\[
\tau \geq - \frac{2}{n} \Bigl( 1 + \frac{2}{5} b \Bigr), \qquad s \geq \frac{2(n-2)(n-3) b \xi^2 - (n-1)(n-4) \tau - \frac{2(n-2)^2}{n}}{n^2 - 7n+14}.
\]
Then, it holds that $\cF(\xi,s,\tau) \geq 0$.
\end{enumerate}
\end{lemma}
\begin{proof}
We separate the proof of the lemma into each item.

\smallskip \noindent \textbf{Property $(i)$.}
The positivity of $P, Q$ is clear. 
For $\kappa$, we observe that
\[
2b + n(n-1) b^2 \leq \tfrac{1}{2(n-2)} + \tfrac{n(n-1)}{16(n-2)^2} < \tfrac{n+2}{4(n-2)} \leq \tfrac{7}{12} < 1.
\]
Thus, $\kappa>0$.
Next, we prove $P \in ( \frac{7}{4}, 2)$.
Because $b \leq \frac{1}{4(n-2)} \leq \frac{1}{12}$ for $n \geq 5$, we have
\[
\tfrac{1+(n-2) b}{1+(n-1) b} = 1 - \tfrac{b}{1+(n-1) b} \geq 1 - \tfrac{1}{5n-9} \geq \tfrac{15}{16}.
\]
Therefore, $P \geq 2 ( \frac{15}{16})^2 = \frac{225}{128} > \frac{7}{4}$.
The reverse bound $P<2$ is clear.

\smallskip \noindent \textbf{Property $(ii)$.}
The bound $K \geq \frac{2b}{5n}$ follows from the identity
\[
K - \frac{2b}{5n} = b^2 \frac{10 n - 9 + (10n^2 - 24n+16) b}{5n (1+(n-1) b)^2} > 0
\]
where all terms are positive for $n \geq 5$.
To study the asserted cubic function, we first collect the powers of $z$ to see that this expression equals
\[
    (n-2) b z^3 - bz^2 + K \left( \frac{3n-8}{2(n-3)} z^2 - \frac{1+ \frac{2}{5} b}{n-3} z + \frac{3 ( 1 + \frac{2}{5} b)^2}{2n(n-3)} \right) + \frac{b(1 + \frac{2}{5} b) (2 +(5n-9) b)}{10 n^2}.
    \]
To estimate the first term, we use Young's inequality in the form
\[
(n-2) b z^3 - bz^2 \geq - \tfrac{4b}{27(n-2)^2}
\]
for $z \geq 0$.
Moreover, the function multiplying $K$ in the above expression can be estimated by
\[
\frac{3n-8}{2(n-3)} z^2 - \frac{1+ \frac{2}{5} b}{n-3} z + \frac{3 ( 1 + \frac{2}{5} b)^2}{2n(n-3)} \geq \frac{4(1 + \frac{2}{5} b)^2}{n(3n-8)} \geq \frac{4}{n(3n-8)}.
\]
Finally, the constant term of the cubic function satisfies the bound
\[
\tfrac{1}{10n^2} b ( 1 + \tfrac{2}{5} b) \, [ 2 + (5n-9) b] \geq \tfrac{1}{5n^2} b \qquad \text{for } \; b \geq 0.
\]
Combining these bounds with $K \geq \tfrac{2}{5n} b$, we deduce that
\begin{align*}
&(n-2) b z^3 - bz^2 + K \bigl( \tfrac{3n-8}{2(n-3)} z^2 - \tfrac{1+ \frac{2}{5} b}{n-3} z + \tfrac{3 ( 1 + \frac{2}{5} b)^2}{2n(n-3)} \bigr) + \tfrac{b(1 + \frac{2}{5} b) (2 +(5n-9) b)}{10 n^2} \\
&\geq b \bigl[ \tfrac{1}{5n^2} + \tfrac{8}{5n^2(3n-8)} - \tfrac{4}{27(n-2)^2} \bigr] = \tfrac{21n^2 - 164 n + 324}{135 n (n-2)^2 (3n-8)} b > 0.
\end{align*}
Indeed, the numerator of the last expression is increasing for $n \geq 5$ and equals $29$ at $n=5$.

\smallskip \noindent \textbf{Property $(iii)$.}
For the function $\cF(\xi,s,\tau)$, we use the identities~\eqref{eqn:P-Q-kappa-identites} for $P$ and $Q$ to write
\[
\tfrac{1}{2} ( \tfrac{1}{n}P+Q) - \tfrac{4(n-2)}{n} a = \tfrac{1}{n} (1 -2(n-3) b - (n-1)(n-3) b^2).
\]
Next, we treat $\cF$ as a quadratic function of $s$.
Completing the square, we obtain the inequality
\begin{align*}
      s \xi^2 + \frac{s^2}{2(n-2)^2 b} &= \frac{(s + 3(n-2) b \xi^2)^2}{2(n-2)^2 b} + \frac{n-5}{n-2} s \xi^2 - \frac{9}{2} b \xi^4 \\
    &\geq \xi^2 \Bigl( \frac{n-5}{n-2} s - \frac{9}{2} b \xi^2 \Bigr).  
\end{align*}
We now apply the assumed lower bound for $s$ to obtain
\begin{align*}
 s \xi^2 + \frac{s^2}{2(n-2)^2 b} &\geq \Bigl[ \frac{2(n-3)(n-5) b}{n^2-7n+14} - \frac{9}{2} b \Bigr] \, \xi^4 - \frac{(n-1)(n-4)(n-5)}{(n-2)(n^2-7n+14)} \tau \xi^2 - \frac{2(n-2)(n-5)}{n(n^2-7n+14)} \xi^2.
\end{align*}
We substitute these expressions into the definition of $\cF(\xi,s,\tau)$ to obtain
\begin{align*} 
\xi^{-2} \cF(\xi, s, \tau) &\geq \Bigl( \frac{P}{4} + \frac{2(n-3)(n-5) b}{n^2-7n+14} - \frac{9}{2} b  \Bigr) \xi^2 + \frac{nP}{4(n-2)} \tau^2 + \Bigl( 4b + \frac{3n^2-15n+20}{(n-2)(n^2-7n+14)} \Bigr) \tau \\
    &\quad + \frac{1-2(n-3)b-(n-1)(n-3)b^2}{n} + \frac{1}{2(n-2)^2 b} - \frac{2(n-2)(n-5)}{n(n^2-7n+14)}.
\end{align*}
The coefficient of $\xi^2$ on the right-hand side is positive, due to $b \leq \frac{1}{12}$, $P> \frac{7}{4}$, and
\[
\tfrac{1}{4} P - \tfrac{9}{2} b + \tfrac{2(n-3)(n-5)}{n^2-7n+14} b \geq \tfrac{7}{16} - \tfrac{3}{8} = \tfrac{1}{16}.
\]
Thus, to show that $\cF(\xi, s, \tau) \geq 0$, it suffices to evaluate the right-hand side at $\xi=0$.
The remaining expression on the right-hand side is a quadratic polynomial in the variable $\tau \geq - \frac{2}{n} ( 1+ \frac{2}{5} b)$, given by
\begin{align*}
    \cG(\tau) &:= \tfrac{nP}{4(n-2)} \tau^2 + \bigl( 4b + \tfrac{3n^2-15n+20}{(n-2)(n^2-7n+14)} \bigr) \tau \\
    & \quad + \tfrac{1}{2(n-2)^2} b^{-1} - \tfrac{2(n-2)(n-5)}{n(n^2-7n+14)} + \tfrac{1}{n} ( 1-2(n-3)b-(n-1)(n-3)b^2 ).
\end{align*}
For $n \geq 5$, we observe the bound
\[
\tfrac{3n^2-15n+20}{(n-2)(n^2-7n+14)} = \tfrac{3}{n-2} + \tfrac{3n^2-15n+20}{(n-2)(n^2-7n+14)} > \tfrac{3}{n-2}.
\]
Thus, using the bound $P < 2$, we see that for $\tau \geq - \frac{2}{n}(1 + \frac{2}{5} b)$, the derivative $\cG'(\tau)$ satisfies
\begin{align*}
    \cG'(\tau) = \tfrac{n}{2(n-2)} P\tau + 4b + \tfrac{3n^2-15n+20}{(n-2)(n^2-7n+14)} &\geq - \tfrac{n}{2(n-2)} \cdot 2 \cdot \tfrac{2}{n} ( 1 + \tfrac{2}{b}) + 4b + \tfrac{3}{n-2} \\
    &\geq - \tfrac{2}{n-2} ( 1 + \tfrac{2}{5} b) + 4b + \tfrac{3}{n-2}. 
\end{align*}
Hence, $\cG'(\tau) \geq \frac{1}{n-2} + b > 0$ for $\tau \geq - \frac{2}{n}(1 + \frac{2}{5} b)$, so to prove $\cG \geq 0$, it suffices to compute that $\cG ( - \frac{2}{n} (1 + \frac{2}{5} b)) \geq 0$.
Using $P > \frac{7}{4}$ and $(1 + \frac{2}{5} b)^2 > 1$, we can bound $\cG ( - \frac{2}{n} (1 + \frac{2}{5} b))$ from below by
\begin{align*}
    n \, \cG \Bigl( - \frac{2}{n} \bigl(1 + \frac{2}{5} b \bigr) \Bigr) &\geq \frac{7}{4(n-2)} - 2 \Bigl( 1 + \frac{2}{5} b \Bigr) \Bigl( 4b + \frac{3n^2 - 15n+20}{(n-2)(n^2 - 7n+14)} \Bigr) \\
    & \quad + 1 - 2(n-3) b - (n-1)(n-3) b^2 + \frac{n}{2(n-2)^2 b} - \frac{2(n-2)(n-5)}{n^2-7n+14}.
\end{align*}
Finally, the right-hand side of this inequality is strictly decreasing in $b>0$, so it suffices to evaluate the final expression at $b = \frac{1}{4(n-2)}$.
Collecting the resulting expressions for $b = \frac{1}{4(n-2)}$ over a common denominator, we obtain the expression
\[
\frac{35(n-6)^4+ 315 (n-6)^3 + 601(n-6)^2 - 231 (n-6) + 584}{80 (n-2)^2 (n^2 - 7n+14)}
\]
which is clearly positive for $n \geq 5$.
To summarize, we have proved that $\cF(\xi,s,\tau) \geq 0$, as claimed.
\end{proof}

\begin{lemma}\label{lemma:S-curvature-optimal-bound}
    Consider an algebraic curvature operator $S \in \cE_1(b)$ and let $\lambda_1 \leq \lambda_2 \leq \lambda_3$ be the three smallest eigenvalues of $\textup{Ric}(S)$.
    We decompose $S = U + N$ in $S^2(\fs \fo(n))$ according to Definition~\ref{def:first-cone-family}, with $U \in C_{\textup{PIC}}$ and $N$ a non-negative operator.
    Suppose that $\lambda_1 + \lambda_2 = - \frac{4}{5n} b\,  \textup{scal}(S)$.
    Then, we have
    \[
    \sum_{p=3}^n (S_{1p1p} + S_{2p2p} ) \Bigl( \lambda_p - \frac{\lambda_1 + \lambda_2}{2} \Bigr) \geq \Bigl( \sum_{p=3}^n N_{1p1p} + N_{2p2p} \Bigr) \Bigl( \lambda_3 - \frac{\lambda_1+\lambda_2}{2} \Bigr). 
    \]
    Moreover, we have the bounds
    \[
    \lambda_1 \geq - \frac{n +  \frac{4}{5}(n-1) b}{n(n-2)} \textup{scal}(S), \qquad \lambda_3 \leq \frac{1 + \frac{4}{5n}b}{n-2} \textup{scal}(S), \qquad \lambda_3 - \lambda_1 \leq \frac{2(1 + \frac{2}{5} b)}{n-2} \textup{scal}(S).
    \]
    Finally, it holds that
    \begin{equation}\label{eqn:ric-0-inequality}
        |\overset{\circ}{\textup{Ric}}(S)|^2 \geq \frac{1}{2} ( \lambda_2 - \lambda_1)^2 + \frac{2(1+\frac{2}{5} b)^2}{n(n-2)} \textup{scal}(S)^2 + \frac{n-2}{n-3} \Bigl( \frac{ n + \frac{4}{5} b}{n(n-2)} \textup{scal}(S) - \lambda_3 \Bigr)^2.
    \end{equation}
\end{lemma}
\begin{proof}
    The first inequality is a consequence of~\cite{brendle-PIC}*{Lemma A.3}, applied as in~\cite{brendle-PIC}*{Lemma 3.9}.
    For the second bound, let $\sigma := \textup{scal}(S)$.
    The hypothesis gives $\lambda_1 + \lambda_2 = - \frac{4}{5n} b  \sigma$, hence
    \[
    \sigma = \lambda_1+\lambda_2 +\sum_{p=3}^n \lambda_p \geq - \frac{4}{5n} b \sigma + 
    (n-2) \lambda_3  \implies \lambda_3 \leq \frac{1 + \frac{4}{5n} b}{n-2} \sigma.
    \]
    We also observe that $\lambda_1 + (\lambda_2 - \lambda_1) = \lambda_2 \leq \lambda_3$, while $2 \lambda_1 + (\lambda_2 - \lambda_1) = \lambda_1 + \lambda_2 = - \frac{4}{5n} b \sigma$.
    Therefore,
    \[
    \lambda_1 \geq - \tfrac{\frac{4}{5} (n-1) b + n}{n(n-2)} \sigma, \qquad \implies \qquad \lambda_3 - \lambda_1 \leq \tfrac{n + \frac{4}{5} b n + n}{n(n-2)} \sigma = \tfrac{2(1+\frac{2}{5} b)}{n-2} \sigma
    \]
    as claimed.
    This proves the first part of our assertion.

    To prove the bound~\eqref{eqn:ric-0-inequality}, we use $\lambda_1 + \lambda_2 = - \frac{4}{5n} b \sigma$ to compute
    \begin{align*}
    \Bigl( \lambda_1 - \frac{\sigma}{n} \Bigr)^2 + \Bigl(\lambda_2 - \frac{\sigma}{n} \Bigr)^2
    %\\ &= \Bigl( - \frac{1+\frac{2}{5} b}{n} \sigma - \frac{\lambda_2 - \lambda_1}{2} \Bigr)^2 + \Bigl( - \frac{1+\frac{2}{5} b}{n} \sigma + \frac{\lambda_2 - \lambda_1}{2} \Bigr)^2 \\
    &= \frac{2}{n^2} \Bigl( 1+ \frac{2}{5} b  \Bigr)^2 \sigma^2 + \frac{1}{2} (\lambda_2 - \lambda_1)^2.
    \end{align*}
    In addition, for the $n-3$ terms involving $\lambda_4, \dots, \lambda_n$, the Cauchy-Schwarz inequality implies
    \[
    \sum_{p=4}^n \Bigl( \lambda_p - \frac{1}{n} \sigma \Bigr)^2 \geq \frac{1}{n-3} \Bigl( \sum_{p=4}^n \lambda_p - \frac{n-3}{n} \sigma \Bigr)^2.
    \]
    On the other hand, we can write
    \[
    \sum_{p=4}^n \lambda_p = \sigma - ( \lambda_1 + \lambda_2 + \lambda_3) = \sigma + \frac{4}{5n} b \sigma - \lambda_3 = \frac{n+ \frac{4}{5} b }{n} \sigma - \lambda_3. 
    \]
    Adding these terms, we therefore obtain
    \begin{align*}
        \sum_{p=3}^n \Bigl( \lambda_p - \frac{\sigma}{n} \Bigr)^2 &\geq \Bigl( \lambda_3 - \frac{\sigma}{n} \Bigr)^2 + \frac{1}{n-3} \Bigl( \frac{n+ \frac{4}{5} b }{n} \sigma - \lambda_3 - \frac{n-3}{n} \sigma \Bigr)^2 \\
        &= \Bigl( \lambda_3 - \frac{\sigma}{n} \Bigr)^2 + \frac{1}{n-3} \Bigl( \frac{3+ \frac{4}{5} b }{n} \sigma - \lambda_3 \Bigr)^2 \\
        &= \frac{n-2}{n-3} \Bigl( \lambda_3 - \frac{n+\frac{4}{5} b}{n(n-2)} \sigma \Bigr)^2 + \frac{4(1 + \frac{2}{5} b)^2}{n^2(n-2)} \sigma^2.
    \end{align*}
    Simplifying the last expression, adding $|\overset{\circ}{\textup{Ric}}(S)|^2 = \sum_{p=1}^n ( \lambda_p - \frac{1}{n} \sigma)^2$, and combining the above computations establishes the desired inequality~\eqref{eqn:ric-0-inequality}.
    This completes the proof.
\end{proof}

\begin{lemma}\label{lemma:DS(a,b)-lower bound}
For any $S \in \cE_1(b) \setminus \{ 0 \}$ and any orthonormal four-frame $\{ e_1, e_2, e_3, e_4 \}$, we have
\begin{align*}
    \frac{b^2}{6n^2}&\leq  \frac{1}{\textup{scal}(S)^2} \Biggl[ b^2 \sum_{i=1}^4 ( \overset{\circ}{\textup{Ric}}(S)^2)_{ii} - \frac{1}{2} b^2  \bigl( \overset{\circ}{\textup{Ric}}(S)_{11} + \overset{\circ}{\textup{Ric}}(S)_{22} \bigr) \bigl( \overset{\circ}{\textup{Ric}}(S)_{33} + \overset{\circ}{\textup{Ric}}(S)_{44} \bigr) \Biggr] \\
    & \quad + \Bigl( b + \frac{n-1}{2} b^2  \Bigr) \, \frac{\bigl( \textup{Ric}(S)_{11} + \textup{Ric}(S)_{22} \bigr) \bigl( \textup{Ric}(S)_{33}+ \textup{Ric}(S)_{44} \bigr)}{\textup{scal}(S)^2}.   
\end{align*}
\end{lemma}
\begin{proof}
We fix $S \in \cE_1(b) \setminus \{ 0 \}$, denote by $L(b)$ the infimum of the expression appearing on the right-hand side over orthonormal four-frames, and compute at a particular four-frame $\{ e_1, e_2, e_3, e_4 \}$ attaining $L(b)$.
For abbreviation, let $\sigma := \textup{scal}(S)$.
The Cauchy-Schwarz inequality implies 
\begin{align*}
    2 \sum_{i=1}^4 ( \overset{\circ}{\textup{Ric}}(S)^2)_{ii} &\geq 2(\overset{\circ}{\textup{Ric}}(S)_{11})^2 + 2(\overset{\circ}{\textup{Ric}}(S)_{22})^2 + 2(\overset{\circ}{\textup{Ric}}(S)_{33})^2 + 2( \overset{\circ}{\textup{Ric}}(S)_{44})^2 \\
    &\geq \bigl( \overset{\circ}{\textup{Ric}}(S)_{11} + \overset{\circ}{\textup{Ric}}(S)_{22} \bigr)^2 + \bigl( \overset{\circ}{\textup{Ric}}(S)_{33} + \overset{\circ}{\textup{Ric}}(S)_{44} \bigr)^2.
\end{align*}
We now introduce the variables $x := \frac{n}{\sigma} ( \overset{\circ}{\textup{Ric}}(S)_{11} + \overset{\circ}{\textup{Ric}}(S)_{22})$ and $y := \frac{n}{\textup{scal}(S)} ( \overset{\circ}{\textup{Ric}}(S)_{33} + \overset{\circ}{\textup{Ric}}(S)_{44})$.
The two-Ricci assumption of Definition~\ref{def:first-cone-family} for $S \in \cE_1(b) \setminus \{ 0 \}$ implies that $x,y \geq - 2(1+\frac{2}{5}b)$.
Thus, the Cauchy-Schwarz inequality shows that the right-hand side is bounded from below by
\begin{equation}\label{eqn:bound-DS-from-psi}
L(b) \geq \tfrac{1}{2n^2} \, \bigl[ b^2(x^2+y^2 - xy) + (2 b + (n-1) b^2)  (x+2)(y+2) \bigr].
\end{equation}
Next, we observe the identity
\begin{align*}
0&\leq b^2 \bigl( x+y+\tfrac{11}{5} - \tfrac{2(n-4)}{5} b \bigr)^2 + (2b+(n-4) b^2) ( x + 2 + \tfrac{4}{5} b) (y+2+\tfrac{4}{5} b) \\
    &= b^2 (x^2+y^2-xy) + (2b+(n-1) b^2)(x+2)(y+2) - \tfrac{1}{25} b^2 [ 19 - 4 (9n-28) b - 4n(n-4) b^2 ].
\end{align*}
Combining this inequality with~\eqref{eqn:bound-DS-from-psi} and using $x,y \geq - 2(1 + \frac{2}{5} b)$, we deduce that
\[
L(b) \geq \tfrac{1}{50 n^2} b^2 [ 19 - 4 (9n-28) b - 4n(n-4) b^2 ].
\]
Finally, for $0 \leq b \leq b_1(n) = \frac{1}{4(n-2)}$, we notice that
\begin{align*}
    19 - 4 (9n-28) b - 4n(n-4) b^2 &\geq 19 - \tfrac{9n-28}{n-2} - \tfrac{n(n-4)}{4(n-2)^2}  = \tfrac{39}{4} + \tfrac{10}{n-2} + \tfrac{1}{(n-2)^2} > \tfrac{39}{4}.
\end{align*}
Consequently, $L(b) \geq \frac{39}{200 n^2} b^2 \geq \frac{1}{6n^2}b^2$.
This completes the proof of our assertion.
\end{proof}

\begin{corollary}\label{cor:interior-of-pic-cone}
For every tensor $S \in \cE_1(b) \setminus \{ 0 \}$, we have $D_{a,b}(S) \in \textup{int}( C_{\textup{PIC}})$.
Moreover, we recall the term $\kappa$ from~\eqref{eqn:kappa} and define the function $\mu_b(S)$ by
\begin{equation}\label{eqn:epsilon-S-general-definition}
    \mu_b(S) := \kappa \sigma^{-2} |\overset{\circ}{\textup{Ric}}(S)|^2 + \tfrac{1}{12 n^2} b^2.
\end{equation}
Then, it holds that
\[
D_{a,b}(S) - \mu_b(S)  \, \textup{scal}(S)^2 \textup{id} \owedge \textup{id} \in  C_{\textup{PIC}}.
\]
\end{corollary}
\begin{proof}
We let $\sigma := \textup{scal}(S)$ and use $a = b + \frac{n-1}{2} b^2$ to define the term
\[
V = - \tfrac{1}{2} b^2 \, \overset{\circ}{\textup{Ric}}(S) \owedge \overset{\circ}{\textup{Ric}}(S) + ( b + \tfrac{n-1}{2} b^2) \, \textup{Ric}(S) \owedge \textup{Ric}(S) + b^2 \, \overset{\circ}{\textup{Ric}}(S)^2 \owedge \textup{id} \, . 
\]
Then, the expression~\eqref{eqn:D-a-b} gives $D_{a,b}(S) = 2 \, V + \kappa \, |\overset{\circ}{\textup{Ric}}(S)|^2 \, \textup{id} \owedge \textup{id}$.
We recall the notation $Z_1(T;F)$ from~\eqref{eqn:I(T;F)} for the isotropic curvature functional.
Then, our claim will follow from
\[
Z_1(D_{a,b}(S) - \mu_b(S) \, \textup{scal}(S)^2 \, \textup{id} \owedge \textup{id} ; F) > 0
\]
for every orthonormal four-frame $\{ e_1, e_2, e_3, e_4\}$.
Since $( \textup{id} \owedge \textup{id})_{1p1p} = ( \textup{id} \owedge \textup{id})_{2p2p} = 2$, we have $Z_1(\textup{id} \owedge \textup{id}) = 8$, so it suffices to prove that
\begin{equation}\label{eqn:epsilon-S(b)}
\mu_b(S) \leq \tfrac{1}{8} \sigma^{-2} Z_1( D_{a,b}(S); F)
\end{equation}
for every orthonormal four-frame.
Following~\cite{brendle-PIC}*{Lemma 3.5}, we fix a frame $F= \{ e_1, e_2, e_3, e_4 \}$ and assume, without loss of generality, that
\[
\textup{Ric}(S)_{13} + \textup{Ric}(S)_{24} = 0, \qquad \textup{Ric}(S)_{14} - \textup{Ric}(S)_{23} = 0.
\]
In this frame, we compute
\begin{align*}
    Z_1(V;F) &= 2 b^2 \sum_{i=1}^4 ( \overset{\circ}{\textup{Ric}}(S)^2)_{ii} - b^2  \bigl( \overset{\circ}{\textup{Ric}}(S)_{11} + \overset{\circ}{\textup{Ric}}(S)_{22} \bigr) \bigl( \overset{\circ}{\textup{Ric}}(S)_{33} + \overset{\circ}{\textup{Ric}}(S)_{44} \bigr) \\
    & \quad + (2b + (n-1) b^2 ) \, \bigl( \textup{Ric}(S)_{11} + \textup{Ric}(S)_{22} \bigr) \bigl( \textup{Ric}(S)_{33}+ \textup{Ric}(S)_{44} \bigr).
\end{align*}
Thus, $\frac{1}{2 \, \textup{scal}(S)^2} Z_1(V;F)$ is precisely the quantity considered in Lemma~\ref{lemma:DS(a,b)-lower bound}.
Therefore,  
\[
Z_1(V;F) \geq \tfrac{1}{3n^2} b^2 \, \textup{scal}(S)^2.
\]
Moreover, $Z_1(\kappa \, |\overset{\circ}{\textup{Ric}}(S)|^2 \, \textup{id} \owedge \textup{id} ; F) = 8 \kappa \, |\overset{\circ}{\textup{Ric}}(S)|^2$.
Combining these terms and taking the infimum over four-frames $\{ e_1, e_2, e_3, e_4 \}$, we find 
\begin{align*}
    Z_1 ( D_{a,b}(S);F) &= 2 \, Z_1(V;F) + Z_1 ( \kappa \, |\overset{\circ}{\textup{Ric}}(S)|^2 \, \textup{id} \owedge \textup{id} ; F) \\
    &\geq \tfrac{2}{3n^2} b^2 \textup{scal}(S)^2 +  8 \kappa \, |\overset{\circ}{\textup{Ric}}(S)|^2.
\end{align*}
This proves the claimed inequality~\eqref{eqn:epsilon-S(b)}, and hence our assertion.
\end{proof}

Given a decomposition $S = U + N \in \cE_1(b)$ as in Definition~\ref{def:first-cone-family}, where the tensors $U,N \in S^2(\fs \fo(n))$ are not required to satisfy the first Bianchi identity, we define the auxiliary evolution of $U,N$ by
\begin{equation}\label{eqn:d-dt-N-evolution}
\begin{split}
\tfrac{d}{dt} N &= S^2 + \tfrac{1}{2} \mu_b(S) \, \textup{scal}(S)^2 \, \textup{id} \owedge \textup{id}, \\
\tfrac{d}{dt} U &= S^{\#} + D_{a,b}(S) - \tfrac{1}{2}  \mu_b(S) \, \textup{scal}(S)^2 \, \textup{id} \owedge \textup{id}.
\end{split}
\end{equation}
In view of Corollary~\ref{cor:interior-of-pic-cone}, this evolution of $U$ preserves the PIC cone, due to 
\[
D_{a,b}(S) - \tfrac{1}{2}  \mu_b(S) \, \textup{scal}(S)^2 \, \textup{id} \owedge \textup{id} \in \textup{int} (C_{\textup{PIC}}).
\]
Furthermore, this evolution preserves the tensor decomposition $\frac{d}{dt} S = \frac{d}{dt} U + \frac{d}{dt} N$, so it is admissible.
The tensors $U,N \in S^2(\fs \fo(n))$ here are not required to satisfy the first Bianchi identity.

Using $( \textup{id} \owedge \textup{id})_{1p1p} = (\textup{id} \owedge \textup{id})_{2p2p} = 2$ and the expression for $\mu_b(S)$ from~\eqref{eqn:epsilon-S-general-definition}, we see that the evolution of the non-negative operator $N$ according to~\eqref{eqn:d-dt-N-evolution} satisfies
\begin{equation}\label{eqn:N1p1p-evolution}
    \begin{split}
         \frac{d}{dt} \sum_{p=3}^n (N_{1p1p} + N_{2p2p}) &= \sum_{p=3}^n \bigl( (S^2)_{1p1p} + (S^2)_{2p2p} \bigr) + 2(n-2) \mu_b(S) \, \textup{scal}(S)^2 \\
         &= \sum_{p=3}^n \bigl( (S^2)_{1p1p} + (S^2)_{2p2p} \bigr) + 2(n-2) \kappa \, |\overset{\circ}{\textup{Ric}}(S)|^2 + \frac{n-2}{6 n^2} b^2 \textup{scal}(S)^2
    \end{split}
    \end{equation}
We now obtain a lower bound for the right-hand side of~\eqref{eqn:N1p1p-evolution}.
\begin{lemma}\label{lemma:very-simple}
For every tensor $S \in S^2(\fs \fo(n))$, we have the inequality
\begin{equation}\label{eqn:simple-inequality}
\begin{split}
    & \sum_{p=3}^n ( (S^2)_{1p1p} + (S^2)_{2p2p}) \\
    &\geq \sum_{p=3}^n \Bigl( S_{2p2p} - S_{1p1p} - \frac{\textup{Ric}(S)_{22} - \textup{Ric}(S)_{11}}{n-2} \Bigr)^2 + \frac{( \textup{Ric}(S)_{22} - \textup{Ric}(S)_{11})^2}{n-2}  \\
    & \quad  + 2 \sum_{p=3}^n ( (S_{121p})^2 + (S_{212p})^2 ) + \frac{2}{n-3} \sum_{p=3}^n \Biggl[ \Biggl( \sum_{ \substack{ q=3 \\ q \neq p}}^n S_{q 1 qp} \Biggr)^2 + \Biggl( \sum_{ \substack{ q=3 \\ q \neq p}}^n S_{q 2 qp} \Biggr)^2 \Biggr] \\
    & \quad + \frac{1}{n-2} \bigl( 2 \, S_{1212} - \textup{Ric}(S)_{11} - \textup{Ric}(S)_{22} \bigr)^2.
\end{split}
\end{equation}
\end{lemma}

\begin{proof}
Let us recall that 
\[
(S^2)_{ijk \ell} = \sum_{p,q=1}^n S_{ijpq} S_{k \ell pq}, \qquad (S^2)_{ijij} = \sum_{\alpha, \beta=1}^n (S_{ij\alpha \beta})^2 = 2 \sum_{1 \leq \alpha < \beta \leq n} (S_{ij \alpha \beta})^2.
\]
We therefore obtain, for every $3 \leq p \leq n$,
\[
(S^2)_{1p1p} = 2 \sum_{1 \leq \alpha < \beta \leq n} (S_{1 p \alpha \beta})^2 \geq 2 (S_{1p1p})^2 + 2 (S_{1q12})^2 + 2 \sum_{ \substack{q=3 \\ q \neq p} }^n (S_{1ppq})^2,
\]
and likewise for $(S^2)_{2p2p}$.
For $p \geq 3$, we can express $(S_{121p})^2 = (S_{1p12})^2$, and using the Cauchy-Schwarz inequality, we can bound
\[
\Biggl( \sum_{ \substack{ q=3 \\ q \neq p}}^n S_{q i qp} \Biggr)^2  \leq (n-3) \sum_{ \substack{ q = 3 \\ q \neq p } }^n (S_{iqqp})^2, \qquad \text{for } \; i \in \{ 1,2 \},
\]
Summing over $p \geq 3$ for $i \in \{ 1, 2\}$, we obtain
\begin{align*}
    \sum_{p=3}^n \bigl( (S^2)_{1p1p} + (S^2)_{2p2p} \bigr) &\geq 2 \sum_{p=3}^n \bigl( (S_{1p1p})^2 + (S_{2p2p})^2 \bigr) + 2 \sum_{p=3}^n [ (S_{121p})^2 + (S_{212p})^2] \\
    & \quad + \frac{2}{n-3} \sum_{p=3}^n \Biggl[ \Biggl( \sum_{ \substack{ q=3 \\ q \neq p}}^n S_{q 1 qp} \Biggr)^2 + \Biggl( \sum_{ \substack{ q=3 \\ q \neq p}}^n S_{q 2 qp} \Biggr)^2 \Biggr] .
\end{align*}
Moreover, for each $p \geq 3$, we can write
\[
2 (S_{1p1p})^2 + 2 (S_{2p2p})^2 = (S_{2p2p} - S_{1p1p})^2 + (S_{2p2p} + S_{1p1p})^2.
\]
Using the fact that $\sum_{p=3}^n ( S_{2p2p} - S_{1p1p}) = \textup{Ric}(S)_{22} - \textup{Ric}(S)_{11}$ and the Cauchy-Schwarz inequality, we can sum over all $3 \leq p \leq n$ to write
\begin{align*}
\sum_{p=3}^n \Bigl(S_{2p2p} - S_{1p1p} - \frac{\textup{Ric}(S)_{22} - \textup{Ric}(S)_{11}}{n-2} \Bigr)^2 &= \sum_{p=3}^n (S_{2p2p} - S_{1p1p})^2 - \frac{(\textup{Ric}(S)_{22} - \textup{Ric}(S)_{11})^2}{n-2}, \\
\sum_{p=3}^n (S_{1p1p} + S_{2p2p})^2 &\geq \frac{1}{n-2} \bigl( 2 \, S_{1212} - \textup{Ric}(S)_{11} - \textup{Ric}(S)_{22} \bigr)^2.
\end{align*}
In the last step, we used the fact that $\sum_{p=3}^n (S_{1p1p} + S_{2p2p}) = \textup{Ric}(S)_{11} + \textup{Ric}(S)_{22} - 2 \, S_{1212}$.
The claimed inequality~\eqref{eqn:simple-inequality} now follows from summing these terms.
\end{proof}

\begin{lemma}\label{lemma:bound-for-t-and-s}
For a tensor $S = U+N \in \cE_1(b)$ with $N \geq 0$, we form the quantity
\[
h := \textup{scal}(S)^{-1} \sum_{p=3}^n (N_{1p1p} + N_{2p2p}).
\]
Then, the variables
\[
\tau := \frac{\textup{Ric}(S)_{11} + \textup{Ric}(S)_{22}}{\textup{scal}(S)} - \frac{2}{n}, \qquad s := \frac{2 \, S_{1212}}{\textup{scal}(S)} - \Bigl( \tau + \frac{2}{n} \Bigr)
\]
satisfy the inequality
\[
        \frac{2(n-3) h - (n-1) (n-4) ( \tau + \frac{2}{n}) - 2}{n^2 - 7n + 14} \leq s \leq - h.
\]
\end{lemma}
\begin{proof}
Since $N \geq 0$ is non-negative as an operator on $\fs \fo(n)$, each diagonal entry of the form $N_{abab}$ is non-negative.
This implies that
\[
\sum_{3 \leq p \neq q \leq n} (N_{1212} + N_{1p1p} + N_{2q2q} + N_{pqpq}) \geq \sum_{3 \leq p \neq q \leq n} (N_{1p1p} + N_{2q2q} ) = (n-3) h \, \textup{scal}(S). 
\]
Moreover, because $U \in C_{\textup{PIC}}$, we can also bound
\begin{align*}
    (n-3) h \, \textup{scal}(S) &\leq \sum_{3 \leq p \neq q \leq n} (N_{1212} + N_{1p1p} + N_{2q2q} + N_{pqpq}) \\
    &\leq \sum_{3 \leq p \neq q \leq n} (S_{1212} + S_{1p1p} + S_{2q2q} + S_{pqpq}) \\
    &= (n-2) (n-3) S_{1212} + (n-4) \sum_{p=3}^n (S_{1p1p} + S_{2p2p}) + \sum_{p=3}^n \textup{Ric}(S)_{pp}.
\end{align*}
In terms of our notation, with $\sum_{p=3}^n \textup{Ric}(S)_{pp} = \sigma(1-\tau - \frac{2}{n})$, this implies the inequality
\[
(n-3) h \, \textup{scal}(S) \leq \tfrac{1}{2}(n^2 - 7n+14) s \, \textup{scal}(S) + \tfrac{1}{2}(n-1) (n-4) ( \tau + \tfrac{2}{n}) \, \textup{scal}(S) + \textup{scal}(S).
\]
Hence, rearranging terms proves the first bound above.

For the second bound, we use $U \in C_{\textup{PIC}}$ and average the PIC inequalities over four-frames of type $(e_1, e_2, e_p, \pm e_q)$, as in~\eqref{eqn:Tijkl}.
We then sum over $p<q$ to obtain
\[
- s \, \textup{scal}(S) = \sum_{p=3}^n (S_{1p1p} + S_{2p2p}) \geq \sum_{p=3}^n (N_{1p1p} + N_{2p2p}) = h \, \textup{scal}(S),
\]
hence $-s \geq h$.
This proves our assertion.
\end{proof}

\begin{proposition}\label{prop:evolution-second-barrier}
Consider a tensor $S = U+N \in \cE_1(b)$ with $\textup{scal}(S) > 0$.
Suppose that, for some orthonormal frame, we have
\[
\textup{Ric}(S)_{22} - \textup{Ric}(S)_{11} = \left[ \frac{\textup{scal}(S)^{\frac{1}{2}}}{(n-2) b} \sum_{p=3}^n (N_{1p1p} +N_{2p2p}) \right]^{\frac{1}{2}}.
\]
Then, the evolution~\eqref{eqn:d-dt-N-evolution} satisfies
\[
\frac{d}{dt} \left( \frac{\textup{scal}(S)^{\frac{1}{2}}}{ \sqrt{(n-2) b} } \Bigl( \sum_{p=3}^n (N_{1p1p} + N_{2p2p}) \Bigr)^{\frac{1}{2}}- ( \textup{Ric}(S)_{22} - \textup{Ric}(S)_{11}) \right) > 0.
\]
Consequently, the barrier~\eqref{eqn:S-define-ON} moves inward under the Hamilton ODE.
\end{proposition}
\begin{proof}
If $\textup{Ric}(S)_{22} - \textup{Ric}(S)_{11} = 0$, then the claim is automatic, by the strict positivity of the evolution of $N$ obtained in~\eqref{eqn:N1p1p-evolution}.
We henceforth assume $\textup{Ric}(S)_{22} - \textup{Ric}(S)_{11}>0$.
After rotating in $\textup{span} \{ e_1, e_2 \}$, assume that $\textup{Ric}(S)_{12} = 0$.
Next, we rotate the frame in $\textup{span} \{ e_3, \dots, e_n \}$ so that $\textup{Ric}(S)_{pq} = 0$ for $3 \leq p < q \leq n$.
For brevity, let us write $\sigma := \textup{scal}(S)$.
We introduce the variables
\begin{equation}\label{eqn:introduce-xi,s,t}
    \begin{split}
    \tau &:= \sigma^{-1}  (\textup{Ric}(S)_{11} + \textup{Ric}(S)_{22}) - \tfrac{2}{n}, \qquad s := 2 \sigma^{-1} S_{1212} - ( \tau + \tfrac{2}{n}), \\
    \xi &:= \left[ \frac{1}{(n-2) b \sigma} \sum_{p=3}^n (N_{1p1p} + N_{2p2p}) \right]^{\frac{1}{2}} = \sigma^{-1} ( \textup{Ric}(S)_{22} - \textup{Ric}(S)_{11}).        
    \end{split}
\end{equation}
The second equality for $\xi$ follows from our assumption.
Notice that
\begin{align}
    \sum_{p=3}^n \textup{Ric}(S)_{pp} &= \textup{scal}(S) - ( \textup{Ric}(S)_{11} + \textup{Ric}(S)_{22}) = \sigma \, ( 1 - \tau - \tfrac{2}{n}), \label{eqn:sum-lambda-p} \\
    (\tau + \tfrac{2}{n} ) \sigma &= \textup{Ric}(S)_{11} + \textup{Ric}(S)_{22}, \qquad \sum_{p=3}^n (S_{2p2p} - S_{1p1p}) = \textup{Ric}(S)_{22} - \textup{Ric}(S)_{11} = \sigma \xi. \label{eqn:S2p2p-S1p1p-trace}
\end{align}
Clearly $\xi \geq 0$.
We observe that the variables $s, \tau$ satisfy the bounds
\begin{equation}\label{eqn:key-s-tau-bounds}
\tau \geq - \frac{2}{n} \Bigl( 1 + \frac{2}{5} b \Bigr), \qquad s \geq \frac{2(n-2)(n-3) b \xi^2 - (n-1)(n-4) \tau - \frac{2(n-2)^2}{n}}{n^2 - 7n+14}.
\end{equation}
Indeed, the condition~\eqref{eqn:ric-11-first} implies that $\tau = - \frac{2}{n}+ \frac{\textup{Ric}(S)_{11} + \textup{Ric}(S)_{22}}{\textup{scal}(S)} \geq - \frac{2(1+ \frac{2}{5} b)}{n}$.
The bound for $s$ follows from the leftmost bound obtained in Lemma~\ref{lemma:bound-for-t-and-s}, here the term $h := \textup{scal}(S)^{-1} \sum_{p=3}^n (N_{1p1p} + N_{2p2p})$ in the notation of that lemma satisfies $h^2 = (n-2) b \xi^2$ due to the relation~\eqref{eqn:introduce-xi,s,t}.

Following the computation of~\cite{brendle-PIC}*{Proposition 3.12}, we can write
\[
\tfrac{d}{dt} ( \textup{Ric}(S)_{22} - \textup{Ric}(S)_{11}) = J_1 + J_2 + J_3 + J_4 + J_5 +  J_6
\]
where the terms $J_i$ are produced by the evolution equation~\eqref{eqn:evolve-ric-S}, and given by
\begin{align*}
J_1 &:= 2 \sum_{p=3}^n (S_{2p2p} - S_{1p1p}) \textup{Ric}(S)_{pp}, \\
J_2 &:= - 4 \sum_{p=3}^n S_{121p} \textup{Ric}(S)_{2p}, \\
J_3 &:= 4 \sum_{p=3}^n S_{212p} \textup{Ric}(S)_{1p}, \\
J_4 &:= - 2 S_{1212} ( \textup{Ric}(S)_{22} - \textup{Ric}(S)_{11}), \\
J_5 &:= 4b \bigl( ( \textup{Ric}(S)^2)_{11} - ( \textup{Ric}(S)^2)_{22} \bigr), \\
J_6 &:= \tfrac{4}{n} \bigl( 2b + (n-2) a \bigr) \, \textup{scal}(S) \, ( \textup{Ric}(S)_{22} - \textup{Ric}(S)_{11} ).
\end{align*}
We combine the expression~\eqref{eqn:sum-lambda-p} with Young's inequality to obtain
\begin{equation}\label{eqn:J1-estimate}
\begin{split}
    J_1 &= 2 \sum_{p=3}^n \Bigl( S_{2p2p} - S_{1p1p} - \frac{\sigma \xi }{n-2}\Bigr) \Bigl[ \textup{Ric}(S)_{pp} - \frac{\sigma \bigl(1 - \tau - \frac{2}{n} \bigr)}{n-2} \Bigr] + \frac{2\bigl( 1 - \tau - \frac{2}{n} \bigr)}{n-2} \sigma^2 \xi \\
    &\leq \frac{1}{2(n-2) b \xi} \sum_{p=3}^n \bigl(S_{2p2p} - S_{1p1p} - \frac{\sigma \xi}{n-2} \bigr)^2 \\
    & \quad+ 2(n-2) b \xi \sum_{p=3}^n \Bigl[ \textup{Ric}(S)_{pp} - \frac{\sigma \bigl( 1 - \tau - \frac{2}{n} \bigr)}{n-2} \Bigr]^2 + \frac{2 \bigl( 1 - \tau - \frac{2}{n} \bigr)}{n-2} \sigma^2 \xi.
\end{split}
\end{equation}
Next, we split the term $J_5$ into diagonal and off-diagonal components.
Since $\textup{Ric}(S)_{12} = 0$ and $\textup{Ric}(S)_{11} = \frac{1}{2} \sigma( \tau + \frac{2}{n} - \xi)$, while $\textup{Ric}(S)_{22} = \frac{1}{2} \sigma ( \tau + \frac{2}{n} + \xi)$, we obtain
\begin{align*}
    ( \textup{Ric}(S)^2)_{11} - ( \textup{Ric}(S)^2)_{22} &= ( \textup{Ric}(S)_{11})^2 - (\textup{Ric}(S)_{22})^2 + \sum_{p=3}^n (\textup{Ric}(S)_{1p})^2 - \sum_{p=3}^n (\textup{Ric}(S)_{2p})^2 \\
    &= - \sigma^2 \Bigl( \tau + \frac{2}{n} \Bigr) \xi + \sum_{p=3}^n (\textup{Ric}(S)_{1p})^2 - \sum_{p=3}^n (\textup{Ric}(S)_{2p})^2.
\end{align*}
This implies the bound
\begin{equation}\label{eqn:J5-formula}
    J_5 \leq - 4 b \sigma^2 \Bigl( \tau + \frac{2}{n} \Bigr) \xi + 4 b \sum_{p=3}^n (\textup{Ric}(S)_{1p})^2.
\end{equation}
Next, for $\epsilon = \pm 1$ and $(n-2)$-vectors $A,r$, we observe the identity
\begin{align*}
    0&\leq \frac{n-2}{n-3} \Bigl| A - \bigl( \frac{1}{n-2} + 2 (n-3) \xi \epsilon b \Bigr) r \Bigr|^2 \\
    &=|A|^2 + \frac{1}{n-3} |A-r|^2 + 4(n-2) b \xi \Bigl( \frac{4(n-2)b \xi \epsilon- 1}{4(n-2)^2 b \xi} + (n-3) b \xi \Bigr) |r|^2 - 4 (n-2) b \xi \epsilon \, \la A, r \rg .
\end{align*}
We can use this inequality to bound the terms $J_2, J_3$.
We first apply it to the $(n-2)$-vectors
\[
A = (S_{1213}, \dots, S_{121n}), \qquad r := ( \textup{Ric}(S)_{23}, \dots, \textup{Ric}(S)_{2n}), \qquad \epsilon = -1,
\]
for which $J_2 = 4 \epsilon \la A, r \rg$.
Our inequality implies the estimate
\begin{equation}\label{eqn:J2-inequality}
    \begin{split}
    J_2 &\leq \frac{1}{(n-2) b \xi} \sum_{p=3}^n (S_{121p})^2 + \frac{1}{(n-2)(n-3) b \xi} \sum_{p=3}^n ( S_{121p} - \textup{Ric}(S)_{2p})^2 \\
    & \quad + 4\left( (n-3) b \xi - \frac{1}{4(n-2)^2 b \xi} - \frac{1}{n-2} \right) \sum_{p=3}^n (\textup{Ric}(S)_{2p})^2.
    \end{split}
\end{equation}
Next, we apply this inequality to the $(n-2)$-vectors
\[
A = (S_{2123}, \dots, S_{212n}), \qquad r := ( \textup{Ric}(S)_{13}, \dots, \textup{Ric}(S)_{1n}), \qquad \epsilon = 1,
\]
for which $J_3 = 4 \epsilon \la A, r \rg$.
Our inequality implies the estimate
\begin{equation}\label{eqn:J3-inequality}
    \begin{split}
    J_ 3&\leq \frac{1}{(n-2) b \xi} \sum_{p=3}^n (S_{212p})^2 + \frac{1}{(n-2)(n-3) b \xi} \sum_{p=3}^n ( S_{212p} - \textup{Ric}(S)_{1p})^2 \\
    & \quad + 4\left( (n-3) b \xi - \frac{1}{4(n-2)^2 b \xi} + \frac{1}{n-2} \right) \sum_{p=3}^n (\textup{Ric}(S)_{1p})^2.
    \end{split}
\end{equation}
We recall that
\[
S_{121p} - \textup{Ric}(S)_{2p} = - \sum_{\substack{q = 3 \\ q \neq p}}^n S_{q2qp}, \qquad S_{212p} - \textup{Ric}(S)_{1p} = - \sum_{\substack{q = 3 \\ q \neq p}}^n S_{q1qp}.
\]
Thus, we can combine the bounds~\eqref{eqn:J1-estimate},
~\eqref{eqn:J2-inequality}, ~\eqref{eqn:J3-inequality}, and~\eqref{eqn:J5-formula} for $J_1,J_2,J_3$ and $J_5$ to obtain
\begin{equation}\label{eqn:J2J3J5}
    \begin{split}
    J_1 + J_2 + J_3 + J_5 &\leq  \frac{1}{2(n-2) b \xi} \sum_{p=3}^n \Bigl(S_{2p2p} - S_{1p1p} - \frac{\textup{Ric}(S)_{22} - \textup{Ric}(S)_{11}}{n-2} \Bigr)^2 \\
    & \quad+ 2(n-2) b \xi \sum_{p=3}^n \Bigl[ \textup{Ric}(S)_{pp} - \frac{\sigma \bigl( 1 - \tau - \frac{2}{n} \bigr)}{n-2} \Bigr]^2 + \frac{2 \bigl( 1 - \tau - \frac{2}{n} \bigr)}{n-2} \sigma^2 \xi \\
    & \quad + \frac{1}{(n-2) b\xi} \sum_{p=3}^n \bigl[ (S_{121p})^2 + (S_{212p})^2 \bigr] - 4 b \sigma^2 \Bigl( \tau + \frac{2}{n} \Bigr) \xi \\
    & \quad + \frac{1}{(n-2) (n-3) b\xi} \sum_{p=3}^n \Biggl[ \Biggl( \sum_{ \substack{ q=3 \\ q \neq p}}^n S_{q 1 qp} \Biggr)^2 + \Biggl( \sum_{ \substack{ q=3 \\ q \neq p}}^n S_{q 2 qp} \Biggr)^2 \Biggr] \\
    & \quad + 4\left( (n-3) b \xi - \frac{1}{4(n-2)^2 b \xi} + b + \frac{1}{n-2} \right) \sum_{p=3}^n \bigl[ (\textup{Ric}(S)_{1p})^2 + (\textup{Ric}(S)_{2p})^2 \bigr]      
    \end{split}
\end{equation}
We will absorb the first terms of the first, third, and fourth lines using the estimate~\eqref{eqn:simple-inequality} from Lemma~\ref{lemma:very-simple}, multiplied by $\frac{1}{2(n-2) b \xi}$.
In our present notation, the last term in the bound~\eqref{eqn:simple-inequality} equals
\[
\tfrac{1}{n-2} ( 2 \, S_{1212} - \textup{Ric}(S)_{11} - \textup{Ric}(S)_{22})^2 = \tfrac{1}{n-2} \sigma^2 s^2.
\]
The inequality also involves a term $\frac{(\textup{Ric}(S)_{22} - \textup{Ric}(S)_{11})^2}{n-2} = \frac{\sigma^2 \xi^2}{n-2}$.
Thus, the bound~\eqref{eqn:J2J3J5} becomes
\begin{align*}
    J_1 + J_2 + J_3 + J_5 &\leq \frac{1}{2(n-2) b \xi} \sum_{p=3}^n \bigl( (S^2)_{1p1p} + (S^2)_{2p2p} \bigr) + 2(n-2) b \xi \sum_{p=3}^n \Bigl[ \textup{Ric}(S)_{pp} - \frac{\sigma \bigl( 1 - \tau - \frac{2}{n} \bigr)}{n-2} \Bigr]^2 \\
    &\quad + 4 \left( (n-3) b \xi - \frac{1}{4(n-2)^2 b \xi} + b + \frac{1}{n-2} \right) \sum_{p=3}^n \Bigl( (\textup{Ric}(S)_{1p})^2 + (\textup{Ric}(S)_{2p})^2 \Bigr) \\
    & \quad + \Bigl[ \frac{2 \bigl( 1 - \tau - \frac{2}{n} \bigr)}{n-2} - 4 b \Bigl( \tau + \frac{2}{n} \Bigr) \Bigr] \sigma^2 \xi - \frac{\sigma^2 (\xi^2+s^2)}{2(n-2)^2 b \xi}.
\end{align*}
The remaining terms in the evolution of the barrier are $J_4 = - \sigma^2 (s + \tau + \frac{2}{n}) \xi$ and $J_6 = \frac{4}{n}(2b+(n-2) a) \sigma^2 \xi$.
Combining these estimates, we therefore obtain
\begin{equation}\label{eqn:evolution-upper-bound}
    \begin{split}
    &\frac{d}{dt} ( \textup{Ric}(S)_{22} - \textup{Ric}(S)_{11}) \\
    &= J_1 + J_2 + J_3 + J_4 + J_5 +  J_6 \\
    &\leq \frac{1}{2(n-2) b \xi} \sum_{p=3}^n \bigl( (S^2)_{1p1p} + (S^2)_{2p2p} \bigr) + 2(n-2) b \xi \sum_{p=3}^n \Bigl[ \textup{Ric}(S)_{pp} - \frac{\sigma \bigl( 1 - \tau - \frac{2}{n} \bigr)}{n-2} \Bigr]^2 \\
    &\quad + 4 \left( (n-3) b \xi - \frac{1}{4(n-2)^2 b \xi} + b + \frac{1}{n-2} \right) \sum_{p=3}^n \bigl[ (\textup{Ric}(S)_{1p})^2 + (\textup{Ric}(S)_{2p})^2 \bigr] \\
    & \quad + \Bigl[ \frac{2(1-\tau - \frac{2}{n})}{n-2} - s - (4b+1) \Bigl( \tau + \frac{2}{n} \Bigr) + \frac{4}{n} (2b+(n-2) a) \Bigr] \, \sigma^2 \xi - \frac{\sigma^2 (\xi^2+s^2)}{2 (n-2)^2 b \xi}.
    \end{split}
\end{equation}
On the other hand, since $\textup{Ric}(S)_{pq} =0$ for $3 \leq p < q \leq n$ and $\textup{Ric}(S)_{12} = 0$ by our choice of frame, the traceless Ricci tensor satisfies
\begin{equation}\label{eqn:ricci-block}
\begin{split}
    |\overset{\circ}{\textup{Ric}}(S)|^2 &= \sum_{p=3}^n \Bigl[ \textup{Ric}(S)_{pp} - \frac{\sigma (1-\tau - \frac{2}{n})}{n-2} \Bigr]^2 + \frac{1}{2} \xi^2 \sigma^2 + \frac{n}{2(n-2)} \tau^2 \sigma^2 \\
    & \quad + 2 \sum_{p=3}^n \bigl[ (\textup{Ric}(S)_{1p})^2 + (\textup{Ric}(S)_{2p})^2 \bigr].
\end{split}
\end{equation}
Next, we use $\xi^2 \, \textup{scal}(S) = \frac{1}{(n-2) b} \sum_{p=3}^n (N_{1p1p} + N_{2p2p})$ and the expression~\eqref{eqn:N1p1p-evolution} to compute
\begin{align*}
2 \xi \frac{d}{dt} ( \sigma \xi) &= \xi^2 \frac{d}{dt} \textup{scal}(S) + \frac{1}{(n-2) b} \frac{d}{dt} \sum_{p=3}^n (N_{1p1p} + N_{2p2p}) \\
&\geq \frac{1}{(n-2) b} \sum_{p=3}^n ( (S^2)_{1p1p} + (S^2)_{2p2p}) + \xi^2 \, [ P |\overset{\circ}{\textup{Ric}}(S)|^2 + ( \tfrac{1}{n} P + Q) \sigma^2] \\
& \quad + 2 \kappa b^{-1} \, |\overset{\circ}{\textup{Ric}}(S)|^2 + \tfrac{1}{6 n^2} b \, \textup{scal}(S)^2.
    \end{align*}
In the last step, we used the expressions~\eqref{eqn:ddt-scal} and~\eqref{eqn:N1p1p-evolution}, where $\kappa$ is defined in~\eqref{eqn:kappa}.
    We divide the preceding inequality by $2 \sigma^2$, subtract the bound~\eqref{eqn:evolution-upper-bound} multiplied by $\sigma^{-2} \xi$ from the resulting expression, and substitute $|\overset{\circ}{\textup{Ric}}(S)|^2$ from~\eqref{eqn:ricci-block}.
    Combining these expressions with~\eqref{eqn:evolution-upper-bound}, we arrive at
    \allowdisplaybreaks{
    \begin{align*}
    &    \xi \sigma^{-2} \frac{d}{dt} \left( \frac{\textup{scal}(S)^{\frac{1}{2}}}{ \sqrt{(n-2) b} } \Bigl( \sum_{p=3}^n (N_{1p1p} + N_{2p2p}) \Bigr)^{\frac{1}{2}}- ( \textup{Ric}(S)_{22} - \textup{Ric}(S)_{11}) \right) \\
    &\geq \frac{P \xi^4}{4} + \Bigl[ \frac{P}{2n} + \frac{Q}{2} + s + \frac{n}{4(n-2)} P \tau^2 + \Bigl( 4b + \frac{n}{n-2} \Bigr) \tau - \frac{4(n-2)}{n} a \Bigr] \, \xi^2  +\frac{\xi^2+s^2}{2(n-2)^2 b} \\
    & \quad + \frac{\kappa}{2b} \Bigl( \frac{n \tau^2}{n-2} + \xi^2 \Bigr) + \frac{b}{12n^2} + \Bigl[ \Bigl( \frac{P}{2} - 2(n-2) b \Bigr) \xi^2 + \frac{\kappa}{b} \Bigr] \, \sigma^{-2} \sum_{p=3}^n \Bigl[ \textup{Ric}(S)_{pp} - \frac{\sigma(1-\tau - \frac{2}{n})}{n-2} \Bigr]^2 \\
    & \quad + 4 \left[ \frac{\kappa}{2b} + \frac{1}{4(n-2)^2 b} - \Bigl( b + \frac{1}{n-2} \Bigr) \xi + \Bigl( \frac{P}{4} - (n-3) b \Bigr) \xi^2 \right]  \sigma^{-2} \sum_{p=3}^n [ (\textup{Ric}(S)_{1p})^2 + (\textup{Ric}(S)_{2p})^2 ].
\end{align*}}
We observe that the first line on the right-hand side of this expression is precisely the function $\cF(\xi, s, \tau)$ from Lemma~\hyperref[lemma:important-properties]{\ref{lemma:important-properties}~$(iii)$}, which satisfies $\cF(\xi,s,\tau) \geq 0$ for the variables $\xi,s,\tau$ in the range~\eqref{eqn:key-s-tau-bounds}.
Moreover, Lemma~\hyperref[lemma:important-properties]{\ref{lemma:important-properties}~$(i)$} shows that $\kappa>0$, and $\xi>0$ by our assumption.
Thus, we obtain the strictly positive combination
\[
\tfrac{1}{2} \kappa b^{-1} ( \tfrac{n}{n-2} \tau^2  +\xi^2) + \tfrac{1}{12 n^2} b > 0.
\]
Moreover, for $b \leq b_1(n) = \frac{1}{4(n-2)}$, the bound $P > \frac{7}{4}$ from Lemma~\hyperref[lemma:important-properties]{\ref{lemma:important-properties}~$(i)$} implies
\begin{align*}
    \tfrac{1}{2} P - 2(n-2) b &> \tfrac{1}{2} \cdot \tfrac{7}{4} - 2(n-2) \cdot \tfrac{1}{4(n-2)} > \tfrac{7}{8} - \tfrac{4}{8} > 0.
\end{align*}
Thus, the coefficient of the third line is positive.
Next, the factor appearing in the fourth line is a quadratic function of $\xi$, with leading coefficient
\[
\tfrac{1}{4} P - (n-3) b > \tfrac{1}{4} \cdot \tfrac{7}{4} - \tfrac{1}{4} \geq \tfrac{3}{16}.
\]
Therefore, for $n \geq 5$ and $b \leq b_1(n) = \frac{1}{4(n-2)}$, we can discard the term $\frac{\kappa}{2b}$ and observe that the full factor is bounded from below by
\begin{align*}
    \tfrac{1}{4(n-2)^2 b} - ( b + \tfrac{1}{n-2}) \xi + \tfrac{3}{16} \xi^2 &\geq \tfrac{1}{n-2} - \tfrac{5}{4(n-2)} \xi + \tfrac{3}{16} \xi^2 = \tfrac{3}{16} ( \xi - \tfrac{10}{3(n-2)} )^2 + \tfrac{12n-49}{12(n-2)^2} > 0.
\end{align*}
To summarize, we have proved that the right-hand side of this expression is strictly positive.
Consequently, the barrier~\eqref{eqn:S-define-ON} moves inward under the Hamilton ODE.
\end{proof}

\begin{proposition}\label{prop:transversally-invariant-family-1}
For every $n \geq 5$ and $0 < b \leq \frac{1}{4(n-2)}$, the family of cones $C_1(b)$ is transversally invariant under the Hamilton ODE.
\end{proposition}
\begin{proof}
It suffices to prove that the cone $\cE_1(b)$ is transversally invariant under the modified ODE $\frac{d}{dt} S = \cQ(S) + D_{a,b}(S)$.
By Lemma~\ref{lemma:important-properties}, we have $\kappa>0$, so the scalar coefficient in the numerator of the second line of~\eqref{eqn:D-a-b} is positive.
Moreover, Corollary~\ref{cor:interior-of-pic-cone} shows that
\[
D_{a,b}(S) - \tfrac{1}{2} \mu_b(S) \, \textup{scal}(S)^2 \, \textup{id} \owedge \textup{id} \in \textup{int}(C_{\textup{PIC}}).
\]
We decompose $S= U + N$ in $S^2(\fs \fo(n))$, for $U \in C_{\textup{PIC}}$ and $N \geq 0$ a non-negative operator, which evolve according to the system~\eqref{eqn:d-dt-N-evolution}.
Hence, the results of~\cite{brendle-PIC}*{Lemmas 3.6 - 3.8} show that for any $S \in \cE_1(b) \setminus \{ 0\}$ and $N \geq 0$ with $U = S-N \in C_{\textup{PIC}}$, we have $\frac{d}{dt} N>0$ and $\tfrac{d}{dt} U \in \textup{int} ( T_U C_{\textup{PIC}})$ lies in the interior of the tangent cone of $C_{\textup{PIC}}$ at $U$.
Notably, the conditions $U \in C_{\textup{PIC}}$ are $N \geq 0$ are preserved under this evolution.

The properties $\kappa, P, Q > 0$ are obtained in Lemma~\hyperref[lemma:important-properties]{\ref{lemma:important-properties}~$(i)$}.
    Moreover, Proposition~\ref{prop:evolution-second-barrier} shows that the second barrier~\eqref{eqn:S-define-ON} moves inward under the Hamilton ODE.
    Thus, it remains to show that the same property holds for the barrier~\eqref{eqn:ric-11-first}.
    Indeed, let us suppose that
    \[
    \textup{Ric}(S)_{11} + \textup{Ric}(S)_{22} + \tfrac{4}{5n} b \, \textup{scal}(S) = 0
    \]
    for an orthonormal pair $\{ e_1, e_ 2\}$.
    By a standard reduction, we may assume that $e_1, e_2$ are eigenvectors of $\textup{Ric}(S)$ corresponding to the two smallest eigenvalues, which extend to an eigenbasis $\{ e_1, \dots, e_n\}$ with $\lambda_i := \textup{Ric}(S)_{ii}$.
    We introduce the variables
    \[
    \sigma := \textup{scal}(S) > 0, \qquad x := \frac{\lambda_2 - \lambda_1}{2 \sigma}, \qquad z := \frac{\lambda_3 - \frac{1}{2}  (\lambda_1 + \lambda_2)}{\sigma} = \frac{\lambda_3}{\sigma} + \frac{2}{5n} b. 
    \]
    Let $N \in S^2 ( \fs \fo(n))$ be the non-negative operator from Definition~\ref{def:first-cone-family}, and denote
    \[
    h := \textup{scal}(S)^{-1} \sum_{p=3}^n (N_{1p1p} + N_{2p2p}).
    \]
    Using the condition~\eqref{eqn:S-define-ON} and the upper bound for $\lambda_3$ from Lemma~\ref{lemma:S-curvature-optimal-bound}, we obtain
    \[
        0 \leq x \leq z \leq \tfrac{1+\frac{2}{5} b}{n-2}, \qquad h \geq 4(n-2) b x^2.
    \]
    Moreover, we can use Lemma~\ref{lemma:S-curvature-optimal-bound} to obtain
    \begin{align}
        \sum_{p=3}^n (S_{1p1p} + S_{2p2p} ) \Bigl( \lambda_p - \frac{\lambda_1+\lambda_2}{2} \Bigr) &\geq \sigma^2 hz \geq 4(n-2) b \sigma^2 x^2 z, \label{eqn:first-barrier-first-term}  \\
        \sigma^{-2} |\overset{\circ}{\textup{Ric}}(S)|^2 &\geq 2 x^2 + \frac{2( 1 + \frac{2}{5} b)^2}{n(n-2)} + \frac{n-2}{n-3} 
        \Bigl( \frac{1 + \frac{2}{5} b}{n-2} - z \Bigr)^2. \label{eqn:first-barrier-second-term}
    \end{align}
    Using the computation of~\cite{brendle-PIC}*{Lemma 3.9} with $a = b + \frac{n-1}{2} b^2$ and $\gamma = \frac{2}{5} b$, we obtain
    \begin{align*}
        &\frac{1}{8 \sigma^2} \frac{d}{dt} \bigl( \textup{Ric}(S)_{11} + \textup{Ric}(S)_{22} + \frac{4}{5n} b \, \textup{scal}(S) \bigr) \\
        &= \frac{1}{4 \sigma^2} \sum_{p=3}^n (S_{1p1p} + S_{2p2p}) \Bigl( \lambda_p - \frac{\lambda_1+\lambda_2}{2} \Bigr) + \frac{1 + \frac{2}{5} b}{2 n^2} \Bigl[(n-1) b + \frac{2}{5} (1-2b) \Bigr] b - b x^2 \\
& \quad + \left( \Bigl( 1 + \frac{2}{5}  b \Bigr) \frac{n^2 b^2 - (n-1)^2 b^2 (1-2b)}{2 n (1 + (n-1) b)^2} + \frac{b (1-2b)}{5n} \right) \sigma^{-2} |\overset{\circ}{\textup{Ric}}(S)|^2
    \end{align*}
In the next step, we apply the inequalities~\eqref{eqn:first-barrier-first-term} and~\eqref{eqn:first-barrier-second-term} on the right-hand side.
This preserves the direction of inequalities as all coefficients are non-negative.
The resulting expression is computed as
\begin{align*}
        & (n-2) b x^2 z + (K-b) x^2 + K \Bigl[ \frac{(1 + \frac{2}{5} b)^2}{n(n-2)} + \frac{n-2}{2(n-3)} \Bigl( \frac{1 + \frac{2}{5} b}{n-2} - z \Bigr)^2 \Bigr]  + \frac{b(1 + \frac{2}{5} b) ( 2 +(5n-9) b)}{10n^2}.
\end{align*}    
Here, $K$ is the term introduced in Lemma~\hyperref[lemma:important-properties]{\ref{lemma:important-properties}~$(ii)$}.

For fixed $z$, the right-hand side of the above expression is affine in $x^2$.
Since $0 \leq x \leq z$, it suffices to check $x=0$ and $x=z$; at $x=0$, Lemma~\hyperref[lemma:important-properties]{\ref{lemma:important-properties}~$(ii)$} gives $K>0$, so the expression is strictly positive.
At $x=z$, the right-hand side is precisely the cubic function considered in Lemma~\hyperref[lemma:important-properties]{\ref{lemma:important-properties}~$(ii)$}, so it is again strictly positive.
Therefore,
\[
\frac{d}{dt} \Bigl( \textup{Ric}(S)_{11} + \textup{Ric}(S)_{22} + \frac{4b}{5n} \textup{scal}(S) \Bigr) > 0.
\]
Hence, the first defining barrier~\eqref{eqn:ric-11-first} also moves inward. 
We conclude that $C_1(b)$ is transversally invariant under Hamilton's ODE.
\end{proof}

\section{A one-parameter family of invariant cones \texorpdfstring{$C_2(b)$}{C2(b)} that pinch toward PIC1}\label{section:second-family}

In this section, we construct another one-parameter family of cones $C_2(b)$ that are invariant under the Hamilton ODE. 
These cones pinch toward the PIC1 cone and can be joined continuously to the family $C_1(b)$ constructed in Section~\ref{section:first-family}. 
More precisely, we start from the cone $C_1(b_1) \cap C_{\textup{PIC1}}$  and deform it outward in such a way that the deformed cones are all preserved under the ODE. 
To that end, we consider curvature conditions that interpolate between the conditions PIC and PIC1.

We introduce the following definition.
\begin{definition}\label{def:a=b}
For $n \geq 5$, we define the parameters $b_2(n), A_n, B_n$ by
\begin{equation}\label{eqn:b2(n)-parameters}
b_2(n) := \begin{cases}
    \frac{1}{2(n-1)^2}, & n \in \{ 5, 6 \}, \\
    \frac{1}{2(n-2)^2}, & n \geq 7,
\end{cases} \qquad A_n := \begin{cases}
    0, & n \in \{ 5, 6\}, \\
    \frac{n-2}{2}, & n \geq 7,
\end{cases} \qquad B_n := \begin{cases}
    20, & n \in \{ 5, 6 \}, \\
    4, & n \geq 7.
\end{cases}
\end{equation}
Moreover, let $a(b) := b + A_n b^2$.

Given an algebraic curvature operator $T \in S^2_B(\fs \fo(n))$, an orthonormal four-frame $F = \{ e_1, e_2, e_3, e_4 \}$, and a $\lambda \in [0,1]$, we define the operators $Z_{\lambda}(S;F)$ and $\tilde{Z}_{\lambda}(S;F)$ by
\begin{align}
    Z_{\lambda}(S; F) &:= S_{1313} + \lambda^2 S_{1414} + S_{2323} + \lambda^2 \, S_{2424} - 2 \lambda \, S_{1234}, \label{eqn:Y-lambda(S)} \\
    \tilde{Z}_{\lambda}(S;F) &:= Z_{\lambda}(S;F) + (1-\lambda^2) \sqrt{2b} \, [ \textup{Ric}(S)_{11} + \textup{Ric}(S)_{22} + 2\, B_n b \, \textup{scal}(S) ] . \label{eqn:Z-tensor}
\end{align}
For $b \in (0,b_2(n)]$, we denote by $\cE_2(b)$ the set of algebraic curvature operators $S \in S^2_B( \fs \fo(n))$ such that $\ell_{a,b}(S) \in C_1(b_1)$ and $\tilde{Z}_{\lambda}(S; F) \geq 0$ for every orthonormal four-frame $F = \{ e_1, e_2, e_3, e_4 \}$ and every $\lambda \in [0,1]$.
Finally, we define $C_2(b) := \ell_{a,b}( \cE_2(b))$.
\end{definition}

We notice that every $S \in \cE_2(b) \setminus \{ 0 \}$ satisfies $\textup{scal}(S) > 0$.
We retain the parameters $a_1, b_1$ coming from the first family as follows:
\begin{equation}\label{eqn:a1-omega1-b1}
    b_1 := \tfrac{1}{4(n-2)}, \qquad a_1 := b_1 + \tfrac{n-1}{2} b_1^2 = \tfrac{9n-17}{32(n-2)^2}.
\end{equation}
The operator $T = \ell_{a_1, b_1}^{-1} \ell_{a,b}(S) \in \cE_1(b_1)$ is related to the operator $S$ by 
\begin{equation}\label{eqn:S-T-p}
    \begin{split}
        S &= T + p \, \textup{Ric}(T) \owedge \textup{id} + \tfrac{1}{n} ( q -p) \, \textup{scal}(T) \, \textup{id} \owedge \textup{id} \, , \\
        p &= \frac{b_1 - b}{1+(n-2) b}, \qquad q = \frac{a_1 - a}{1+2(n-1) a},
    \end{split}
\end{equation}
and has $T \in C_{\textup{PIC}}$ and $\textup{Ric}(T)_{11} + \textup{Ric}(T)_{22} + \frac{4}{5 n} b_1 \textup{scal}(T) \geq 0$.
Therefore, taking trace-free Ricci parts and scalar parts gives
    \begin{equation}\label{eqn:ric-scal-ST-equalities}
    \begin{split}
        (1 + (n-2)b) \overset{\circ}{\textup{Ric}}(S) = (1 + (n-2) b_1) \overset{\circ}{\textup{Ric}}(T) \quad & \implies \quad \overset{\circ}{\textup{Ric}}(S) = \frac{1 + (n-2) b_1}{1 + (n-2) b} \overset{\circ}{\textup{Ric}}(T), \\
        (1 + 2(n-1) a) \textup{scal}(S) = (1 + 2(n-1) a_1) \textup{scal}(T) \quad & \implies \quad \textup{scal}(S) = \frac{1 + 2(n-1) a_1}{1 + 2(n-1) a} \textup{scal}(T).
        \end{split}
    \end{equation}
We now collect some properties resulting from these definitions.
\begin{lemma}\label{lemma:properties-II}
We define $a = b + A_n b^2$ and $b_2(n)$ as above.
For every $n \geq 5$ and every $b \in (0,b_2(n)]$, the following properties are satisfied.
\begin{enumerate}[(i)]
    \item The function $\kappa$, defined in terms of $a,b$ in~\eqref{eqn:kappa}, is strictly positive.
    \item For every $S \in \cE_2(b)$, $\textup{Ric}(S)$ is weakly two-positive.
    \item For $n=5$, we have $8 q - 3p > 0$.
    For $n \geq 6$, we have $q > p$.
    Also, $q, p > 0$.
    \item The function $P$ defined in~\eqref{eqn:P(ab)Q(ab)} satisfies $P \in ( \frac{19}{10}, 2)$ for $b \in (0, b_2(n)]$.
    \item The function $Q$ defined in~\eqref{eqn:P(ab)Q(ab)} satisfies the lower bound
    \[
    Q(b) \geq (4 - \tfrac{1}{n-1}) b \quad \text{for } \; n \in \{ 5,6\}, \qquad Q(b) \geq 4b \quad \text{for } \; n \geq 7.
    \]
\end{enumerate}
\end{lemma}
\begin{proof}
Using the definition~\eqref{eqn:kappa}, for $n \in \{5,6\}$ we find $\kappa = \frac{1-2b}{1+2(n-1) b}  b^2 > 0$ due to $b \leq \frac{1}{32}$.
Moreover, for $n \geq 7$, we compute
\[
n b^{-2} (1+2(n-1) a) \, \kappa = n(1-2b) - (n-2)(1-2b+nb^2) = 2 - 4b - n(n-2) b^2
\]
For $b \leq \frac{1}{2(n-2)^2}$, the right-hand side of this expression is at least $\tfrac{8n^3 - 48n^2 + 87n - 48}{4(n-2)^3}$, whose denominator is positive for $n \geq 7$.
This proves $(i)$.

Next, we turn to the claim $(ii)$.
We define the function $\zeta(b)$ by
\begin{align}
    \zeta(b) &:= \frac{1 + 2(n-1) a}{1 + 2(n-1) a_1} \, \frac{1 + (n-2) b_1}{1 + (n-2) b} \, \Bigl( 1 + \frac{2}{5} b_1 \Bigr) = \Bigl( 1 + \frac{2}{5} b_1 \Bigr) \frac{1 + (n-2) p}{1+2(n-1) q}. \label{eqn:zeta-a-b}
\end{align}
We claim that $\zeta(b) < 1$ for $b \leq b_2(n)$.
We will first prove that $\zeta(b)$ is strictly increasing.
Using~\eqref{eqn:a1-omega1-b1} and the expression $a = b + A_n b^2$, we compute 
\begin{equation}\label{eqn:Xi-p-zeta}
\begin{aligned}
    n&=5: \quad &p &= \tfrac{1-12b}{12(1+3b)}, \quad &q &= \tfrac{7-72b}{72(1+8b)}, \quad &\zeta &= \tfrac{93(1+8b)}{128(1+3b)}. \\
    n&=6: \quad &p &= \tfrac{1-16b}{16(1+4b)}, \quad & q &= \tfrac{37-512b}{512(1+10b)}, \quad & \zeta &= \tfrac{328(1+10b)}{441(1+4b)}. 
\end{aligned}
\end{equation}
Hence, the fact that $\zeta$ is strictly increasing and satisfies $\zeta \in (0,1)$ is clear from the bounds $b \leq b_2(n) = \frac{1}{2(n-1)^2}$ for $n \in \{ 5, 6\}$.
For $n \geq 7$, we first obtain
\[
\zeta(b) = \tfrac{2(n-2)(10n-19)}{(5n-9)^2 (1 + (n-2) b) } [1 + 2(n-1) b + (n-2)(n-3) b^2]
\]
We observe the simplification
\begin{equation}\label{eqn:zeta(b)-inequality-simple}
\begin{split}
\tfrac{(5n-9)^2}{2(n-2)(10n-19)} \zeta(b)
&= \tfrac{1 + 2(n-1) b + (n^2 - 3n+2) b^2}{1 + (n-2) b} = 1 + \bigl( n-1 + \tfrac{1}{1+(n-2) b} \bigr) b, \\
\tfrac{(5n-9)^2}{2(n-2)(10n-19)} \zeta'(b) &= n-1+ \tfrac{1}{(1 + (n-2) b)^2}  > 0.
\end{split}
\end{equation}
Therefore, $\zeta$ is strictly increasing in $b$, and it suffices to show that $\zeta( \frac{1}{2(n-2)^2}) < 1$.
We compute
\[
1 - \zeta ( \tfrac{1}{2(n-2)^2}) = \tfrac{10(n-2)^4 + (n-2)^3 - 42(n-2)^2 - 14(n-2) - 1}{(n-2)(2n-3) (5n-9)^2}
\]
where the numerator satisfies, for $n \geq 7$, 
\[
(n-2)^2 [ 10(n-2)^2 + (n-2) - 50] > 0, \qquad 8(n-2)^2 - 14(n-2) - 1 > 0.
\]
Thus, $\zeta(b)<1$ for $b \in (0,b_2(n)]$ and $n \geq 5$.
Finally, the claimed weak two-positivity follows from the expression~\eqref{eqn:S-T-p} as in~\cite{brendle-PIC}*{Lemma 4.5}. 
Applying the lower bound therein, we obtain
\[
\textup{Ric}(S)_{11} + \textup{Ric}(S)_{22} \geq \tfrac{2(1-\zeta)}{n} \textup{scal}(S) \geq 0
\]
using the bound $\zeta<1$ and $\textup{scal}(S)>0$ for $S \in \cE_2(b) \setminus \{ 0 \}$. 
This proves the property $(ii)$.

We now study $(iii)$.
For $n=5$, the computations~\eqref{eqn:Xi-p-zeta} imply
\[
36 (1+8b)(1+3b) ( 8 q - 3p) = 19 - 168 b > 0, \qquad \text{for } \; b \leq \tfrac{1}{32}.
\]
For $n=6$, we find
\[
512 (1+4b)(1+10b) (q - p) = 3072 b^2 - 172 b + 5
\]
where the quadratic has negative discriminant.
Thus, $q > p$.
For $n \geq 7$, we compute
\begin{align*}
    &32 (n-2)^2 (1 + (n-2) b) \, (1 + 2(n-1) a)  \, (q - p) \\
    &= 16 b^3(n-2)^4 + 32b^3(n-2)^3 + 8b^2(n-2)^3 + 56b^2(n-2)^2 - 7b(n-2)^2 - 15b(n-2) + (n-1) \\
    &\geq - 7b(n-2)^2 - 15b (n-2)+ (n-1)
\end{align*}
by discarding all the non-constant positive terms.
Finally, since $b \leq \frac{1}{2(n-2)^2}$, the last expression is bounded from below by $n-1-\frac{7}{2} - \frac{15}{2(n-2)} > 0$ for $n \geq 7$.
Therefore, $q>p$ in all cases.

We next prove the property $(iv)$.
Using the expression~\eqref{eqn:P(ab)Q(ab)}, for $n \in \{ 5, 6\}$ we find
\[
P(b) = 2 \tfrac{(1 + (n-2) b)^2}{1+2(n-1) b}, \qquad P'(b) = 4 \tfrac{1+(n-2) b}{(1 + 2(n-1) b) ^2} [ -1 + (n-2)(n-3) b].
\]
Because $b \leq \frac{1}{2(n-1)^2}$, we have $(n-1)(n-2) b \leq \frac{n-2}{2(n-1)} < 1$, so $P$ is strictly decreasing in $b$ and attains its minimum at $b = b_2(n)$.
We thus compute
\[
n=5: \quad P \geq 2 \tfrac{(1 + \frac{3}{32})^2}{1 + \frac{8}{32}} = \tfrac{245}{128}, \qquad n=6: \quad P \geq 2 \tfrac{(1 + \frac{4}{50})^2}{1 + \frac{10}{50}} = \tfrac{243}{125} 
\]
so $P > \frac{19}{10}$ for $n \in \{5,6\}$.
Moreover, $\frac{(1+(n-2) b)^2}{1+2(n-1) b} < 1$ for $b \in (0,b_2(n)]$, so $P<2$ for $n \in \{ 5, 6\}$.

For $n \geq 7$, we use $a = b + \frac{n-2}{2} b^2$ to obtain
\[
P(b) = 2 \tfrac{(1 + (n-2) b)^2}{1 + 2(n-1)b + (n-1)(n-2) b^2}, \qquad P'(b) = - 4 \tfrac{1+(n-2) b}{(1 + 2(n-1) b + (n-1)(n-2) b^2)^2} < 0.
\]
Thus, $P$ is again strictly decreasing.
At $b_2 = \frac{1}{2(n-2)^2}$, we compute the expressions
\begin{align*}
P( \tfrac{1}{2(n-2)^2}) &= \tfrac{2(n-2) ( 2(n-2) + 1)^2}{4 (n-2)^3 + 4(n-2)^2 + 5(n-2) + 1}, \\
\tfrac{\partial}{\partial n} P ( \tfrac{1}{2(n-2)^2}) &= \tfrac{2 ( 2(n-2)+1) ( 16 (n-2)^2 + 6(n-2) + 1) }{ [ 4(n-2)^3 + 4(n-2)^2 + 5(n-2) + 1]^2} > 0, 
\end{align*}
so $P ( \frac{1}{2(n-2)^2})$ is strictly increasing in $n$.
For $n=7$, we compute $P \geq \frac{605}{313}> \frac{19}{10}$.
To summarize, we have proved that $P > \frac{19}{10}$ for all $b \in (0,b_2(n)]$ and all $n \geq 7$.
Finally, the bound 
\[
(1+(n-2) b)^2 < 1+2(n-1) b + (n-1)(n-2) b^2
\]
again implies that $P<2$.
This proves the property $(iv)$ for all $n \geq 5$.

We finally prove the lower bound for $Q(b)$.
For $n \in \{ 5,6\}$, since $a=b$, we compute $Q = 2b \frac{2+(3n-4) b}{1+2(n-1) b}$.
Thus, $\frac{Q}{b}$ is strictly decreasing in $b$, and computing at the endpoints $b_2(5)= \frac{1}{32}$ and $b_2(6) = \frac{1}{50}$ yields
\begin{align*}
Q &\geq \tfrac{15}{4} b \quad \text{for } \; n=5 \quad \text{and } \; b \in (0, \tfrac{1}{32}], \\ 
Q &\geq \tfrac{19}{5} b \quad \text{for } \; n=6 \quad \text{and } \; b \in (0, \tfrac{1}{50}].
\end{align*}
For $n \geq 7$, using $a = b + \frac{n-2}{2} b^2$ gives
\[
Q - 4b = 2b^2 \frac{(n-2)^2 (1+(n-1) b)^2 - 2n}{n( 1 + 2(n-1) b + (n-1)(n-2) b^2)} > 0.
\]
Thus, $Q>4b$ completes the proof of our assertion.
\end{proof}

We now introduce a crucial novelty compared to the construction of~\cite{brendle-PIC} and later modifications.
Those constructions prove the transversality of the cone $C_2(b)$ by showing that $\frac{d}{dt} \bigr|_{t=0} \tilde{Z}_{\lambda}(S)>0$ at every face where $\tilde{Z}_{\lambda}(S) = 0$.
In lower dimensions, this approach encounters a key difficulty: the quantity $\frac{d}{dt} \bigr|_{t=0} \tilde{Z}_{\lambda}(S)$ has a large negative contribution as $\lambda \uparrow 1$, and arranging parameters to make it uniformly positive becomes challenging.
To overcome this issue, we leverage the PIC inequalities for $S$ and the corresponding tensor $T = \ell_{a_1, b_1}^{-1} \ell_{a,b}(S)$ to first rule out the vanishing $\tilde{Z}_{\lambda}(S) = 0$ for any $S \in \cE_2(b) \setminus \{ 0 \}$, whenever $\lambda$ is quantitatively close to $1$.

To quantify this property, we define functions $\eta_n(b)$, for $n \geq 5$, as follows:
\begin{equation}\label{eqn:eta-n(b)}
    \eta_5 := \frac{248 p + 360 q}{383}, \qquad \eta_6 := \frac{82 p + 160 q}{121}, \qquad \eta_n := \frac{2(n-4) p + 8 q}{n}, \quad \text{for } \; n \geq 7.
\end{equation}

\begin{proposition}\label{prop:why-the-long-face}
Let $n \geq 5$ and let $S \in S^2_B(\fs \fo(n)) \setminus \{ 0 \}$ be an algebraic curvature operator such that $T := \ell^{-1}_{a_1, b_1} \ell_{a,b}(S) \in \cE_1(b_1)$.
Then $S$ is strictly PIC.

Moreover, suppose that, for some orthonormal four-frame and some $\lambda \in [0,1)$, the function $\rho \mapsto \tilde{Z}_{\rho}(S)$ attains a minimum on $[0,1]$ at $\lambda$, and that $\tilde{Z}_{\lambda}(S) \leq 0$.
Then,
\begin{equation}\label{eqn:lambda-inequality}
    \lambda \leq ( \textstyle{\sqrt{1+\eta_n}} - \textstyle{\sqrt{\eta_n}} )^2. 
\end{equation}
If $\tilde{Z}_{\lambda}(S) < 0$, the inequality is strict.
\end{proposition}
\begin{proof}
For brevity, let $\tau := \textup{scal}(T) > 0$ and $v := \tau^{-1} \textup{tr}_F \textup{Ric}(T)$ for a four-frame $F$.
Note that
\begin{equation}\label{eqn:v-intervals}
v \in [ \tfrac{1}{2}, \tfrac{20 + \frac{16}{5} b_1}{15}] \quad \text{for } \; n=5, \qquad v \in [ 0, 1 + \tfrac{2}{15} b_1] \quad \text{for } \; n= 6,
\end{equation}
and $v \in [0,1]$ for $n \geq 7$.
The bounds $v \geq \frac{1}{2}$ for $n \geq 5$, as well as $v \leq \frac{20 + \frac{16}{5} b_1}{15}$ for $n = 5$ and $v \leq 1 + \frac{2(n-4)}{5n} b_1$ for $n=6$, follow directly from Lemmas~\ref{lemma:consequences-of-PIC} and~\ref{lemma:pinching-from-PIC}.
To see the last bound $v \leq 1$ for $n  \geq 7$, notice that $\dim F^{\perp} \geq 3$ for $n \geq 7$.
Therefore, Lemma~\ref{lemma:consequences-of-PIC} implies that $\textup{tr}_{F^{\perp}} \textup{Ric}(T) = \textup{scal}(T) - \textup{tr}_F \textup{Ric}(T) \geq 0$ due to $T \in C_{\textup{PIC}}$.

Now, we observe that $Z_1(S;F) > 0$ for every frame $F$.
Indeed, Lemma~\ref{lemma:properties-II} shows that $q > \frac{3}{8} p$ for $n = 5$ and $q>p$ for $n \geq 6$.
The expression~\eqref{eqn:S-T-p} implies that
\begin{equation}\label{eqn:A1(SF)-inequality}
Z_1(S;F) = Z_1(T;F) + 2 \, [ p v + \tfrac{4}{n}(q-p) ] \, \tau \, .
\end{equation}
For $n=5$, the bracketed term is at least $\frac{1}{2} p+\frac{4}{5}(q-p) = \frac{8 q- 3p}{10}>0$.
For $n \geq 6$, it is at least $\frac{4}{n}(q-p)>0$.
Thus, $Z_1(S;F) > Z_1(T;F) \geq 0$ in all cases, by $T \in C_{\textup{PIC}}$, so $S$ is strictly PIC. 

Next, we consider a four-frame $F = \{e_1, e_2, e_3, e_4 \}$ and a $\lambda \in [0,1)$ where $\tilde{Z}_{\rho}(S)$ attains its minimum.
If $\lambda=0$, the conclusion is immediate, so we can assume that $\lambda \in (0,1)$ is an interior minimum of the quadratic $\rho \mapsto \tilde{Z}_{\rho}(S)$.
We define the expression
\[
D : = S_{1414} + S_{2424} - \sqrt{2b} \, \bigl[  \textup{Ric}(S)_{11} + \textup{Ric}(S)_{22} + 2 B_n b \,  \textup{scal}(S) \bigr]
\]
Because $\lambda \in (0,1)$ is an interior minimum of the map $\rho \mapsto \tilde{Z}_{\rho}(S)$, evaluating the first and second derivative produces
\begin{equation}\label{eqn:take-derivatives-Z-rho}
    0= \tfrac{1}{2} \tfrac{d}{d \rho} \big\rvert_{\rho = \lambda} \tilde{Z}_{\rho}(S) = \lambda D - S_{1234} , \qquad
    0 \leq \tfrac{1}{2} \tfrac{d^2}{d \rho^2} \big\rvert_{\rho = \lambda} \tilde{Z}_{\rho}(S).
\end{equation}
This property allows us to write $\tilde{Z}_{\rho}(S) = \tilde{Z}_{\lambda}(S) + D (\rho - \lambda)^2$.
Recall that $Z_1(S;F) = \tilde{Z}_{\lambda}(S) + (1-\lambda)^2 D$, so $Z_1(S;F) > 0$ and $\tilde{Z}_{\lambda}(S) \leq 0$ yield $D>0$.
Combining this relation with~\eqref{eqn:A1(SF)-inequality} gives
\begin{equation}\label{eqn:D-next-inequality}
    (1-\lambda)^2 D \geq 2 \, [ pv + \tfrac{4}{n} (q- p)] \, \tau - \tilde{Z}_{\lambda}(S).
\end{equation}
Next, we obtain an upper bound for $D$.
We apply the identity~\eqref{eqn:A1(SF)-inequality} to the two other ordered four-planes $\tilde{F}_1 = (e_1, e_3, e_2, e_4)$ and $\tilde{F}_2 = (e_1, e_4, e_2, - e_3)$ obtained from the vectors $e_1, e_2, e_3, e_4$, which span the same four-plane and hence have the same value of $v$.
Using $Z_1(T; \tilde{F}_i) \geq 0$, we obtain
\begin{align*}
    S_{1212} + S_{1414} + S_{2323} + S_{3434} - 2 \, S_{1324} &\geq 2 \, [ pv + \tfrac{4}{n} ( q-p)] \, \tau, \\
    S_{1212} + S_{1313} + S_{2424} + S_{3434} + 2 \, S_{1423} &\geq 2 \, [ pv + \tfrac{4}{n} (q-p)] \, \tau, \\
    S_{1313} + S_{1414} + S_{2323} + S_{2424} + 2 S_{1234} &= \tilde{Z}_{\lambda}(S) + (1+\lambda)^2 D.
\end{align*}
We add these relations, and then use the first Bianchi identity $S_{1324} - S_{1423} = S_{1234} = \lambda D$ together with the property~\eqref{eqn:take-derivatives-Z-rho}.
Combining these properties produces
\begin{equation}\label{eqn:add-first-step}
    2 \sum_{1 \leq i < j \leq 4} S_{ijij} \geq \tilde{Z}_{\lambda}(S) +  (1+\lambda)^2 D + 4 \, [ pv + \tfrac{4}{n}(q-p)] \, \tau.
\end{equation}
On the other hand, the transformation~\eqref{eqn:S-T-p} from $T$ to $S$ implies that
\[
2 \sum_{1 \leq i < j \leq 4} S_{ijij} = 2 \sum_{1 \leq i < j \leq 4} T_{ijij} + 6 \, [ pv + \tfrac{4}{n}(q-p)] \, \tau.
\]
Combining these computations, we deduce that
\begin{equation}\label{eqn:1+lambda2-D-Tijij}
    (1+\lambda)^2 D \leq 2 \sum_{1 \leq i < j \leq 4} T_{ijij} + 2 \, [ pv + \tfrac{4}{n}(q-p)] \, \tau - \tilde{Z}_{\lambda}(S).
\end{equation}
Finally, we obtain an upper bound for the expression $2 \sum_{1 \leq i < j \leq 4} T_{ijij}$.
First, if $n =5$, then
\[
\tau v = 2 \sum_{1 \leq i < j \leq 4} T_{ijij} + \sum_{i=1}^4 T_{i5i5}, \qquad \tau = 2 \sum_{1 \leq i < j \leq 4} T_{ijij} + 2 \sum_{i=1}^4 T_{i5i5}
\]
by the identities for the Ricci and scalar curvatures.
Subtracting terms, we obtain $\sum_{1 \leq i < j \leq 4} T_{ijij} = \frac{2v-1}{2} \tau$.
On the other hand, for $n \geq 6$, Lemma~\ref{lemma:consequences-of-PIC} implies $\sum_{i=1}^4 \sum_{\alpha=5}^n T_{i \alpha i \alpha} \geq 0$: indeed, we apply the inequality~\eqref{eqn:Tijkl} on a frame $(e_i, e_j, e_{\alpha}, e_{\beta})$ to obtain
\[
T_{i \alpha i \alpha} + T_{i \beta i \beta} + T_{j \alpha j \alpha} + T_{j \beta j \beta} \geq 0.
\]
Summing over $i<j$ and $a<b$ gives a positive multiple of $\sum_{i=1}^4 \sum_{\alpha=5}^n T_{i \alpha i \alpha} \geq 0$, hence this sum is also non-negative.
As a result, we arrive at
\[
\tau v = 2 \sum_{1 \leq i < j \leq 4} T_{ijij} + \sum_{i=1}^4 \sum_{\alpha=5}^n T_{i \alpha i \alpha} \geq 2 \sum_{1 \leq i < j \leq 4} T_{ijij}.
\]
Therefore, $2 \sum_{1 \leq i < j \leq 4} T_{ijij} \leq \tau v$, and using this bound in~\eqref{eqn:1+lambda2-D-Tijij} shows that 
\begin{equation}\label{eqn:1+lambda-pic-inequalities}
    \begin{aligned}
        (1+\lambda)^2 D &\leq (2v-1) \tau + 2 \, [ pv + \tfrac{4}{5} (q-p)] \, \tau  - \tilde{Z}_{\lambda}(S), \qquad & & \textup{for } \; n = 5, \\
        (1+\lambda)^2 D &\leq v \tau + 2 [ pv + \tfrac{4}{n}(q-p)] \, \tau - \tilde{Z}_{\lambda}(S), \qquad & & \textup{for } \; n \geq 6. 
    \end{aligned}
\end{equation}
We introduce the rational functions $G_n$, for $n \geq 5$, given by
\begin{align*}
G_5(v) &:= \tfrac{2v-1}{pv + \frac{4}{5} (q-p)}, \qquad & G_n(v) &:= \tfrac{v}{pv + \frac{4}{n} (q-p)}, \quad & n &\geq6. \\
G'_5(v) &= \tfrac{5 \, ( 8 q - 3p)}{[ 5 pv + 4 (q-p)]^2 } > 0, \qquad & G'_n(v) &= \tfrac{4 n \, (q-p)}{[ n pv+ 4(q- p)]^2} >0, \quad & n &\geq6.
\end{align*}
Recall from Lemma~\hyperref[lemma:properties-II]{\ref{lemma:properties-II}~$(iii)$} that $q > \frac{3}{8}p$ for $n=5$ and $q>p$ for $n \geq 6$.
Hence, this computation shows that $G_n$ is strictly increasing for all $n \geq 5$.
Moreover, $\tilde{Z}_{\lambda}(S) \leq 0$ by our assumption, so we can combine the inequalities~\eqref{eqn:D-next-inequality} and~\eqref{eqn:1+lambda-pic-inequalities} to obtain
\begin{align*}
    \Bigl( \frac{1+\lambda}{1-\lambda} \Bigr)^2 &\leq 1 + \frac{2v-1}{2 [ pv + \frac{4}{5} (q-p)] - \tau^{-1} \tilde{Z}_{\lambda}(S)} \leq 1 + \frac{1}{2} G_5(v), \qquad & n &=5, \\
    \Bigl( \frac{1+\lambda}{1-\lambda} \Bigr)^2&\leq 1+ \frac{v}{2 [ pv + \frac{4}{n}(q-p)] - \tau^{-1} \tilde{Z}_{\lambda}(S)} \leq 1 + \frac{1}{2} G_n(v), \qquad & n &\geq 6.
\end{align*}
Thus, we can evaluate $G_n(v)$ at its right endpoint $v_n$ to obtain
\begin{align*}
\Bigl( \frac{1 + \lambda}{1 - \lambda} \Bigr)^2 &\leq 1 + \frac{1}{2} G_n ( v_{\max}) = \begin{cases}
    1 + \frac{1}{2} G_5 \bigl( \frac{20 + \frac{16}{5} b_1}{15} \bigr), &n=5, \\
    1 + \frac{1}{2} G_6 \bigl( 1 + \frac{2}{15} b_1 \bigr), & n = 6, \\
    1 + \frac{1}{2} G_n(1), & n \geq 7.
\end{cases}
\end{align*}
We evaluate each of the expressions on the right-hand side to see that they equal $1 + \eta_n^{-1}$ for every $n \geq 5$, where $\eta_n$ is defined in~\eqref{eqn:eta-n(b)}.
Therefore,
\[
\frac{1+\lambda}{1-\lambda} \leq \sqrt{\frac{1+\eta_n}{\eta_n}}, \qquad \implies \qquad \lambda \leq \frac{\sqrt{1+\eta_n} - \sqrt{\eta_n}}{\sqrt{1+\eta_n} + \sqrt{\eta_n}} = ( \sqrt{1+\eta_n} - \sqrt{\eta_n} ) ^2. 
\]
This establishes the inequality~\eqref{eqn:lambda-inequality}.
Moreover, the above inequalities are strict for $\tilde{Z}_{\lambda}(S) < 0$, unless $2v-1=0$ for $n=5$, or $v=0$ for $n \geq 6$.
In either exceptional case, the first inequality would imply $( \frac{1+\lambda}{1-\lambda})^2 \leq 1$, contradicting $\lambda>0$.
Hence, the final bound for $\lambda$ is strict whenever $\tilde{Z}_{\lambda}(S)<0$.
This completes the proof of our assertion.
\end{proof}
We now obtain two further arithmetic properties of our constructed families and the functions $\eta_n(b)$ considered above, which will be important for establishing the transversal invariance of the family $C_2(b)$ under the Hamilton ODE.

\begin{lemma}\label{lemma:the-eta-values}
    For $n \geq 5$ and $b \in (0,b_2(n)]$, the functions $\eta_n(b)$ defined in~\eqref{eqn:eta-n(b)} satisfy
    \[
    \tfrac{1}{2(n+2)} < \eta_n(b) < \tfrac{1}{6}, \qquad  1+ 2 \eta_n \leq \sqrt{n+4} \, \textstyle{ \sqrt{2 \eta_n (1+\eta_n)} }.
    \]
\end{lemma}
\begin{proof}
Using the definition~\eqref{eqn:S-T-p}, we observe that $p$ is strictly decreasing in $b$.
Moreover, $q$ is strictly decreasing in $a = b + A_n b^2$, hence also strictly decreasing in $b$.
Finally, the definition~\eqref{eqn:eta-n(b)} expresses $\eta_n(b)$ as a positive linear combination of $p, q$, hence $\eta_n$ is also strictly decreasing in $b$.
Moreover, $p(0) = b_1 = \frac{1}{4(n-2)}$ and $q(0) = a_1 = \frac{9n-17}{32(n-2)^2}$, so substitution into~\eqref{eqn:eta-n(b)} produces
\[
\eta_5(0) = \tfrac{167}{1149}, \qquad \eta_6(0) = \tfrac{267}{1936}, \qquad \eta_n(0) = \tfrac{2n^2 - 3n - 1}{4n(n-2)^2} \leq \eta_7(0) = \tfrac{19}{175}.
\]
Therefore, $\eta_n(0) < \frac{1}{6}$ in all cases.
Also, $\eta_5 ( \frac{1}{32}) - \frac{1}{14} = \frac{145}{16086}$ and $\eta_6 (\frac{1}{50} ) - \frac{1}{16} = \frac{2275}{104544}$, while
\begin{align*}
\eta_n\Bigl( \frac{1}{2(n-2)^2} \Bigr) - \frac{1}{2(n+2)} &= \frac{16 n^5 - 182 n^4 + 611 n^3 - 715 n^2 + 162 n + 16}{2n(n^2-4) (4n^3 - 20n^2 + 35n-23)}.
\end{align*}
The denominator is manifestly positive for $n \geq 7$, while the numerator satisfies
\[
n^3( 16 n^2 - 162 n +500) > 0, \qquad 111 n^3 - 715 n^2 > 0
\]
for $n \geq 7$, so it is positive.
This proves that $\eta_n ( \frac{1}{2(n-2)^2})  >\frac{1}{2(n+2)}$ for all $n \geq 5$.

Finally, we observe that the function $x \mapsto \frac{1+2x}{\sqrt{2x(1+x)}}$ is strictly decreasing, with derivative $x \mapsto ( - 2x(1+x))^{- \frac{3}{2}}$.
Therefore, the upper bound for $\frac{1+2 \eta_n}{\sqrt{2 \eta_n(1+\eta_n)}}$ follows from evaluating this at $x = \frac{1}{2(n+2)}$.
This completes the proof of our assertion.
\end{proof}

\begin{lemma}\label{lemma:the-quadratic-family-2}
For $n \geq 5$, we consider the parameters $b_2(n), A_n, B_n$ introduced in~\eqref{eqn:b2(n)-parameters}.
Then, for every $b \in (0 , b_2(n)]$ and every $(x,y) \in [0,1] \times [ \sqrt{2b}, \sqrt{(n+4) b} ]$, the expression
\begin{align*}
    &\bigl[ ( \tfrac{1}{2} +b^{-1}) y + B_n P - 2 \bigr] x^2 + [ 4 - 2B_n + (2B_n-1) y] \, x - \tfrac{n-2}{n-4} \sqrt{2b} \, x(1-x) + \tfrac{2}{n-2} B_n P (1-x)^2 \\
    & \quad - \tfrac{2}{n-2}(2B_n+1) y (1-x) - \tfrac{(2B_n+1)^2}{2B_n P} \bigl( \tfrac{n-4}{2(n-2)} y^2 + b \bigr) + 2 B_nQ + \tfrac{2(n - 2A_n)}{n^2} y
\end{align*}
is strictly positive.
\end{lemma}
\begin{proof}
We denote the given function by $\psi(x,y,P,Q)$.
For fixed $(x,P,Q)$, the expression $y \mapsto \psi(x,y,P,Q)$ defines a concave quadratic with top-order coefficient $- \frac{(2B_n+1)^2}{2 B_n P} \frac{n-4}{2(n-2)} < 0$, so it is minimized at one of the boundary values $y \in \{ \sqrt{2b}, \sqrt{(n+4) b} \}$.
Moreover, every term of $\psi(x,y,P,Q)$ that depends on $P$ is strictly increasing in $P$, so $\partial_P \psi > 0$ and $\partial_Q \psi = 2B_n > 0$.
We recall the bounds $P>\frac{19}{10}$ and $Q \geq (4 - \frac{1}{n-1})b$ for $n \in \{ 5, 6\}$, as well as $Q \geq 4b$ for $n \geq 7$, due to items~\hyperref[lemma:properties-II]{$(iv)$, $(v)$} of Lemma~\ref{lemma:properties-II}.
Therefore, it suffices to prove the positivity of the functions
\begin{align*}
    \theta_1(x) &:= \begin{cases}
        \psi \bigl( x, \sqrt{2b}, \tfrac{19}{10}, (4 - \tfrac{1}{n-1}) b  \bigr), & \qquad \quad \!\! \text{for } \; n \in \{ 5,6 \}, \\
        \psi ( x, \sqrt{2b}, \tfrac{19}{10}, 4b), & \qquad \quad \!\! \text{for } \; n \geq 7,
    \end{cases} \\
    \theta_2(x) &:= \begin{cases}
        \psi \bigl( x, \textstyle{\sqrt{(n+4) b}}, \tfrac{19}{10}, (4 - \tfrac{1}{n-1} ) b \bigr), & \textup{for } \; n \in \{ 5,6\}, \\
        \psi ( x, \textstyle{\sqrt{(n+4) b}}, \tfrac{19}{10}, 4b), & \textup{for } \; n \geq 7.
    \end{cases}
\end{align*}
We now consider two cases, depending on whether $n \in \{ 5,6\}$ or $n \geq 7$.
For this, it is simpler to work with the variable $t := \sqrt{2b}$.
Notice that $t \leq \sqrt{2 b_2(n)}$, hence $t \leq \frac{1}{n-1}$ for $n \in \{ 5, 6\}$ and $t \leq \frac{1}{n-2}$ for $n \geq 7$.
In each situation, we will express $\theta_i(x)$ as a quadratic function of $x$, writing
\begin{equation}\label{eqn:theta-i-collect}
\theta_i(x) = U_i(t) \, x^2 + V_i(t) x + W_i(t), \qquad i \in \{ 1, 2\}, \qquad t := \textstyle{\sqrt{2b}}.
\end{equation}
We will prove that the expression~\eqref{eqn:theta-i-collect} remains positive for $x \in [0,1]$, for each fixed $b \in (0,b_2(n)]$.
In each case, we will use the quadratic lower bound
\begin{equation}\label{eqn:quadratic-lower-bound}
    \theta_i(x) \geq U_i x^2 + V_i x + W_i \geq W_i - \tfrac{1}{4 U_i} V_i^2.
\end{equation}
In particular, this implies that $\theta_i(x) > 0$ whenever we can produce some $c>0$ such that 
\begin{equation}\label{eqn:produce-some-c}
    U_i \geq c \, |V_i| \qquad \text{and} \qquad W_i > \tfrac{1}{4c} |V_i|.
\end{equation}

\noindent \textbf{Case 1: $n=5$}.
In this case, $A_n = 0$ and $B_n = 20$.
Moreover, $t \leq \frac{1}{4}$.
We obtain the coefficients
\begin{align*}
    U_1(t) &= 2t^{-1} + \tfrac{7}{2}t + \tfrac{184}{3}, \qquad & U_2(t) &= 3 \sqrt{2} \, t^{-1} + \tfrac{3}{4}(4 + \sqrt{2}) t +  \tfrac{184}{3}, \\
    V_1(t) &=\tfrac{10}{3} (19t - 26), \qquad & V_2(t) &= - \tfrac{260}{3} + \tfrac{1}{2} ( 199 \sqrt{2} - 6) t, \\
    W_1(t) &= \tfrac{76}{3} - \tfrac{404}{15} t + \tfrac{6869}{114} t^2, \qquad & W_2(t) &= \tfrac{76}{3} - \tfrac{202 \sqrt{2}}{5} t + \tfrac{14395}{304} t^2.
\end{align*}
For $i=1$, we establish the bound~\eqref{eqn:produce-some-c} with $c = \frac{19}{20}$.
To see this, for $t \leq \frac{1}{4}$,
\begin{align*}
    U_1 - \tfrac{19}{20} |V_1| &= \tfrac{1}{3t} (191 t^2 - 63 t + 6) > 0, \\
    W_1 - \tfrac{5}{19} |V_1| &= \tfrac{1}{570}( 34345 t^2 - 5852 t + 1440) > 0.
\end{align*}
Indeed, the first quadratic has discriminant $63^2 - 4 \cdot 191 \cdot 6 = - 615 < 0$, while the second one has
\[
5852^2 - 4 \cdot 34345 \cdot 1440 < (6 \cdot 10^3)^2 - 4 \cdot (3 \cdot 10^4) \cdot 10^3 < 4 \cdot 10^7 - 10^8 < - 5 \cdot 10^7 < 0.
\]
For $i=2$, we establish the bound~\eqref{eqn:produce-some-c} with $c = \frac{4}{3}$.
To see this, for $t \leq \frac{1}{4}$,
\begin{align*}
    U_2 - \tfrac{4}{3} |V_2| &= \tfrac{1}{36 t} [ (4803 \sqrt{2} - 36) t^2 - 1952 t + 108 \sqrt{2} ] > 0, \\
    W_2 - \tfrac{3}{16} |V_2| &= \tfrac{1}{9120} [ 431850 t^2 - (198303 \sqrt{2}+ 5130) t + 82840] > 0.
\end{align*}
Indeed, the quadratics have discriminants
\[
32(-10609 + 486 \sqrt{2}) < 0, \qquad 114(-565117013+17847270\sqrt{2})<0.
\]
Thus, $\theta_1(x)>0$ and $\theta_2(x)>0$ as claimed.

\smallskip \noindent \textbf{Case 2: $n=6$}.
In this case, $A_n = 0$ and $B_n = 20$.
Moreover, $t \leq \frac{1}{5}$.
We obtain the coefficients
\begin{align*}
    U_1(t)&=  2t^{-1} + \tfrac{5}{2} t + 55, \qquad & U_2(t) &= 2 \sqrt{5} \, t^{-1} + \tfrac{1}{2} (4 + \sqrt{5}) \, t + 55, \\
    V_1(t) &= -74 + \tfrac{115}{2} t, \qquad & V_2(t) &= -74+ \tfrac{1}{2} (119 \sqrt{5}-4) t, \\
    W_1(t) &= 19 - \tfrac{121}{6} t + \tfrac{18061}{304} t^2, \qquad & W_2(t) &= 19 - \tfrac{121\sqrt{5}}{6} t + \tfrac{11337}{304} t^2
\end{align*}
For $i=1$, we obtain the bound~\eqref{eqn:produce-some-c} with $c = 1$, using
\begin{align*}
    U_1 - |V_1| &= t^{-1} (60t^2 - 19t + 2) > 0, \\
    W_1 - \tfrac{1}{4} |V_1| &= \tfrac{1}{912} ( 54183 t^2 - 5282 t + 456) > 0.
\end{align*}
Indeed, the quadratics have discriminants
\[
19^2 - 4 \cdot 60 \cdot 2 = - 119 < 0, \qquad 5282^2 - 4 \cdot 54183 \cdot 456 = -70930268<0.
\]
Finally, for $i=2$, we establish the bound~\eqref{eqn:produce-some-c} with $c = \frac{4}{3}$.
To see this, we compute
\begin{align*}
    U_2 - \tfrac{4}{3} |V_2| &= \tfrac{1}{6t} [ (479 \sqrt{5} - 4) t^2 - 262 \, t + 12 \sqrt{5}] > 0, \\
    W_2 - \tfrac{3}{16} |V_2| &= \tfrac{1}{1824} [ 68022t^2 - (16435 \sqrt{5} + 684)t + 9348 ] > 0.
\end{align*}
The quadratics have negative discriminants,
\[
4(-11579+48 \sqrt{5}) < 0, \qquad  19(-62761297 + 1183320 \sqrt{5}) <0.
\]
Therefore, $\theta_1(x)>0$ and $\theta_2(x)>0$.
This completes the proof for $n \in \{ 5, 6\}$.

\smallskip \noindent \textbf{Case 3: $n \geq 7$}.
Finally, in this case, $A_n = \frac{n-2}{2}$ and $B_n = 4$.
We obtain the coefficients
\begin{align*}
    U_1(t) &= 2 t^{-1} + \tfrac{1}{2} t + \tfrac{28}{5} + \tfrac{n-2}{n-4} t + \tfrac{76}{5(n-2)}, \\
    V_1(t) &= - 4- \tfrac{152}{5(n-2)} - \tfrac{n-2}{n-4} t + ( 7 + \tfrac{18}{n-2} ) t, \\
    W_1(t) &= \tfrac{76}{5(n-2)} - \tfrac{18}{n-2} t + \tfrac{4}{n^2} t + \tfrac{811(n-2) + 405}{76(n-2)} t^2, \\
    U_2(t)  &= 2 t^{-1} \textstyle{\sqrt{\tfrac{n+4}{2}}} + \tfrac{1}{2} \textstyle{\sqrt{\tfrac{n+4}{2}}} t + \frac{28}{5} + \frac{n-2}{n-4} t + \frac{76}{5(n-2)}, \\
    V_2(t) &= - 4- \tfrac{152}{5(n-2)} - \tfrac{n-2}{n-4} t + \textstyle{\sqrt{\tfrac{n+4}{2}}} ( 7 + \tfrac{18}{n-2}) t , \\
    W_2(t) &= \tfrac{76}{5(n-2)} - \tfrac{18}{n-2} \textstyle{ \sqrt{\tfrac{n+4}{2}}} t + \tfrac{4 }{n^2} \textstyle{\sqrt{\tfrac{n+4}{2}}} t - \tfrac{405(n-2)^2-2434(n-2) - 4860}{304(n-2)} t^2.
\end{align*}
Here, $t = \sqrt{2b} \leq \frac{1}{n-2}$ and $V_1<0$.
Moreover, the coefficient of $t$ in $V_1(t)$ is positive, due to $\frac{n-2}{n-4}<7$.
Thus, we can discard terms with clear signs to find
\begin{equation}\label{eqn:A1-lower-bound}
U_1 \geq \tfrac{2}{t} \geq 2(n-2), \qquad |V_1| \leq 4 + \tfrac{152}{5(n-2)}, \qquad W_1 \geq \tfrac{76}{5(n-2)} - \tfrac{18}{n-2}t \geq \tfrac{1}{n-2} ( \tfrac{76}{5} - \tfrac{18}{n-2}).
\end{equation}
Consequently, for $n \geq 8$, we use the lower bound~\eqref{eqn:quadratic-lower-bound} and~\eqref{eqn:A1-lower-bound} to obtain the estimate
\begin{align*}
    \theta_1(x) &\geq \tfrac{1}{n-2} \bigl[ \tfrac{76}{5} - \tfrac{18}{n-2} - \tfrac{1}{8} ( 4 + \tfrac{152}{5(n-2)})^2 \bigr] = \tfrac{2}{25(n-2)^3} [ 165(n-2)^2 - 605(n-2) - 1444].
\end{align*}
The final expression is positive for $n \geq 8$.
For $n =7$, we use the slightly sharper estimates 
\[
U_1 \geq 2 \cdot 5  + \tfrac{28}{5} = \tfrac{78}{5}, \qquad W_1 \geq \tfrac{76}{25} - \tfrac{862}{245} \cdot \tfrac{1}{5} = \tfrac{2862}{1225} > \tfrac{58}{25}, \qquad |V_1| \leq \tfrac{252}{25}.
\]
Thus, for $n=7$, we arrive at $\theta_1(x) \geq \frac{58}{25} - \frac{1}{4} \frac{5}{78} ( \frac{781}{75})^2 = \frac{1124}{1625} > 0$.

The analysis of $\theta_2(x)$ is similar.
We first treat the case $n \geq 8$, using the bound $\frac{12}{5} < \sqrt{\frac{n+4}{2}} < \frac{n-2}{2}$.
The coefficient of $t$ in $V_2$ is positive, while $V_2 < 0$ is negative, due to 
\[
V_2 \leq - 4 - \tfrac{152}{5(n-2)} + \tfrac{1}{n-2} \tfrac{n-2}{2} ( 7 + \tfrac{18}{n-2}) = - \tfrac{1}{2} - \tfrac{107}{5(n-2)} < 0.
\]
Thus, we again obtain $|V_2| \leq 4 + \frac{152}{5(n-2)} \leq \frac{136}{15}$, while $U_2 \geq 2 \sqrt{\frac{n+4}{2}} t^{-1} \geq \frac{24}{5}(n-2)$.
Moreover, we use the bounds $\frac{12}{5} < \sqrt{\frac{n+4}{2}} < \frac{n-2}{2}$ and $t \leq \frac{1}{n-2}$ to control the negative terms of $W_2$ as
\begin{align*}
    \tfrac{405 (n-2)^2- 2434 (n-2) - 4860}{304(n-2)} \leq \tfrac{405}{304}(n-2), \qquad \tfrac{18}{n-2} \textstyle{\sqrt{\tfrac{n+4}{2}}} \, t \leq \tfrac{9}{n-2}, \qquad  \tfrac{405}{304}(n-2) t^2 \leq \tfrac{405}{304(n-2)}. 
\end{align*}
Finally, we can discard the positive term $\frac{4}{n^2} \sqrt{\frac{n+4}{2}} t$ from the expression for $W_2$ to obtain
\[
W_2 \geq \tfrac{1}{n-2} ( \tfrac{76}{5} - 9 - \tfrac{405}{304}) = \tfrac{7399}{1520(n-2)}.
\]
Combining these bounds, we arrive at
\begin{align*}
    \theta_2(x) &\geq W_2 - \tfrac{1}{4} U_2^{-1} V_2^2 \geq \tfrac{7399}{1520 (n-2)} - \tfrac{5}{96 (n-2)} ( \tfrac{136}{15})^2 = \tfrac{7399}{1520 (n-2)} - \tfrac{578}{135(n-2)} = \tfrac{24061}{41040 (n-2)}  > 0. 
\end{align*}
In the remaining case $n=7$, we use the above computations and the bound $t \leq \frac{1}{5}$ to compute
\begin{align*}
    W_2 &= \tfrac{76}{25} - \tfrac{431 \sqrt{22}}{245} t + \tfrac{1381}{304} t^2 \geq \tfrac{76}{25} - 9 \cdot \tfrac{1}{5} = \tfrac{31}{25}, \qquad |V_2| \leq 4 + \tfrac{152}{25} = \tfrac{252}{25}, \qquad U_2 \geq 5\sqrt{22} > \tfrac{70}{3}.
\end{align*}
Hence, $\theta_2(x) \geq \frac{31}{25} - \frac{1}{4} \frac{3}{70} ( \frac{252}{25})^2 = \frac{473}{3125} > 0$.
To summarize, we have proved that $\theta_1, \theta_2 > 0$ for every $n\geq 7$.
This completes the proof of our assertion.
\end{proof}

To establish the transversal invariance of the family $C_2(b)$, we first record a consequence of Lemmas A.2 and A.6 of~\cite{brendle-PIC}.
\begin{lemma}\label{lemma:weakly-two-positive}
Let $H$ be a weakly two-positive symmetric bilinear form.
Then,
\[
Z_{\lambda} ( H \owedge H) \geq - \tfrac{2}{n-4} (1-\lambda^2) (H_{11} + H_{22} ) ( \textup{tr} (H) - H_{11} - H_{22}).
\]
\end{lemma}
\begin{proof}
Because $H$ is weakly two-positive,~\cite{brendle-PIC}*{Lemma A.2} shows that $H \owedge H \in C_{\textup{PIC}}$.
Therefore,~\cite{brendle-PIC}*{Lemma A.6} implies that
\[
Z_{\lambda} (H \owedge H) \geq - \tfrac{1-\lambda^2}{n-4} ( \textup{Ric}( H \owedge H)_{11} + \textup{Ric}(H \owedge H)_{22} - 2 (H \owedge H)_{1212})
\]
For the operator $H \owedge H$, the notation of Section~\ref{sec:preliminaries} produces
\begin{align*}
    \textup{Ric}(H \owedge H)_{ij} &= 2 \, [ \textup{tr}(H) \, H_{ij} - (H^2)_{ij} ] , \qquad (H \owedge H)_{ij k \ell} = 2 ( H_{ik} H_{j \ell} - H_{i \ell} H_{jk}).
\end{align*}
Combining these terms, we arrive at
\begin{align*}
& \textup{Ric}( H \owedge H)_{11} + \textup{Ric}(H \owedge H)_{22} - 2 \, (H \owedge H)_{1212} \\
&= 2 \, \textup{tr}(H) ( H_{11} + H_{22}) - 2 ( (H^2)_{11} + (H^2)_{22}) - 4 (H_{11}H_{22} - (H_{12})^2) \\
&= 2 \, (H_{11} + H_{22}) [ \textup{tr}(H) - H_{11} - H_{22}] - 2 \sum_{p=3}^n ( (H_{1p})^2 + (H_{2p})^2) \\
&\leq 2 \, (H_{11} + H_{22}) \, [ \textup{tr}(H) - H_{11} - H_{22}].
\end{align*}
The asserted inequality follows from combining these bounds.
\end{proof}

\begin{proposition}\label{prop:trans-inv-fam-2}
For every $n \geq 5$ and $0 < b \leq b_2(n)$, the family $C_2(b)$ is transversally invariant under the Hamilton ODE.
\end{proposition}
\begin{proof}
    Since $C_2(b) = \ell_{a,b}(\cE_2(b))$, it suffices to prove the the cone $\cE_2(b)$ is transversally invariant under the modified ODE $\frac{d}{dt} S = \cQ(S) + D_{a,b}(S)$.
    Namely, we show that for $S \in \partial \cE_2(b) \setminus \{ 0 \}$, the vector $\cQ(S) + D_{a,b}(S)$ lies in the interior of the tangent cone $T_S \cE_2(b)$.
    By Proposition~\ref{prop:transversally-invariant-family-1}, the family $C_1(b_1)$ is transversally invariant under the Hamilton ODE, so it remains to examine the face $\{ \tilde{Z}_{\lambda}(S) = 0 \}$ of the cone $\cE_2(b)$.
    We claim that whenever $S \in \cE_2(b)$ satisfies $\tilde{Z}_{\lambda}(S) = 0$ for an orthonormal four-frame $\{ e_1,e_2, e_3, e_4\}$ and some $\lambda \in [0,1]$, then $\frac{d}{dt} \big\rvert_{t=0} \tilde{Z}_{\lambda}(S)>0$.
    
    In what follows, we set $\sigma := \textup{scal}(S) > 0$ 
    By Proposition~\ref{prop:why-the-long-face}, we have $\tilde{Z}_{\rho}(S;F)>0$ for every four-frame $F$ and every $\rho > ( \sqrt{1+\eta_n(b)} - \sqrt{\eta_n(b)})^2$, with $\eta_n$ defined in~\eqref{eqn:eta-n(b)}.
    Hence, our $\lambda$ has 
    \[
    \lambda \in \bigl[ 0, \textstyle{ ( \sqrt{1+\eta_n(b)} - \sqrt{\eta_n(b)})^2 } \bigr] \subseteq [ 0,1).
    \]
    In particular, $\lambda$ is an interior minimum of the map $\rho \mapsto \tilde{Z}_{\rho}(S;F)$.
    This allows us apply~\cite{brendle-PIC}*{Proposition A.8}, where we take $H$ to be the symmetric bilinear form 
    \[
    H = \sqrt{2b} \, [ \textup{Ric}(S) + B_n b \, \textup{scal}(S) \, \textup{id}].
    \]
    In our notation, this result implies that $\tilde{Z}_{\rho}(S) = A_{\rho}(S) + (1-\rho^2) (H_{11} + H_{22})$ and
    \begin{equation}\label{eqn:simon-says}
        \begin{split}
        & Z_{\lambda}(\cQ(S)) + Z_{\lambda}( H \owedge H) + 2 (1-\lambda^2) ( (S \ast H)_{11} + (S \ast H)_{22}) - (1+\lambda^2) (H_{11} + H_{22})^2 \geq 0.
        \end{split}
    \end{equation}
    for this particular orthonormal four-frame and this particular $\lambda \in [0,1)$.

    The operator $S$ evolves by the ODE $\frac{d}{dt} S = \cQ(S) + \ell_{a,b}(S)$, so we compute the derivative as
    \begin{align*}
    \tfrac{d}{dt} \big\rvert_{t = 0} \tilde{Z}_{\lambda} (S) &=  Z_{\lambda}( \cQ(S)) + Z_{\lambda} (D_{a,b}(S)) + 2 B_n b \sqrt{2b}  (1-\lambda^2) \, \textup{scal}( \cQ(S) + D_{a,b}(S)) \\
    & \quad + (1-\lambda^2) \, \sqrt{2b} \, [ \textup{Ric}(\cQ(S) + D_{a,b}(S))_{11} + \textup{Ric}( \cQ(S) +D_{a,b}(S))_{22} ] 
    \end{align*}
    following~\cite{brendle-PIC}*{Theorem 4.6}.
    We now use the expressions
    \begin{align*}
        \textup{Ric} ( \cQ(S)) &= 2 S \ast \textup{Ric}(S), \qquad \textup{scal}(\cQ(S)) = 2 \, |\textup{Ric}(S)|^2, \\
        S \ast H &= \sqrt{2b} \, S \ast \textup{Ric}(S) + B_n b \sqrt{2b} \, \sigma \, \textup{Ric}(S).
    \end{align*}
    Substituting these terms into the computation of $\frac{d}{dt} \big \rvert_{t=0} \tilde{Z}_{\lambda}(S)$, we obtain
    \begin{equation}\label{eqn:d-dt-Z-lambda}
        \begin{split}
        \tfrac{d}{dt} \big\rvert_{t=0} \tilde{Z}_{\lambda}(S) &= Z_{\lambda}(\cQ(S)) + Z_{\lambda}( D_{a,b}(S)) + 2 (1-\lambda^2) ( (S \ast H)_{11} + (S \ast H)_{22}) \\
        & \quad + \sqrt{2b} \, (1-\lambda^2) \, ( \textup{Ric}(D_{a,b}(S))_{11} + \textup{Ric}(D_{a,b}(S))_{22} ) \\
        & \quad + 2 B_n b \sqrt{2b} (1-\lambda^2) \, \bigl[ ( P \, |\textup{Ric}(S)|^2 + Q \sigma^2) - \sigma ( \textup{Ric}(S)_{11} + \textup{Ric}(S)_{22}) \bigr]
        \end{split}
    \end{equation}
    Here, we collected the terms $\textup{scal}(D_{a,b}(S)) + 2 \, |\textup{Ric}(S)|^2 = P |\textup{Ric}(S)|^2 + Q \sigma^2$ according to~\eqref{eqn:ddt-scal}.
    
    In the next step, we expand the second line of the above expression in terms of $\textup{Ric}(S)$, recalling from~\eqref{eqn:kappa} the coefficient $\kappa$ that appears in the expansion of $D_{a,b}(S)$ given in~\eqref{eqn:D-a-b}.
    Using the expression~\eqref{eqn:ric-dab} and $a-b = A_n b^2$, we can bound the two relevant components of $\textup{Ric}(D_{a,b}(S))$ by 
    \begin{equation}\label{eqn:Ric-Dab11-22}
        \begin{split}
        &\textup{Ric}(D_{a,b}(S))_{11} + \textup{Ric}(D_{a,b}(S))_{22} + 4b [ ( \textup{Ric}(S)^2)_{11} + ( \textup{Ric}(S)^2)_{22}] \\
        &= \tfrac{8}{n} ( b + \tfrac{n-2}{2} a) \, \sigma \, ( \textup{Ric}(S)_{11}+ \textup{Ric}(S)_{22}) + 4 ( b^2 + (n-1) \kappa) \, |\overset{\circ}{\textup{Ric}}(S)|^2 + \tfrac{8}{n^2} A_n b^2 \sigma^2 \\
        &\geq 4 b \sigma \, ( \textup{Ric}(S)_{11} + \textup{Ric}(S)_{22}).
        \end{split}
    \end{equation}
    In the last step, we wrote $a = b + A_n b^2$ and discarded the non-negative terms 
    \[
    4 (b^2+(n-1) \kappa) \, |\overset{\circ}{\textup{Ric}}(S)|^2 + 4(n-2) A_n b^2 \, \sigma \, ( \textup{Ric}(S)_{11}+ \textup{Ric}(S)_{22}) + \tfrac{8}{n^2} A_n b^2 \sigma^2 \geq 0.
    \]
    Likewise, using the expression~\eqref{eqn:D-a-b} for $D_{a,b}(S)$ and writing $a = b + A_n b^2$, we obtain
    \begin{equation}\label{eqn:Dab(S)-HH}
        \begin{split}
        D_{a,b}(S) - H \owedge H &= (n-2) b^2 \, \textup{Ric}(S) \owedge \textup{Ric}(S)+ 2b^2 \, \textup{Ric}(S)^2 \owedge \textup{id}  \\
        & \quad - 2( 2B_n+1) \, b^2 \sigma \, \textup{Ric}(S) \owedge \textup{id} + \tfrac{4}{n} A_n b^2 \sigma \, \textup{Ric}(S) \owedge \textup{id} \\
        & \quad + \bigl[ \kappa \, |\overset{\circ}{\textup{Ric}}(S)|^2 + \bigl( \tfrac{n - 2 A_n}{n^2} b^2 - 2 B_n^2 b^3 \bigr) \sigma^2 \bigr] \, \textup{id} \owedge \textup{id}.
        \end{split}
    \end{equation}
    On the other hand, because $\textup{Ric}(S)$ is weakly two-positive, by Lemma~\hyperref[lemma:properties-II]{\ref{lemma:properties-II}~$(ii)$}, we find
    \begin{equation}\label{eqn:throw-out-2-nonneg}
        \begin{split}
            Z_{\lambda}( \textup{Ric}(S) \owedge \textup{id}) &= \textup{Ric}(S)_{11} + \textup{Ric}(S)_{33} + \textup{Ric}(S)_{22} + \textup{Ric}(S)_{33} \\
            & \quad + \lambda^2 ( \textup{Ric}(S)_{11} + \textup{Ric}(S)_{44} + \textup{Ric}(S)_{22} + \textup{Ric}(S)_{44}) \geq 0.
        \end{split}
    \end{equation}
    Moreover, we can apply Lemma~\ref{lemma:weakly-two-positive} for the weakly two-positive form $H = \textup{Ric}(S)$ to bound
\begin{equation}\label{eqn:simon-says-A.6}
       Z_{\lambda}(\textup{Ric}(S) \owedge \textup{Ric}(S)) \geq - \tfrac{2}{n-4}(1-\lambda^2) ( \textup{Ric}(S)_{11} + \textup{Ric}(S)_{22}) ( \sigma - \textup{Ric}(S)_{11} - \textup{Ric}(S)_{22}).
\end{equation}
We now estimate $Z_{\lambda} (D_{a,b}(S) - H \owedge H)$ by applying $Z_{\lambda}$ to the both sides of~\eqref{eqn:Dab(S)-HH} and using the expressions~\eqref{eqn:d-dt-Z-lambda}, \eqref{eqn:throw-out-2-nonneg} and~\eqref{eqn:simon-says-A.6} to estimate the right-hand side.
Furthermore, we discard the non-negative terms $4(1+\lambda^2)\, \kappa \, |\overset{\circ}{\textup{Ric}}(S)|^2$ and $\frac{4}{n} A_n b^2 \sigma \, Z_{\lambda}(\textup{Ric}(S) \owedge \textup{id})$.
Thus, we obtain
\allowdisplaybreaks{
\begin{equation}\label{eqn:koitame-kanw-Dab}
    \begin{split}
        & Z_{\lambda} ( D_{a,b}(S) - H \owedge H) + (1+\lambda^2) (H_{11} + H_{22})^2 - 2b^2 Z_{\lambda}( \textup{Ric}(S)^2 \owedge \textup{id}) \\
        & \quad +  4 (2B_n+1) b^2 \sigma ( \textup{Ric}(S)_{33} +\lambda^2 \textup{Ric}(S)_{44}) \\
        &\geq - \tfrac{2(n-2)}{n-4} b^2(1-\lambda^2) ( \textup{Ric}(S)_{11} + \textup{Ric}(S)_{22}) (\sigma - \textup{Ric}(S)_{11} - \textup{Ric}(S)_{22}) \\
        & \quad + 2b(1+\lambda^2) ( \textup{Ric}(S)_{11} + \textup{Ric}(S)_{22})^2 + \tfrac{4}{n^2} (n- 2A_n) b^2(1+\lambda^2) \sigma^2 \\
    & \quad + 2 (2B_n-1) b^2 (1+\lambda^2) \, \sigma ( \textup{Ric}(S)_{11} + \textup{Ric}(S)_{22}) .
    \end{split}
\end{equation}}
Here, we expanded the expression of~\eqref{eqn:simon-says} in the form
\[
(H_{11} + H_{22})^2 = 2 b \, \bigl[ ( \textup{Ric}(S)_{11} +\textup{Ric}(S)_{22})^2 + 4 B_n b \, ( \textup{Ric}(S)_{11} + \textup{Ric}(S)_{22}) \sigma + 4 B_n^2 b^2 \sigma^2 ].
\]
In particular, the positive contribution $8 (1+\lambda^2) B_n^2 b^3 \sigma^2$ from the expression~\eqref{eqn:simon-says} precisely cancels the negative contribution from $Z_{\lambda} (-2 B_n^2 b^3\sigma^2 \, \textup{id} \owedge \textup{id})$.

In the next step, we use the expansion
\[
Z_{\lambda} ( \textup{Ric}(S)^2 \owedge \textup{id}) = (1+\lambda^2) ( ( \textup{Ric}(S)^2)_{11} + ( \textup{Ric}(S)^2)_{22}) + 2 \bigl( ( \textup{Ric}(S)^2)_{33} + \lambda^2 ( \textup{Ric}(S)^2)_{44} \bigr).
\]
Adding $2B_n b \sqrt{2b}(1-\lambda^2) P \, |\textup{Ric}(S)|^2$ and discarding $4b^2 [( \textup{Ric}(S)^2)_{33} + ( \textup{Ric}(S)^2)_{44}] \geq 0$, we obtain
\begin{equation}\label{eqn:bound-A-lambda-Ric(S)2-id}
    \begin{split}
    & 2 b^2 Z_{\lambda} ( \textup{Ric}(S)^2 \owedge \textup{id}) - 4 b \sqrt{2b}(1-\lambda^2) \, \bigl[ (\textup{Ric}(S)^2)_{11} + ( \textup{Ric}(S)^2)_{22} \bigr] \\
    & \quad - 2 B_n b \sqrt{2b}(1-\lambda^2) P \, \sum_{i=3}^n ( \textup{Ric}(S)^2)_{ii}  + 2B_n b \sqrt{2b}(1-\lambda^2) P \, |\textup{Ric}(S)|^2 \\
    &\geq 2b \, [ b(1+\lambda^2) + \sqrt{2b} (1-\lambda^2) (B_n P - 2)] \, \bigl[ ( \textup{Ric}(S)^2)_{11} + ( \textup{Ric}(S)^2)_{22} \bigr] . 
    \end{split}
\end{equation}
Recall from~\eqref{eqn:b2(n)-parameters} and Lemma~\hyperref[lemma:properties-II]{\ref{lemma:properties-II}~$(iii)$} that $B_n \geq 4$ and $P > \frac{19}{10}$, so $B_n P-2 \geq 0$ and the coefficient of the right-hand side is non-negative.
Hence, by the Cauchy-Schwarz inequality,
\[
( \textup{Ric}(S)^2)_{11} + ( \textup{Ric}(S)^2)_{22} \geq (\textup{Ric}(S)_{11})^2 + ( \textup{Ric}(S)_{22})^2 \geq \tfrac{1}{2} ( \textup{Ric}(S)_{11} + \textup{Ric}(S)_{22})^2.
\]
Let us write $x := \frac{\textup{Ric}(S)_{11} + \textup{Ric}(S)_{22}}{\sigma}$, so that $\sum_{i=3}^n \textup{Ric}(S)_{ii} =\sigma (1-x)$.
Because $\textup{Ric}(S)$ is weakly two-positive, we have $x \geq 0$; by the same token, $0 \leq \sum_{i=3}^n \textup{Ric}(S)_{ii} = \sigma(1-x)$ implies that $x \leq 1$.
Using this bound in~\eqref{eqn:bound-A-lambda-Ric(S)2-id}, we obtain
\begin{equation}\label{eqn:2b2Y-better}
    \begin{split}
            & 2 b^2 Z_{\lambda} ( \textup{Ric}(S)^2 \owedge \textup{id}) - 4 b \sqrt{2b}(1-\lambda^2) \, \bigl[ (\textup{Ric}(S)^2)_{11} + ( \textup{Ric}(S)^2)_{22} \bigr] \\
    & \quad - 2 B_n b \sqrt{2b}(1-\lambda^2) P \, \sum_{i=3}^n ( \textup{Ric}(S)^2)_{ii}  + 2B_n b \sqrt{2b}(1-\lambda^2) P \, |\textup{Ric}(S)|^2 \\
        &\geq b \sqrt{2b} (1-\lambda^2) [ \tfrac{1}{2} \sqrt{2b} \tfrac{1+\lambda^2}{1-\lambda^2} + (B_n P - 2)] \, x^2. 
    \end{split}
\end{equation}
Next, we recall an elementary covariance inequality: for $(n-2)$-vectors $\alpha, \beta$, it holds that
\[
\left| \sum_{i=1}^{n-2} \alpha_i \beta_i - \frac{1}{n-2} \Bigl( \sum_{i=1}^{n-2} \alpha_i \Bigr) \Bigl( \sum_{i=1}^{n-2} \beta_i \Bigr) \right|^2 \leq \left[ \sum_{i=1}^{n-2} \alpha_i^2 - \frac{1}{n-2} \Bigl( \sum_{i=1}^{n-2} \alpha_i \Bigr)^2 \right] \, \left[ \sum_{i=1}^{n-2} \Bigl( \beta_i - \frac{1}{n-2} \sum_{j=1}^{n-2} \beta_j \Bigr)^2 \right].
\]
We apply this inequality to the $(n-2)$-vectors $\alpha = (1, \lambda^2, 0, \dots, 0)$ and $\beta = ( \textup{Ric}(S)_{33}, \dots, \textup{Ric}(S)_{nn})$.
This produces the bound
\[
\Bigl[ \textup{Ric}(S)_{33} + \lambda^2 \textup{Ric}(S)_{44} - \frac{1+\lambda^2}{n-2} \sigma(1-x) \Bigr]^2 \leq \Bigl( 1+\lambda^4 - \frac{(1+\lambda^2)^2}{n-2} \Bigr) \sum_{i=3}^n \Bigl( \textup{Ric}(S)_{ii} - \frac{\sigma (1-x)}{n-2} \Bigr)^2.
\]
Setting $z := \textup{Ric}(S)_{33} + \lambda^2 \textup{Ric}(S)_{44}$, we therefore obtain
\[
\sum_{i=3}^n (\textup{Ric}(S)^2)_{ii} \geq \frac{\sigma^2(1-x)^2}{n-2} + \frac{ \bigl[ z - \frac{1+\lambda^2}{n-2} \sigma(1-x) \bigr]^2 }{1+\lambda^4 - \frac{(1+\lambda^2)^2}{n-2}}.
\]
This implies that
\begin{align*}
    & 2 B_n b \sqrt{2b} (1-\lambda^2) P \sum_{i=3}^n ( \textup{Ric}(S)^2)_{ii} - 4 (2B_n+1) b^2 \sigma ( \textup{Ric}(S)_{33} + \lambda^2 \textup{Ric}(S)_{44}) \\
    &\geq \frac{2B_n b \sqrt{2b} (1-\lambda^2) P}{n-2} \sigma^2(1-x)^2+  \frac{2 B_n b \sqrt{2b} (1-\lambda^2) P}{1+\lambda^4 - \frac{(1+\lambda^2)^2}{n-2}} \Bigl[ z - \frac{1+\lambda^2}{n-2} \sigma(1-x) \Bigr]^2 - 4 (2B_n+1) b^2 \sigma z.
\end{align*}
Next, we complete the square in $z$ to obtain
\begin{align*}
   & \frac{2 B_n b \sqrt{2b} (1-\lambda^2) P}{1+\lambda^4 - \frac{(1+\lambda^2)^2}{n-2}} \Bigl[ z - \frac{1+\lambda^2}{n-2} \sigma(1-x) \Bigr]^2 - 4 (2B_n+1) b^2 \sigma z \\
   &\geq - \frac{4(2B_n+1) b^2(1+\lambda^2)}{n-2} \sigma^2(1-x) - \frac{2 (2B_n+1)^2 b^3}{B_n \sqrt{2b} (1-\lambda^2)P} \Bigl( 1+\lambda^4 - \frac{(1+\lambda^2)^2}{n-2} \Bigr) \sigma^2.
\end{align*}
To abbreviate the ensuing computations, it is convenient to define 
\[
y := \sqrt{2b} \tfrac{1+\lambda^2}{1-\lambda^2}
\]
and collect terms using
\[
1+\lambda^4 - \tfrac{(1+\lambda^2)^2}{n-2} = \tfrac{1}{2} (1-\lambda^2)^2 + \tfrac{n-4}{2(n-2)} (1+\lambda^2)^2 = \tfrac{1}{2b} (1-\lambda^2)^2 \bigl( b + \tfrac{n-4}{2(n-2)} y^2\bigr).
\]
Consequently, the previous bounds can be summarized as
\begin{equation}\label{eqn:summarize-crazy-bounds}
    \begin{split}
    & 2 B_n b \sqrt{2b} (1-\lambda^2) P \sum_{i=3}^n ( \textup{Ric}(S)^2)_{ii} - 4 (2B_n+1) b^2 \sigma ( \textup{Ric}(S)_{33} + \lambda^2 \textup{Ric}(S)_{44}) \\
    &\geq b \sqrt{2b} (1-\lambda^2) \sigma^2 \, \left[ \frac{2 B_n P}{n-2}(1-x)^2 - \frac{2(2B_n+1)}{n-2} y (1-x) - \frac{(2B_n+1)^2}{2 B_n P} \Bigl( \frac{n-4}{2(n-2)} y^2 + b \Bigr) \right]
    \end{split}
\end{equation}
In the next step, we will combine these terms to estimate the right-hand side of the evolution $\frac{d}{dt}\big\rvert_{t=0} \tilde{Z}_{\lambda}(S)$ from~\eqref{eqn:d-dt-Z-lambda}.
We sum the inequalities
\begin{align*}
    \eqref{eqn:simon-says} + \sqrt{2b} \, (1-\lambda^2) \cdot \eqref{eqn:Ric-Dab11-22} + \eqref{eqn:koitame-kanw-Dab} + \eqref{eqn:2b2Y-better} + \eqref{eqn:summarize-crazy-bounds} + 2 B_n b \sqrt{2b}(1-\lambda^2) (Q-x)  \sigma^2.
\end{align*}
Here, we recall that $x = \frac{\textup{Ric}(S)_{11} + \textup{Ric}(S)_{22}}{\sigma}$.
We observe the cancellations between terms on the left-hand side of this inequality, using $Z_{\lambda}(D_{a,b}(S)) = Z_{\lambda}(D_{a,b}(S) - H \owedge H) + Z_{\lambda}( H \owedge H)$, which produce the expression~\eqref{eqn:d-dt-Z-lambda} for $\frac{d}{dt} \big\rvert_{t=0} \tilde{Z}_{\lambda}$.
Thus, after assembling these estimates in~\eqref{eqn:d-dt-Z-lambda} and dividing by $b \sqrt{2b} (1-\lambda^2) \sigma^2>0$, we obtain
\begin{align*}
    & \bigl( b \sqrt{2b} (1-\lambda^2) \sigma^2 \bigr)^{-1} \tfrac{d}{dt} \big\rvert_{t=0} \tilde{Z}_{\lambda}(S) \\
    &\geq \bigl[ ( \tfrac{1}{2} +b^{-1}) y + B_n P - 2 \bigr] x^2 + [ 4 - 2B_n + (2B_n-1) y] \, x - \tfrac{n-2}{n-4} \sqrt{2b} \, x(1-x) + \tfrac{2}{n-2} B_n P (1-x)^2 \\
    & \quad - \tfrac{2}{n-2}(2B_n+1) y (1-x) - \tfrac{(2B_n+1)^2}{2B_n P} \bigl( \tfrac{n-4}{2(n-2)} y^2 + b \bigr) + 2 B_nQ + \tfrac{2(n - 2A_n)}{n^2} y.
\end{align*}
In Proposition~\ref{prop:why-the-long-face}, we proved the bound $\frac{1+\lambda}{1-\lambda} \leq \sqrt{\frac{1+\eta_n}{\eta_n}}$.
Consequently,
\[
\Bigl( \frac{1+\lambda^2}{1-\lambda^2} \Bigr)^2 = \frac{1}{4} \Bigl( \frac{1+\lambda}{1-\lambda} + \frac{1-\lambda}{1+\lambda} \Bigr)^2 \leq \frac{1}{4} \left( \sqrt{\frac{1+\eta_n}{\eta_n}} + \sqrt{\frac{\eta_n}{1+\eta_n}} \right)^2 = \frac{(1+2 \eta_n)^2}{4 \eta_n(1+\eta_n)}
\]
Here, we used the fact that the function $s \mapsto s + \frac{1}{s}$ is increasing in $s = \frac{1+\lambda}{1-\lambda} \geq 1$.
Thus,
\begin{equation}\label{eqn:y-quadratic}
y := \sqrt{2b} \tfrac{1+\lambda^2}{1-\lambda^2}, \qquad \sqrt{2b} \leq y \leq \sqrt{b} \tfrac{1+ 2 \eta_n}{\sqrt{2 \eta_n (1+\eta_n)}} \leq \textstyle{ \sqrt{(n+4) b}}
\end{equation}
upon applying Lemma~\ref{lemma:the-eta-values}.
Finally, since $(x,y) \in [0,1] \times [ \sqrt{2b} , \sqrt{(n+4) b} ]$, Lemma~\ref{lemma:the-quadratic-family-2} shows that the right-hand side of the final lower bound for $\frac{d}{dt} \big\rvert_{t=0} \tilde{Z}_{\lambda}(S)$ is strictly positive.
Therefore, the family $C_2(b) = \ell_{a,b} ( \cE_2(b))$ is transversally invariant under Hamilton's ODE as claimed.
\end{proof}

\section{Connecting the two families of cones}\label{section:joining-condition}

We finally prove that the two families $C_1(b)$ and $C_2(b)$ can be joined at their endpoints.
Since $C_2(b) \subseteq C_1(b_1)$ by definition, it suffices to prove that $\ell_{a_2,b_2}^{-1} (C_1(b_1)) \subseteq \cE_2(b_2)$.
This property amounts to $\tilde{Z}_{\lambda}(S;F) \geq 0$ for every algebraic curvature operator $S \in S^2_B(\fs \fo(n))$ with $\ell_{a_2,b_2}(S) \in C_1(b_1)$, every orthonormal four-frame $F$, and every $\lambda \in [0,1]$.

We consider the operator $T = \ell_{a_1, b_1}^{-1} \ell_{a,b}(S) \in \cE_1(b_1)$, which satisfies~\eqref{eqn:S-T-p} and has $T \in C_{\textup{PIC}}$ and $\textup{Ric}(T)_{11} + \textup{Ric}(T)_{22} + \frac{4}{5n} b_1\,  \textup{scal}(T) \geq 0$.
    Also, $S$ and $T$ have the same Weyl tensor, while their traceless Ricci tensors and scalar curvatures are related by~\eqref{eqn:ric-scal-ST-equalities}.
Hence,
    \begin{equation}\label{eqn:Ric(S)-expression}
    \begin{split}
    \overset{\circ}{\textup{Ric}}(S) &= ( 1 + (n-2) p) \, \overset{\circ}{\textup{Ric}}(T), \qquad \textup{scal}(S) = (1 + 2(n-1) q) \, \textup{scal}(T), \\
    \textup{Ric}(S)_{ii} &= (1 + (n-2) p) \, \textup{Ric}(T)_{ii} + \tfrac{1}{n} ( 2(n-1) q - (n-2) p) \, \textup{scal}(T).
    \end{split}
    \end{equation}
The verification of the condition $C_1(b_1) = C_2(b_2)$ for connecting the two cones relies on a precise weighting of the various PIC inequalities for the tensor $T$.
To describe these, we use the parameters of the second family, introduced in~\eqref{eqn:b2(n)-parameters}, and consider the auxiliary constants
\begin{equation}\label{eqn:endpoint-parameters}
    r_2(n) := B_n b_2(n) \textstyle{\sqrt{2 b_2(n)}}, \qquad u_2(n) := \textstyle{\sqrt{2 b_2(n)}} \, [1 + (n-2) p_2(n)] \, .
\end{equation}
Substituting our definitions, we directly obtain
\begin{equation}\label{eqn:direct-substitution-n=5,6}
\begin{aligned}
    n&=5: \quad & p_2 &= \tfrac{1}{21}, \quad & q_2 &= \tfrac{19}{360}, \quad & r_2 &= \tfrac{5}{32}, \quad & u_2 &= \tfrac{2}{7}. \\
    n&=6: \quad & p_2 &= \tfrac{17}{432}, \quad & q_2 &= \tfrac{223}{5120}, \quad & r_2 &= \tfrac{2}{25}, \quad & u_2 &= \tfrac{25}{108}.
\end{aligned}
\end{equation}
Moreover, for $n \geq 7$, we obtain
\begin{equation}\label{eqn:direct-substitution-ngeq7}
    \begin{aligned}
    p_2 &= \tfrac{n-4}{2(n-2)(2n-3)}, \quad & q_2 &= \tfrac{9n^2 - 51n+62}{8 (4n^3-20n^2 + 37n-25)}, \qquad & r_2 &= \tfrac{2}{(n-2)^3}, \quad & u_2 &= \tfrac{5}{2(2n-3)}.
    \end{aligned}
\end{equation}
The constants~\eqref{eqn:b2(n)-parameters} family $C_2(b)$ are designed to satisfy
\begin{equation}\label{eqn:u2-p2}
    u_2 - p_2 = \tfrac{1}{n-2}, \qquad \text{for } \; n \geq 7.
\end{equation}
We will prove a lower bound for a tensor combination obtained when joining the two cones.
\begin{proposition}\label{prop:joining-lower-bound}
    For $n \geq 5$, we define the expression
\[
    F_n(T) := \tfrac{2}{n-2} (T_{1212} + T_{1313} + T_{2323}) + (p_2 + u_2 - \tfrac{1}{n-2})( \textup{Ric}(T)_{11} + \textup{Ric}(T)_{22}) + 2 p_2 \, \textup{Ric}(T)_{33}.
\]
Moreover, we define dimensional constants $d_n$ by
\begin{equation}\label{eqn:d5-d6-dn}
d_5 := \frac{7}{16}, \qquad d_6 := \frac{149}{648}, \qquad d_n := \frac{[-n^2+13n - 26]_+}{2(n-2)(n-4)(2n-3)} \quad \text{for } \; n \geq 7.
\end{equation}
Then, every algebraic operator $T \in \cE_1(b_1)$ satisfies $F_n(T) \geq - d_n \, \textup{scal}(T)$.
\end{proposition}
\begin{proof}
We prove the bound separately in the cases $n=5$, $n=6$, and $n \geq 7$.
For this, we take the average of the PIC inequalities under sign reversal over four-frames as in~\eqref{lemma:consequences-of-PIC}, and then average over the indices indicated below, to obtain
\allowdisplaybreaks{
\begin{align}
    \tag{I} 0&\leq \frac{1}{2} ( T_{1313} + T_{2323}) + \frac{1}{2(n-3)} \sum_{\alpha = 4}^n (T_{1 \alpha 1 \alpha} + T_{2 \alpha 2 \alpha}), \label{eqn:eqn-I} \\
    \tag{II} 0 &\leq T_{1212} + \frac{1}{2} (T_{1313} + T_{2323}) + \frac{1}{2(n-3)} \sum_{\alpha = 4}^n (T_{1 \alpha 1 \alpha} + T_{2 \alpha 2 \alpha}) + \frac{1}{n-3} \sum_{ \alpha = 4}^n T_{3 \alpha 3 \alpha}, \label{eqn:eqn-II} \\
    \tag{III} 0 &\leq T_{1212} + \frac{1}{n-3} \sum_{\alpha=4}^n (T_{1 \alpha 1 \alpha} + T_{2 \alpha 2 \alpha}) + \frac{2}{(n-3)(n-4)} \sum_{4 \leq \alpha < \beta \leq n} T_{\alpha \beta \alpha \beta}, \label{eqn:eqn-III} \\
    0&\leq \frac{1}{2} (T_{1313} + T_{2323}) + \frac{1}{2(n-3)} \sum_{\alpha=4}^n (T_{1 \alpha 1 \alpha} + T_{2 \alpha 2 \alpha}) + \frac{1}{n-3} \sum_{\alpha=4}^n T_{3 \alpha 3 \alpha} \tag{IV} \label{eqn:eqn-IV}\\
    & \quad + \frac{2}{(n-3)(n-4)} \sum_{4 \leq \alpha <\beta \leq n} T_{\alpha \beta \alpha \beta}. \notag \intertext{for every $n \geq 5$.
    Moreover, for $n \geq 6$, we also obtain} 
    0 &\leq \frac{1}{n-3} \sum_{\alpha=4}^n T_{3 \alpha 3 \alpha} + \frac{2}{(n-3)(n-4)} \sum_{4 \leq \alpha < \beta \leq n} T_{\alpha \beta \alpha \beta}, \tag{V} \label{eqn:eqn-V} \\
    0&\leq \frac{1}{2(n-3)} \sum_{\alpha=4}^n (T_{1 \alpha 1 \alpha} + T_{2 \alpha 2 \alpha}) + \frac{2}{(n-3)(n-4)} \sum_{4 \leq \alpha < \beta \leq n} T_{\alpha \beta \alpha \beta}, \tag{VI} \label{eqn:eqn-VI}.
\end{align}}
Finally, for $n \geq 7$, we can also restrict the PIC inequalities to $\textup{span} \{ e_4, \dots , e_n \}$ with non-negative scalar curvature, as in Lemma~\ref{lemma:consequences-of-PIC} and~\cite{brendle-PIC}*{Lemma A.4}.
Thus, for $n \geq 7$, we in fact have
\begin{equation}\label{eqn:eqn-VII}
\tag{VII} \qquad 0 \leq \frac{2}{(n-3)(n-4)} \sum_{4 \leq \alpha<\beta \leq n } T_{\alpha \beta \alpha \beta}.
\end{equation}
For brevity, we denote $\tau := \textup{scal}(T)$ in what follows.

\smallskip \noindent \textbf{The case $n=5$}. \label{step:pic-n=5}
For a tensor $T \in S^2( \fs \fo(n))$, we define the expression
\begin{align*}
\Delta(T) &:= \frac{1}{2} ( \textup{Ric}(T)_{44} + \textup{Ric}(T)_{55}) - \frac{1}{3} ( \textup{Ric}(T)_{11} + \textup{Ric}(T)_{22} + \textup{Ric}(T)_{33}) \\
&= \frac{1}{6} \sum_{i=3}^3 \sum_{\alpha=4}^5 ( \textup{Ric}(T)_{\alpha \alpha} - \textup{Ric}(T)_{ii} )
\end{align*}
We notice from~\eqref{eqn:direct-substitution-n=5,6} that the constants $p_2, u_2$ satisfy $p_2+u_2 = \frac{1}{3}$, so $F_5(T)$ produces
\begin{equation}\label{eqn:F5(U)}
F_5(T) + \tfrac{3}{8} \textup{scal} (T) + \tfrac{1}{4} \Delta(T) = \tfrac{11}{12} \cdot \eqref{eqn:eqn-I} + \tfrac{85}{84} \cdot \eqref{eqn:eqn-II} + \tfrac{5}{21} \cdot \eqref{eqn:eqn-III} + \tfrac{16}{21} \cdot \eqref{eqn:eqn-IV}.
\end{equation}
Because $T \in \cE_1(b_1)$, we can decompose $T = U + N$ inside $S^2( \fs \fo(n))$, where $U \in C_\textup{PIC}$ and $N \geq 0$ are as in Definition~\ref{def:first-cone-family} and need not satisfy the first Bianchi identity.
Moreover, at $b_1 = \frac{1}{12}$, the condition~\eqref{eqn:S-define-ON} assumes the form
\begin{equation}\label{eqn:pinching-at-b1.for.n=5}
    (\textup{Ric}(T)_{\alpha \alpha} - \textup{Ric}(T)_{\beta \beta} )^2 \leq 4 \, \tau \sum_{k \not\in \{ \alpha, \beta \} } (N_{ \alpha k \alpha k } + N_{ \beta k \beta k} )
\end{equation}
for all $\{ e_{\alpha}, e_{\beta} \}$ in an orthonormal $5$-frame.
Using the Cauchy-Schwarz inequality and~\eqref{eqn:pinching-at-b1.for.n=5}, we obtain
\begin{equation}\label{eqn:Delta-squared}
    (\Delta(T))^2 \leq \frac{1}{6} \sum_{i=1}^3 \sum_{\alpha=4}^5 ( \textup{Ric}(T)_{\alpha \alpha } - \textup{Ric}(T)_{ii})^2 \leq \frac{4 \tau}{6} \sum_{i=1}^3 \sum_{\alpha =4}^5 \sum_{k \not\in \{ i,\alpha \} } (N_{ikik} + N_{\alpha k \alpha k})
\end{equation}
Consequently, we can apply the decomposition~\eqref{eqn:F5(U)} and use $U \in C_{\textup{PIC}}$ to obtain
\begin{align*}
    F_5(T) + \tfrac{3}{8} \tau + \tfrac{1}{4} \Delta(T) &\geq F_5(N) + \tfrac{3}{8} \textup{scal} (N) + \tfrac{1}{4} \Delta(N) \\
    &= \frac{1}{6} \sum_{i=1}^3 \sum_{\alpha =4}^5 \sum_{k \not\in \{ i, \alpha \} } (N_{ikik} + N_{\alpha k \alpha k}) +\frac{7}{12} \Bigl( N_{1212} + \frac{1}{2} \sum_{i=1}^2 \sum_{\alpha=4}^5 N_{i \alpha i \alpha}  \Bigr) \\
    & \quad + \frac{19}{28} (N_{1313} + N_{2323}) + \frac{65}{168} (N_{3434} + N_{3535}) \\
    &\geq \tfrac{1}{4 \tau} ( \Delta(T))^2.
\end{align*}
In the last step, we used the inequality~\eqref{eqn:Delta-squared} to estimate the first term, and discarded the remaining non-negative terms due to $N \geq 0$.
We therefore obtain
\[
F_5(T) + \tfrac{7}{16} \tau \geq \tfrac{1}{16} \tau - \tfrac{1}{4} \Delta(T) + \tfrac{1}{4 \tau} (\Delta(T))^2 = \tfrac{1}{4 \tau} \bigl[ \Delta(T) - \tfrac{1}{2} \tau \bigr]^2 \geq 0.
\]
This inequality shows that $F_5(T) \geq - \frac{7}{16} \textup{scal}(T)$ as claimed.

\smallskip \noindent \textbf{The case $n=6$}.
The constants $p_2, u_2$ given in~\eqref{eqn:direct-substitution-n=5,6} satisfy $p_2+u_2 - \frac{1}{4} = \frac{1}{48}$ and $2p_2 = \frac{17}{216}$, so
\[
F_6(T) = \tfrac{1}{2} (T_{1212} + T_{1313} + T_{2323}) + \tfrac{1}{48} ( \textup{Ric}(T)_{11} + \textup{Ric}(T)_{22}) + \tfrac{17}{216} \textup{Ric}(T)_{33}.
\]
We observe the non-negative combination of PIC inequalities
\[
F_6(T) + \tfrac{149}{648} \, \textup{scal}(T) = 2 \, \bigl[ \tfrac{181}{324} \cdot \eqref{eqn:eqn-I} + \tfrac{649}{1296} \cdot \eqref{eqn:eqn-II} + \tfrac{199}{648} \cdot \eqref{eqn:eqn-V} + \tfrac{31}{81} \cdot \eqref{eqn:eqn-VI} \bigr] \geq 0.
\]
Therefore, $F_6(T) \geq - \frac{149}{648} \, \textup{scal}(T) = - d_6 \, \textup{scal}(T)$ as claimed.

\smallskip \noindent \textbf{The case $n \geq 7$}.
Finally, for $n \geq 7$, we have $u_2 - p_2 = \frac{1}{n-2}$, hence
\[
F_n(T) = \tfrac{2}{n-2} (T_{1212} + T_{1313} + T_{2323}) + 2 p_2 \, ( \textup{Ric}(T)_{11} + \textup{Ric}(T)_{22} + \textup{Ric}(T)_{33}).
\]
Since $\frac{n-4}{2(2n-3)} < \frac{1}{n-5}$ for $n \leq 10$, we have $p_2 < \frac{1}{(n-2)(n-5)}$.
Thus, for $7 \leq n \leq 10$, we can use the non-negative combination of PIC inequalities
\[
F_n(T) + d_n \, \textup{scal}(T) = \tfrac{2(n-3) (1 + (n-2) p_2)}{(n-4)(n-2)} [ \eqref{eqn:eqn-I} + \eqref{eqn:eqn-II}] + \tfrac{n-3}{n-2} [ 1 - (n-2)(n-5) p_2 ] \cdot \eqref{eqn:eqn-VII} \geq 0.
\]
For $n \geq 11$, it holds that $\frac{n-4}{2(2n-3)} > \frac{1}{n-5}$, so we have $p_2 > \frac{1}{(n-2)(n-5)}$.
In this case, we use
\begin{align*}
& \tfrac{2}{n-2} (T_{1212} + T_{1313} + T_{2323}) +  \tfrac{2}{(n-2)(n-5)} ( \textup{Ric}(T)_{11} + \textup{Ric}(T)_{22} + \textup{Ric}(T)_{33}) \\
&= \tfrac{2(n-3)}{(n-2)(n-5)} \cdot [ \eqref{eqn:eqn-I} + \eqref{eqn:eqn-II}] \geq 0.
\end{align*}
On the other hand, Lemma~\ref{lemma:consequences-of-PIC} shows that $\textup{Ric}(T)$ is weakly $3$-positive, because $T \in C_{\textup{PIC}}$.
Therefore,
\begin{align*}
F_n(T) &= \tfrac{2(n-3)}{(n-2)(n-5)} \cdot [ \eqref{eqn:eqn-I} + \eqref{eqn:eqn-II}] +  2 (p_2 - \tfrac{1}{(n-2)(n-5)}) ( \textup{Ric}(T)_{11} + \textup{Ric}(T)_{22} + \textup{Ric}(T)_{33}) \geq 0.
\end{align*}
Thus, $F_n(T) \geq 0$, corresponding to $d_n=0$.
This completes the proof of our assertion.
\end{proof}

\begin{lemma}\label{lemma:A-lambda-lower-bound}
For every algebraic operator $T \in \cE_1(b_1)$, it holds that
\begin{equation}\label{eqn:A-lambda(T)-brendle-inequality}
\begin{split}
    Z_{\lambda}(T) &\geq \tfrac{1}{n-2} (1-\lambda^2) \, [ 2 \, (T_{1212} + T_{1313} + T_{2323}) - \textup{Ric}(T)_{11} - \textup{Ric}(T)_{22}]\, .
\end{split}
\end{equation}
\end{lemma}
\begin{proof}
First, using the argument of~\cite{brendle-PIC}*{Lemma A.6}, the property $T \in C_{\textup{PIC}}$ implies that
\begin{equation}\label{eqn:A-lambda-H}
    Z_{\lambda}(T) \geq - \frac{1-\lambda^2}{n-4} \sum_{\alpha=5}^n (T_{1 \alpha 1 \alpha} + T_{2 \alpha 2 \alpha})
\end{equation}
for every orthonormal four-frame $\{ e_1, e_2, e_3, e_4 \}$ and every $\lambda \in [0,1]$.
On the other hand, using the property $T \in C_{\textup{PIC}}$ with both orientations of the frame, we find
\[
T_{1313} + T_{1414} + T_{2323} + T_{2424} \geq 2 \, |T_{1234}|.
\]
We use this inequality and rearrange terms inside $Z_{\lambda}(T)$ to obtain
\begin{align*}
    & Z_{\lambda}(T) - \tfrac{1}{2} (1-\lambda^2) (T_{1313} + T_{2323} - T_{1414} - T_{2424}) \\
    &= \tfrac{1}{2} (1+\lambda^2) ( T_{1313} + T_{1414} + T_{2323}+ T_{2424}) - 2 \lambda T_{1234} \\
    &\geq (1+\lambda^2) \, |T_{1234}| - 2 \lambda \, T_{1234}.
\end{align*}
In particular, this term is at least $\frac{1}{2} (1-\lambda)^2 \, |T_{1234}|$, so it is non-negative.
Therefore,
\begin{equation}\label{eqn:Y-lambda(T)}
Z_{\lambda}(T) \geq \tfrac{1}{2} (1-\lambda^2) (T_{1313} + T_{2323} - T_{1414} - T_{2424}) .
\end{equation}
We now combine the inequalities~\eqref{eqn:A-lambda-H} and~\eqref{eqn:Y-lambda(T)} with weights $\frac{n-4}{n-2}$ and $\frac{2}{n-2}$ to obtain
\begin{equation}\label{eqn:A-lambda(T)-first-step}
Z_{\lambda}(T) \geq \frac{1-\lambda^2}{n-2} \Bigl[ (T_{1313} + T_{2323} - T_{1414} - T_{2424}) -  \sum_{\alpha=5}^n (T_{1 \alpha 1 \alpha} + T_{2 \alpha 2 \alpha}) \Bigr].
\end{equation}
Finally, using the definition of $\textup{Ric}(T)$, we arrive at the identity
\begin{align*}
& T_{1313} + T_{2323} - T_{1414} - T_{2424} -\sum_{\alpha=5}^n (T_{1 \alpha 1 \alpha} + T_{2 \alpha 2 \alpha}) \\
&= 2 \, (T_{1212} + T_{1313} + T_{2323}) - \textup{Ric}(T)_{11} - \textup{Ric}(T)_{22}.
\end{align*}
The desired lower bound for $Z_{\lambda}(T)$ follows from substituting this expression into~\eqref{eqn:A-lambda(T)-first-step}.
\end{proof}

We combine these observations to connect the two cones at their endpoints.

\begin{proposition}\label{prop:joining-conditions}
    For every $n \geq 5$, the cones $C_1(b)$ and $C_2(b)$ satisfy $C_1(b_1(n)) = C_2(b_2(n))$.
\end{proposition}
\begin{proof}
Applying the arguments within this section, we are reduced to showing that $\tilde{Z}_{\lambda}(S) \geq 0$ whenever $T = \ell_{a_1, b_1}^{-1} \ell_{a_2, b_2}(S) \in \cE_1(b_1)$.
Using the relations~\eqref{eqn:Ric(S)-expression} and~\eqref{eqn:A-lambda(T)-brendle-inequality}, we compute $\tilde{Z}_{\lambda}(S)$ as
\allowdisplaybreaks{
\begin{align*}
    \tilde{Z}_{\lambda}(S) &= Z_{\lambda}(T) + p_2 \, [ (1+\lambda^2) ( \textup{Ric}(T)_{11} + \textup{Ric}(T)_{22}) + 2 \, \textup{Ric}(T)_{33} + 2 \lambda^2 \, \textup{Ric}(T)_{44}] \\
    & \quad + \tfrac{4}{n} (q_2 - p_2)(1+\lambda^2) \, \textup{scal}(T) + \textstyle{\sqrt{2b_2}} (1-\lambda^2) ( 1+(n-2) p_2) \, ( \textup{Ric}(T)_{11} + \textup{Ric}(T)_{22}) \\
    & \quad + 2 (1-\lambda^2) \bigl[ \tfrac{1}{n}\textstyle{\sqrt{2b_2}} \, ( 2(n-1) q_2 - (n-2) p_2) + B_n b_2 \textstyle{\sqrt{2b_2}} (1+2(n-1) q_2) \bigr] \, \textup{scal}(T) \\
    &\geq \tfrac{1}{n-2} (1-\lambda^2) \, [ 2 (T_{1212} + T_{1313} + T_{2323}) - \textup{Ric}(T)_{11} - \textup{Ric}(T)_{22}] + 2 p_2 \lambda^2 \sum_{i=1}^4 \textup{Ric}(T)_{ii} \\
    & \quad + (1-\lambda^2) \, \Bigl[ (u_2+p_2) \, ( \textup{Ric}(T)_{11} + \textup{Ric}(T)_{22}) + 2 p_2 \textup{Ric}(T)_{33} \Bigr] + \tfrac{4}{n}(q_2-p_2)(1+\lambda^2) \, \textup{scal}(T)  \\
    & \quad + 2 \, (1-\lambda^2) \bigl[ \tfrac{1}{n} \textstyle{\sqrt{2b_2}} ( 2(n-1) q_2 - (n-2) p_2) + r_2 (1+2(n-1) q_2) \bigr] \, \textup{scal}(T) .
\end{align*}}
In the second step, we substituted $u_2 = \textstyle{\sqrt{2b_2}}( 1 + (n-2) p_2)$ and $r_2 = B_n b_2 \sqrt{2b_2}$ from~\eqref{eqn:endpoint-parameters}.

For a fixed algebraic curvature operator $T \in \cE_1(b_1)$, the right-hand side is an affine function of $\lambda^2$, so we need only prove the non-negativity of the resulting expression for $\lambda = 0$ and $\lambda = 1$.

We first show that this expression is non-negative for $\lambda=1$.
Indeed, notice that the bound~\eqref{eqn:v-intervals} for $v := \frac{1}{\textup{scal}(T)} \textup{tr}_{ \la e_1, e_2, e_3 , e_4 \rg } \textup{Ric}(T)$ implies that
\[
\sum_{i=1}^4 \textup{Ric}(T)_{ii} \geq \frac{1}{2} \textup{scal}(T) \quad \text{for } \; n =5 , \qquad \sum_{i=1}^4 \textup{Ric}(T)_{ii} \geq 0 \quad \text{for } \; n \geq 6.
\]
Using the above expression and the bound from Lemma~\hyperref[lemma:properties-II]{\ref{lemma:properties-II}~$(iii)$}, we obtain 
\begin{align*}
    n&=5: \qquad & \tilde{Z}_1(S) &\geq [p_2 + \tfrac{8}{5} (q_2- p_2) ] \, \textup{scal}(T) = \tfrac{1}{5} ( 8 q_2 - 3p_2) \, \textup{scal}(T) > 0, \\
    n&\geq 6: \qquad  & \tilde{Z}_1(S) &\geq \tfrac{8}{n} (q_2 - p_2) \, \textup{scal}(T) > 0,
\end{align*}
since $q > \frac{3}{8} p$ for $n=5$ and $q>p$ for $n \geq 6$.
Thus, $\textup{scal}(T)^{-1} \tilde{Z}_1(S)>0$ for $n \geq 5$.

Next, we turn to the case $\lambda=0$.
Our previous lower bound for $\tilde{Z}_0(S)$ becomes
\begin{align*}
    \tilde{Z}_0(S) &\geq \tfrac{2}{n-2}(T_{1212} + T_{1313} + T_{2323}) + (u_2+p_2 - \tfrac{1}{n-2}) \, ( \textup{Ric}(T)_{11} + \textup{Ric}(T)_{22}) +  2 p_2 \textup{Ric}(T)_{33} \\
    & \quad + \tfrac{2}{n} \bigl[ 2 (q_2 - p_2) + \textstyle{\sqrt{2b_2}} \, (2(n-1) q_2 - (n-2) p_2) + n r_2( 1+ 2(n-1) q_2) ] \, \textup{scal}(T).
\end{align*}
By Proposition~\ref{prop:joining-lower-bound}, the first line is at least $- d_n \, \textup{scal}(T)$, where the dimensional constant $d_n$ is defined in~\eqref{eqn:d5-d6-dn}.
Then, the above inequality implies that
\[
\textup{scal}(T)^{-1} \tilde{Z}_0(S) \geq \Gamma_n, 
\]
where $\Gamma_n$ is the dimensional constant
\begin{align*}
\Gamma_n := - d_n + \tfrac{4}{n} (q_2 - p_2) + \tfrac{2 \sqrt{2b_2}}{n} (2(n-1) q_2 - (n-2) p_2) + 2r_2 (1 + 2(n-1) q_2).
\end{align*}
We compute this quantity using~\eqref{eqn:direct-substitution-n=5,6}, ~\eqref{eqn:direct-substitution-ngeq7}, and~\eqref{eqn:d5-d6-dn}.
First, for $5 \leq n \leq 10$, we find
\begin{align*}
    \Gamma_5 &= \tfrac{983}{25200}, \qquad \Gamma_6 = \tfrac{27}{1280}, \qquad \Gamma_7 = \tfrac{4123}{516450}, \qquad \Gamma_8 = \tfrac{2095}{149616}, \qquad \Gamma_9 = \tfrac{ 27823}{1768410}, \qquad \Gamma_{10} = \tfrac{612103}{38270400}.
\end{align*}
Then, for $n \geq 11$, the expression~\eqref{eqn:direct-substitution-ngeq7} gives $d_{11} = 0$.
Since $q_2> p_2$, we also have $2(n-1) q_2 > (n-2) p_2$, so all the summands of $\Gamma_n$ are positive.
We conclude that $\tilde{Z}_0(S) \geq \Gamma_n \,  \textup{scal}(T) > 0$ for every $n \geq 5$.

To summarize, we have proved that $\tilde{Z}_{\lambda}(S) \geq 0$ for all $\lambda \in [0,1]$, and hence $S \in \cE_2(b_2)$ and $\ell_{a_2, b_2}(S) \in C_2(b_2)$.
This shows that $C_1(b_1) = C_2(b_2)$ and completes the proof of our assertion.
\end{proof}

We now verify that the constructions of Sections~\ref{section:first-family} and~\ref{section:second-family} produce the pinching cones of Theorem~\ref{thm:technical-construction}.
\begin{proof}[Proof of Theorem~\ref{thm:technical-construction}]
Since $C_2(b_2) = C_1(b_1)$, we let $B = b_1 + b_2$ and define
\[
\hat{C}(b) := C_1(b) \quad \text{for } \; b \in (0,b_1], \qquad \hat{C}(b) := C_2(B-b) \quad \text{for } \; b \in (b_1,B].
\]
By Proposition~\ref{prop:joining-conditions}, we have $C_1(b_1) = C_2(b_2)$, so this joining produces an admissible family $\hat{C}(b)$ that depends continuously on $b \in (0,B)$.
By Propositions~\ref{prop:transversally-invariant-family-1} and~\ref{prop:trans-inv-fam-2}, the cone $\hat{C}(b)$ is transversally invariant under the Hamilton ODE for every $b \in (0,B)$.
Definition~\ref{def:first-cone-family} implies property $(i)$, and Definition~\ref{def:a=b} supplies property $(ii)$.
This completes the proof of Theorem~\ref{thm:technical-construction}.
\end{proof}

Finally, Theorem~\ref{thm:technical-construction} can be combined with the results in~\cite{brendle-PIC} to deduce Theorem~\ref{thm:brendle-pinching}.

\bibliography{ref}

@article{brendle-duke,
    author = {Brendle, S.},
    title = {{A general convergence result for the Ricci flow in higher dimensions}},
    journal = {{Duke Math. J.}},
    year = {2008},
    volume = {145},
    number = {3},
    pages = {585-601},
}

@article{brendle-PIC,
    author = {Brendle, S.},
    title = {{Ricci flow with surgery on manifolds with positive isotropic curvature}},
    journal = {Ann. of Math. (2)},
    volume = {190},
    number = {2},
    pages = {465-559},
    year = {2019},
}

@article{brendle-schoen,
    author = {Brendle, S. and Schoen, R.},
    title = {{Manifolds with $\frac{1}{4}$-pinched curvature are space forms}},
    journal = {J. Amer. Math. Soc.},
    year = {2007},
    volume = {22},
    pages = {287-307},
}

@article{bohm-wilking,
    author = {B\"ohm, C. and Wilking, B.},
    title = {{Manifolds with positive curvature operators are space forms}},
    journal = {Ann. of Math. (2)},
    year = {2008},
    volume = {167},
    number = {3},
    pages = {1079-1097},
}

@article{pic-dimension-nine,
    author = {Chen, Z.},
    title = {{Manifolds with positive isotropic curvature of dimension at least nine}},
    journal = {arXiv:2410.21078},
    year = {2024}, 
}

@article{fraser,
    author = {Fraser, A.},
    title = {{Fundamental groups of manifolds with positive isotropic curvature}},
    journal = {Ann. of Math. (2)},
    volume = {158},
    year = {2003},
    pages = {345-354},
}

@article{bamler-cabezas-rivas-wilking,
    author = {Bamler, R. and Cabezas-Rivas, E. and Wilking, B.},
    title = {{The Ricci flow under almost non-negative curvature conditions}},
    journal = {Invent. Math.},
    volume = {217},
    pages = {95-126},
    year = {2019},
    url = {https://doi.org/10.1007/s00222-019-00864-7},
}

@article{chen-tang-zhu,
    author = {Chen, B.-L. and Tang, S.-H. and Zhu, X.-P.},
    title = {{Complete classification of compact four-manifolds with positive isotropic curvature}},
    journal = {J. Diff. Geom.},
    year = {2012},
    volume = {91},
    pages = {41-80},
}

@article{hamilton-four-manifolds,
    author = {Hamilton, R.},
    title = {{Four-manifolds with positive isotropic curvature}},
    journal = {Comm. Anal. Geom.},
    year = {1997},
    volume = {5},
    pages = {1-92},
}

@article{chow-wang-PIC,
    title = {{Minimal Two-Spheres and Manifolds with Positive Isotropic Curvature}},
    author = {Chow, T.K.A. and Wang, Y.},
    journal = {arXiv:2609.06910},
    year = {2026},
}

@article{chen-shrinking,
    author = {Chen, Z.},
    title = {{On Shrinking Ricci solitons with positive isotropic curvature in higher dimensions}},
    journal = {arXiv:2509.13183},
    year = {2025},
}

@article{cho-li,
    author = {Cho, J.H. and Li, Y.},
    title = {{Ancient solutions to the Ricci flow with isotropic curvature conditions}},
    journal = {Math. Ann.},
    year = {2023},
    volume = {387},
    number = {1},
    pages = {1009-1041},
}

@article{li-ni-wang,
author = {Li, X. and Ni, L. and Wang, K.},
title = {{Four-dimensional gradient shrinking solitons with positive isotropic curvature}},
journal = {Int. Math. Res. Not.},
    year = {2018},
    number = {3},
    pages = {949-959},
}

@article{brendle-naff,
    author = {Brendle, S. and Naff, K.},
    title = {{Rotational symmetry of ancient solutions to the Ricci flow in higher dimensions}},
    journal = {Geom. Topol.},
    year = {2023},
    volume = {27},
    pages = {153-226},
}

@article{naff,
    author = {Naff, K.},
    title = {{Shrinking Ricci solitons with positive isotropic curvature}},
    journal = {arXiv:1905.10305},
    year = {2019},
}

@article{huang-classification-2023,
author = {Huang, H.},
title = {{Classification of compact manifolds with positive isotropic curvature}},
journal = {arXiv:2305.18154},
year = {2023},
}

@article{chen-zhu-jdg,
author = {Chen, B.L. and Zhu, X.P.},
title = {{Ricci Flow with Surgery on Four-Manifolds with Positive Isotropic Curvature}},
volume = {74},
number = {2},
pages = {177-264},
journal = {J. Diff. Geom.},
year = {2006},
}

@article{micallef-moore,
    author = {Micallef, M.J. and Moore, J.D.},
    title = {{Minimal two-spheres and the topology of manifolds with positive curvature on totally isotropic two-planes}},
    journal = {Ann. of Math. (2)},
    volume = {127},
    pages = {199-227},
    year = {1988},
}

@incollection{gromov,
author = {Gromov, M.},
title = {{Positive curvature, macroscopic dimension, spectral gaps and higher signatures}},
booktitle = {Functional Analysis on the Eve of the 21st Century, Vol. II},
series = {Progr. Math.},
volume = {132},
publisher = {Birkh\"auser},
address = {Boston, MA},
year = {1996},
pages = {1--213},
}

@incollection{hamilton,
author = {Hamilton, R.S.},
title = {{The formation of singularities in the Ricci flow}},
booktitle = {{Surveys in Differential Geometry, Vol. II}},
publisher = {International Press},
address = {Cambridge, MA},
year = {1995},
pages = {7--136},
doi = {10.4310/SDG.1993.v2.n1.a2},
}

@article{ivey,
author = {Ivey, T. A.},
title = {{Ricci solitons on compact three-manifolds}},
journal = {{Diff. Geom. Appl.}},
volume = {3},
number = {4},
year = {1993},
pages = {301--307},
doi = {10.1016/0926-2245(93)90008-O},
}

@article{perelman-entropy,
author = {Perelman, G.},
title = {{The entropy formula for the Ricci flow and its geometric applications}},
year = {2002},
journal = {arXiv:math$\slash$0211159},
}

@article{perelman-surgery,
author = {Perelman, G.},
title = {{Ricci flow with surgery on three-manifolds}},
year = {2003},
journal = {arXiv:math$\slash$0303109},
}

@article{brendle-daskalopoulos-naff-sesum,
    author = {Brendle, S. and Naff, K. and Daskalopoulos, P. and Sesum, N.},
    title = {{Uniqueness of compact ancient solutions to the higher-dimensional Ricci flow}},
    journal = {J. reine angew. Math.},
    year = {2023},
    volume = {795},
    pages = {85-138},
}

\end{document}